\documentclass[reqno]{amsart}
\usepackage{amsmath,amsthm,amsfonts,amssymb}
\usepackage{hyperref}
\usepackage{color}
\usepackage{graphicx}
\usepackage{psfrag}
\theoremstyle{theorem}
\newtheorem{theorem}{Theorem}[section]
\newtheorem{corollary}[theorem]{Corollary}
\newtheorem{lemma}[theorem]{Lemma}
\newtheorem{prop}[theorem]{Proposition}

\theoremstyle{definition}
\newtheorem{definition}{Definition}[section]
\newtheorem{remark}[definition]{Remark}
\newtheorem{example}[definition]{Example}

\newtheorem{conds}[definition]{Conditions}

\numberwithin{equation}{section}

\newcommand{\R}{\mathbb  R}
\newcommand{\Q}{\mathbb  Q}
\newcommand{\C}{\mathbb  C}
\newcommand{\K}{\mathbb  K}
\newcommand{\N}{\mathbb N}
\newcommand{\BA}{\mathbb{A}}
\newcommand{\Z}{\mathbb{Z}}
\newcommand{\BP}{\mathbb P}

\newcommand{\kk}{\mathfrak{k}}
\newcommand{\fs}{\mathfrak{s}}

\newcommand{\fq}{\mathfrak{q}}
\newcommand{\CP}{\mathcal{P}}
\newcommand{\CE}{\mathcal{E}}
\newcommand{\CF}{\mathcal{F}}
\newcommand{\CO}{\mathcal{O}}
\newcommand{\CM}{\mathcal{M}}

\newcommand{\CL}{\mathcal{L}}
\newcommand{\CK}{\mathcal{K}}

\newcommand{\CJ}{\mathcal{J}}
\newcommand{\AI}{A_\infty}
\newcommand{\Si}{\Sigma}

\newcommand{\I}{\sqrt{-1}}

\newcommand{\bz}{\boldsymbol{z}}
\newcommand{\bb}{\boldsymbol{b}}

\newcommand{\be}{\boldsymbol{e}}
\newcommand{\bu}{\boldsymbol{u}}
\newcommand{\bv}{\boldsymbol{v}}
\newcommand{\bm}{\boldsymbol{m}}
\newcommand{\bx}{\boldsymbol{x}}
\newcommand{\bfx}{\text{\bf x}}
\newcommand{\bfy}{\text{\bf y}}
\newcommand{\bfp}{\text{\bf p}}
\newcommand{\bD}{\mathcal{D}}

\newcommand{\bbD}{\partial \mathcal{ D}}
\newcommand{\bE}{\boldsymbol{E}}
\newcommand{\dbar}{\overline{\partial}}
\newcommand{\bC}{\boldsymbol{C}}
\newcommand{\bsg}{\boldsymbol{\Sigma}}
\newcommand{\WT}[1]{\widetilde{#1}}
\newcommand{\UL}[1]{\underline{#1}}
\newcommand{\WH}[1]{\widehat{#1}}
\newcommand{\OL}[1]{\overline{#1}}

\newcommand{\bX}{\mathcal X}

\newcommand{\CMBR}{\CM^{main,reg}_{k+1,0}(L(u),\beta)}
\newcommand{\TE}{\text{\bf e}}

\begin{document}
\title[Lagrangian Floer theory for toric orbifolds]{Holomorphic orbidiscs and Lagrangian Floer cohomology of symplectic toric orbifolds}
\author{Cheol-Hyun Cho}
\address{ Department of Mathematical Sciences, Seoul National University, Gwanak-Gu, Seoul,151-747 South Korea, Email:
chocheol@snu.ac.kr }

\author{Mainak Poddar}
\address{ Departamento de
Matem\'aticas, Universidad de los Andes, Bogota, Colombia; and
Stat-Math Unit, Indian Statistical Institute, Kolkata, India,
Email: mainakp@gmail.com}
\thanks{First author  supported by the Basic research fund 2009-0074356 funded by the Korean government,
 Second author supported by the Proyecto de Investigaciones grant of Universidad de los Andes}

\begin{abstract}
We develop Floer theory of Lagrangian torus fibers
in compact symplectic toric orbifolds. We first classify
holomorphic orbi-discs with boundary on Lagrangian torus fibers.
We show that there exists a class of basic discs  such that we
have one-to-one correspondences between  a) smooth basic discs and
facets of the moment polytope, and b) between basic orbi-discs and
twisted sectors of the toric orbifold. We show that there is a
smooth Lagrangian Floer theory of these torus fibers, which has a
bulk-deformation by fundamental classes of twisted sectors of the
toric orbifold. We show by several examples that such
bulk-deformation can be used to illustrate the very rigid
Hamiltonian geometry of orbifolds. We define its potential  and
bulk-deformed potential, and develop the notion of leading order
potential. We study leading term equations analogous to the case
of toric manifolds by Fukaya, Oh, Ohta and Ono.
\end{abstract}

\maketitle

\tableofcontents
\bigskip
\section{Introduction}
The  Floer theory of Lagrangian  torus fibers of symplectic toric manifolds
 has been studied very extensively in the last decade,
 starting from the case of $\mathbb{CP}^n$ in \cite{C} and the toric Fano
  case in \cite{CO}. These are based on the Lagrangian Floer theory
 (\cite{Fl}, \cite{Oh1},\cite{Oh2}), whose general construction was developed by Fukaya, Oh, Ohta and Ono \cite{FOOO}.
More recent works \cite{FOOO2},\cite{FOOO3}, \cite{FOOO4} have
used the (bulk) deformation theory developed in \cite{FOOO},
bringing deep understanding of the theory in toric manifolds, and
providing beautiful pictures of (homological) mirror symmetry and
symplectic dynamics.

We develop an analogous theory for compact symplectic toric
orbifolds in this paper. Namely, this paper can be regarded as an
orbifold generalization of \cite{C}, \cite{CO}, \cite{FOOO2}
\cite{FOOO3}. We will see that the main framework is very similar,
but that the characteristics of the resulting Floer theory for
toric orbifolds are somewhat different than those of toric
manifolds.

The main new ingredient is  the orbifold $J$-holomorphic disc
(called orbi-disc). These are $J$-holomorphic discs with orbifold
singularity in the interior. The study of toric manifolds has
illustrated that understanding  holomorphic discs  is a crucial
step in developing Lagrangian Floer theory. The holomorphic discs
can be used to define the potential, corresponding to the Landau-Ginzburg
superpotential for the mirror and the potential essentially
computes the Lagrangian Floer cohomology of the torus fibers.
Holomorphic discs (which are non-singular) were classified in
\cite{CO}.  We  find a classification for holomorphic orbi-discs in section \ref{sec:basic}.

One of the main observations of this paper is that the orbifold
Lagrangian Floer theory should be considered in the following way.
Let us consider a Lagrangian submanifold $L$ which lies in the
smooth locus of a symplectic orbifold $\bX$. Then, there is a
Lagrangian Floer theory of $L$ which only considers, maps from
smooth(non-orbifold) (stable) bordered Riemann surfaces. (Here
smooth means that there is no orbifold singularity, but it could
be a nodal Riemann surface.) Namely, there is a smooth Lagrangian
Floer cohomology, and smooth $\AI$-algebra of $L$, by considering
smooth $J$-holomorphic discs and strips. We remark that smooth
$J$-holomorphic discs can meet orbifold locus if it has correct
multiplicity around the orbifold point, as will be seen in the
basic discs later.

Then, the new ingredients, orbifold $J$-holomorphic strips, and
discs, enter into the theory in the form of bulk deformation of
the smooth Floer theory. Bulk deformation was introduced in
\cite{FOOO} to deform the given Lagrangian Floer theory by an
ambient cycle in the symplectic manifold. Orbifold $J$-holomorphic
strips and discs can be considered to give bulk deformations from
the fundamental cycles of twisted sectors of the symplectic
orbifold $\bX$.  In the case of manifolds, the bulk deformation
utilizes the already existing $J$-holomorphic discs in the Floer
theory, but for orbifolds, the orbifold strips and discs do not
exist in the smooth Floer theory.  We observe that the mechanism
of bulk deformation by orbifold maps captures the
very rigid Hamiltonian geometry of symplectic orbifolds.

As noted in \cite{Wo}, \cite{WW}, \cite{ABM}, the dynamics of Hamiltonian vector fields in symplectic orbifolds
are quite restrictive.  This is because the induced Hamiltonian diffeomorphism should preserve
the isomorphism type of the points in the given orbifold. This phenomenon can be easily seen
in the example of teardrop orbifold which will be explained later in this introduction.
For example, in \cite{FOOO2} or \cite{FOOO3}, Fukaya, Oh, Ohta and Ono find
locations of non-displaceable Lagrangian torus fibers in toric manifolds, which turn out
to be always  codimension one or higher in the corresponding moment polytopes.
For toric orbifolds, already in the case of teardrop orbifold, we find codimension 0 locus
of non-displaceable fibers, and we will find in Proposition \ref{prop:allnon} that
if all the points in the toric divisors have orbifold singularity, then in fact all Lagrangian torus fibers
are non-displaceable. It is quite remarkable that this phenomenon can be explained as
a flexibility to choose bulk deformation coefficient in the leading order potential, which is essentially due to the fact that the orbifold discs and strips do not appear in the  smooth Lagrangian Floer theory.

We remark that the non-displaceability of torus orbits in toric orbifolds such as discussed in our examples  has been recently proved by Woodward \cite{Wo}, Wilson-Woodward \cite{WW} using gauged Floer theory, which is somewhat different from our methods. Their work is roughly based on holomorphic discs
in $\C^m$ and gauged theory for symplectic reduction. But note that the actual bulk orbi-potentials defined in this paper cannot be defined using their methods, as  orbifold discs with more than one orbifold marked point do not come from discs in $\C^m$. Also the formalism of bulk deformation developed in this
paper seems to give more intuitive understanding of these non-displaceabililty results in orbifolds, which
should generalize to a non-toric setting.

Beyond the symplectic dynamics of the toric orbifolds, the development of this theory can be meaningful in the following aspects.
 First, it provides  basic ingredients to study (homological) mirror symmetry(\cite{Ko})
of toric orbifolds. In \cite{FOOO4}, mirror symmetry of toric manifolds has been proved
using Lagrangian Floer theory of toric manifolds. It is easy to see that the similar formalism
may be used to explain mirror symmetry of toric orbifolds, which we leave for future research.

Second, the study of orbifold $J$-holomorphic discs provides a new
approach to study the crepant resolution conjecture, which relates
the invariants of certain orbifolds and its crepant resolutions.
In a joint work of the first author, with K. Chan, S.C. Lau, H.H.
Tseng, we will formulate an open version of the crepant resolution
conjecture for toric orbifolds, and we find a geometric
explanation for the change of variable in K\"ahler moduli spaces of
a toric orbifold and its crepant resolution. Also, this provides a
natural explanation of specialization to the root of unity in the
crepant resolution conjecture, in terms of associated potential
functions.

Now, we explain the basic setting and results of the paper in more
detail. Compact symplectic toric orbifolds  have one-to-one
correspondence with labeled polytopes $(P,\vec{c})$, as explained
by Lerman and Tolman \cite{LT}. Here $P$ is a simple rational
polytope equipped with  positive integer labels $\vec{c}$ for each
facet. Also the underlying complex orbifold may be obtained from
the stacky fan of Borisov, Chen and Smith \cite{BCS}. Stacky fan
is a simplicial fan in a finitely generated $\Z$-module $N$ with a
choice of lattice vectors in one dimensional cones. A stacky fan
corresponds to a toric orbifold when the module $N$ is freely
generated. The moment map $\mu_T$ exists for the Hamiltonian torus
action on a symplectic toric orbifold, and each Lagrangian $T^n$
orbit  is given by $L(u) = \mu_T^{-1}(u)$ for an interior point $u
\in P$.

 We recall that orbifolds are locally given as quotients of Euclidean space with a finite group
 action, and the study of Gromov-Witten theory has been extended to the case of orbifolds  in the last decade,
   starting from the work of Chen and Ruan in \cite{CR}.  In particular, they have
 introduced $J$-holomorphic maps from orbi-curves to an almost complex orbifold and
 have shown that a moduli space of such $J$-holomorphic maps of a fixed type
has a Kuranishi structure and a virtual fundamental cycle, hence
can be used to define Gromov-Witten invariants.

To find holomorphic orbi-discs with boundary on $L(u)$, we first
define what we call
 the desingularized Maslov index for J-holomorphic orbi-discs.
 This is done
 using the desingularization of the pull-back orbi-bundle introduced in \cite{CR}.
  The standard Maslov index cannot be defined here since
 the pull-back tangent bundle is not a vector bundle but an orbi-bundle.
 (See \cite{CS} for related results).
We then establish a desingularized Maslov index formula in terms of
 intersection numbers with toric divisors (analogous to \cite{C}, \cite{CO}).

There is a class of holomorphic (orbi)-discs, which play the role
of Maslov index two discs in the smooth cases. We call them basic
discs, and they are either smooth holomorphic discs of Maslov
index two, or holomorphic orbi-discs with one orbifold marked
point, of desingularized Maslov index zero.  These basic discs are
relevant for the computation of Floer cohomology of Lagrangian
torus fibers. We find that there exist holomorphic orbi-discs
corresponding to each non-trivial twisted sector of the toric
orbifold, which are basic. In addition, we find the area formula
of the basic orbi-discs and prove their Fredholm regularity.

We can use smooth $J$-holomorphic discs to set up smooth $\AI$-algebra for a Lagrangian torus fiber $L$,
 and its smooth potential function $PO(b)$ for bounding cochains $b \in H^1(L;\Lambda_0)$
in the same way as in \cite{FOOO2}. This potential $PO(b)$ can be
used to compute smooth Lagrangian Floer cohomology for $L$, by
considering its critical points. The leading order potential of
$PO(b)$ is in fact, what is usually called Hori-Vafa
Landau-Ginzburg superpotential  of the mirror \cite{HV}.

Now, as explained above, we can use orbifold $J$-holomorphic discs and strips to set up
bulk deformation of the above smooth Lagrangian Floer theory, following \cite{FOOO3}.
 The bulk deformed $\AI$-algebra
gives rise to a bulk deformed potential $PO^{\frak b}(b)$, which is a bulk deformation of
the potential $PO(b)$ above. The leading order potential of $PO^{\frak b}(b)$, which
we denote by $PO^{\frak b}_{orb, 0}(b)$ can be explicitly written down from the
 classification results on basic (orbi)-discs.

Note that full bulk-deformed potential $PO^{\frak b}(b)$ is difficult to
compute, but the leading order potential $PO^{\frak b}_{orb,
0}(b)$ given in \eqref{bv} can be used to determine non-displaceable Lagrangian torus
fibers, by studying the corresponding leading term equation of
$PO^{\frak b}_{orb, 0}(b)$ as in \cite{FOOO3}.

More precisely, consider the bulk deformation term
$\frak b = \frak b_{sm} + \frak b_{orb}$
given by
\begin{equation}\label{bv}
\begin{cases} \frak b_{sm} = \sum_{i=1}^m \frak b_{i} D_i  & \frak b_i \in \Lambda_+ \\
 \frak b_{orb} = \sum_{\nu \in Box'} \frak b_{\nu} 1_{\bX_\nu}  & \frak b_\nu \in \Lambda_+.
 \end{cases}
 \end{equation}
Here, $D_i$'s are toric divisors, and $1_{\bX_\nu}$ are
fundamental classes of twisted sectors.

Leading order potential  $PO^{\frak b}_{orb,0}$ is explicitly defined as
 \begin{equation}
 PO^{\frak b}_{orb,0} = z_1+ \cdots + z_m + \sum_{\nu \in Box'} \frak b_{\nu} z^\nu
 \end{equation}
 Here  $\nu = \sum_{i=1}^m c_i \bb_i \in N$ so that $z^\nu$ is well-defined
 Laurent polynomial of $y_1,\cdots, y_n$, $y^{-1},\cdots, y_n^{-1}$.

 It is important to note that the  leading order potential $PO(b)_0$, in the case of
 toric manifolds,  is independent of bulk parameter $\frak b$, but in our case,
 $PO^{\frak b}_{orb,0}$ {\em depends }on the choice of $\frak b_{\nu}$. In particular, this term
 provides a freedom to choose appropriate values, and different choice of $\frak b_{\nu}$ will change the leading term equation.

 The construction of $\AI$-algebra, and its bulk-deformation, and those of $\AI$-bimodules for
 a pair of Lagrangian submanifolds, and the related isomorphism between the Floer cohomology of bimodule
  and of $\AI$-algebra are almost the same as that of \cite{FOOO2}, and
 \cite{FOOO3}.  To keep the size of the paper reasonable, we only explain how to adapt their constructions in our
 cases.

 % So, rather unfortunately, we have to omit much of these
%constructions and leave readers to the references above for more
%details.

To illustrate the results of this paper, we explain the
conclusions of the paper in the case of a teardrop orbifold
$P(1,3)$. The teardrop example is explained in more detail in
section \ref{sec:tear}.

\begin{figure}[h]
\begin{center}
\includegraphics[height=2 in]{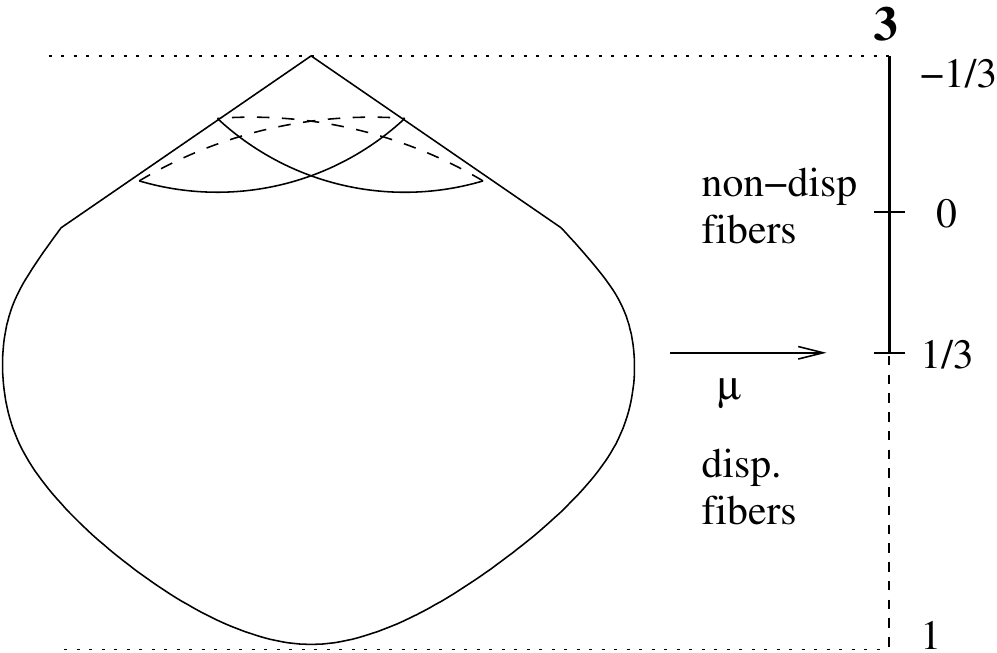}
\caption{Teardrop}
\label{tearfig}
\end{center}
\end{figure}

The teardrop orbifold has an orbifold singularity at the north
pole $N$, with isotropy group $\Z/3$. The moment map $\mu_T$ has
an image given by an interval $[-1/3, 1]$, and we put integer label $3$ at the vertex $(-1/3)$.
The inverse image
$\mu_T^{-1}([-1/3, 0))$ defines an open neighborhood  $U_N$ of the
north pole $N$, with a uniformizing cover $\WT{U}_N \cong D^2$
with $\Z/3$-action. The inverse image of $\mu_T^{-1}((0,1])$
defines a neighborhood $U_S$ of the south pole $S$. The length for
$U_N$ is one third of that of $U_S$, since the symplectic area of
$U_N$ should be considered as that of one third of the
uniformizing cover $D^2$. Hamiltonian function $H$ is  an
invariant function on $\WT{U}_N$ near $U_N$, and hence $N$ is a
critical point of such $H$.Thus any Hamiltonian flow fixes $N$,
and a nearby circle fiber cannot be displaced from itself as
illustrated in the Figure. But the fiber $\mu_T^{-1}(u)$ for $u
>1/3$ can be displaced in the open set $P(1,3) \setminus \{N\}$.
Thus the fibers $\mu_T^{-1}(u)$ for $u \in (-1/3, 1/3]$ are
non-displaceable as shown in \cite{Wo}. This is a prototypical
example of Hamiltonian rigidity in symplectic orbifolds.

As explained in more detail in section \ref{sec:tear}, such
non-displaceability can be proved using our methods. First,  the (central)
fiber $\mu_T^{-1}(0)$ can be shown to be non-displaceable using
smooth Lagrangian Floer theory, since the two smooth holomorphic
discs (of Maslov index two)
 with boundary on $\mu_T^{-1}(0)$ has the same symplectic area and
  cancels each other in the Floer differential. This is because  that the
 smooth disc wraps around $U_N$ three times, which then has the same symplectic area as
 the smooth disc covering $U_S$.

 Then, if we introduce bulk deformation by twisted sectors, then we can show that
fibers $\mu_T^{-1}(u)$ for $u \in (-1/3, 1/3]$ are indeed
non-displaceable. Namely, instead of cancelling smooth discs
covering upper and lower hemisphere, we can cancel the smooth disc
covering $U_S$ with one of the orbi-disc of $N$. Their symplectic
areas do not match, but as the orbi-discs appear as
bulk-deformations, we can adjust the coefficient $\frak b_\nu$ to
match their areas using our freedom to choose such coefficients.
Note that this method does not work for the fibers with $u \in
(1/3, 1)$ since the areas of orbi-discs are bigger than that of
the smooth disc covering $S$, and since $\frak b_\nu$ should lie
in $\Lambda_+$.

{\bf Acknowledgement:}

First author would like to thank Kaoru Ono for the comment regarding the Proposition 15.2 and Siu-Cheong Lau for helpful discussions on
orbi-potentials, and Kwokwai Chan for helpful comments on the draft. Part of the paper is written in the first author's stay at Northwestern University as a visiting professor, and he would like to thank them for their hospitality. The second
author thanks Seoul National University for its hospitality on
many occasions. He also thanks Bernardo Uribe and Andres Angel for
discussions on related topics.

{\bf Notation:}

Throughout the paper, $\bX$ is an orbifold and $X$ is the
underlying topological space. We denote by $I\bX$ the inertia
orbifold, $T=\{0 \} \cup T'$,  the index set of inertia
components. We denote by $\iota_\nu$ the rational number called
{\em age} or {\em degree shifting number} associated to each
connected component $\bX_\nu$. For toric orbifolds, we will
identify $T$ and denote it as $Box$ in definition \ref{def:box}.

The lattice $N \cong \Z^n$ parametrizes the one parameter
subgroups of the group $(\C^{\ast})^n$. Let $M$ be its dual
lattice. $\Sigma$ is a rational simplicial polyhedral fan in
$N_\R$, and $P \subset M_R$ is a rational convex polytope.

The minimal lattice vectors perpendicular to the facets of $P$,
pointing toward the interior are denoted by $\bv_1,\cdots,\bv_m$.
Certain integral multiples $\bb_j = c_j \bv_j$ will be called
stacky vectors.

 For $u \in M_\R$, let \begin{equation}\label{eq:lj} \ell_j(u) =
\langle u,\bb_j \rangle -\lambda_j.
\end{equation}
Then the moment polytope $P$ and its boundary are given by
$$P = \{ u \in M_\R \mid \ell_i(u) \geq 0, i=1,\cdots,m\}, \;\; \partial_iP=\{u\in M_\R \mid \ell_i(u)=0\}.$$
Here $\partial_iP$ is the $i$-th facet of the polytope $P$.

Let $\mu_T :X \to P$ be the moment map.  Consider $u \in Int(P)$
and denote $L(u) = \mu_T^{-1}(u)$. We may write $L$ instead of
$L(u)$ to simplify notations.

We will consider the coefficient ring $R$ to be $\R$ (as we work
in de Rham model of $\AI$-algebra) or $\C$ (when finding the
critical point of the potential). To emphasize  the choice of
coefficient ring $R$ in the Novikov ring below, we may write
$\Lambda^R, \Lambda^R_0$ instead of $\Lambda, \Lambda_0$.

 Universal Novikov ring $\Lambda$ and $\Lambda_0$ is defined as
\begin{eqnarray}\label{de:nov1}
\Lambda
& =& \{ \sum_{i=1}^{\infty} a_i T^{\lambda_i}
\mid  a_i \in R,\lambda_i \in \R,
\lim_{i\to\infty} \lambda_i = \infty \}, \\
\Lambda_{0}
&=& \{ \sum_{i=1}^{\infty} a_i T^{\lambda_i} \in \Lambda
\mid \lambda_i \in \R_{\geq 0}\}.\\
  \Lambda_+ &=&
\{ \sum_{i=1}^{\infty} a_i T^{\lambda_i} \in \Lambda \mid  \lambda_i >0 \}.
\end{eqnarray}

In \cite{FOOO}, the universal Novikov ring $\Lambda_{0,nov}$ is defined as
\begin{equation}\label{de:nov4}
\Lambda_{0,nov}
= \left.\left\{ \sum_{i=1}^{\infty} a_i T^{\lambda_i}e^{n_i}
\,\,\right\vert\,\, a_i \in R, n_i \in \Z, \lambda_i \in \R_{\geq 0},
\lim_{i\to\infty} \lambda_i = \infty\right\}.
\end{equation}
This is a graded ring by defining $\deg T = 0$, $\deg e = 2$. and
$\Lambda_{nov}$ and $\Lambda_{0,nov}^+$ can be similarly defined.
By forgetting $e$ from $\Lambda_{0,nov}$ and working with
$\Lambda_0$, we can only work in $\Z_2$ graded complex.

We define a valuation $\frak v_T$ on $\Lambda$ by
$$\frak v_T(\sum_{i=0}^\infty a_i T^{\lambda_i}) = \textrm{inf}\{\lambda_i \mid a_i \neq 0\}.$$

It is unfortunate but due to the convention, three b's, written as
$b, \frak b, \bb$ will be used throughout the paper, each of which
has a different meaning. Here $\bb_j$ is the stacky normal vector
to $j$-th facet of the polytope, $b=\sum x_i e_i$ denotes a
bounding cochain in $H^1(L,\Lambda_0)$, and $\frak b$ denotes the
bulk bounding cochain.

\section{$J$-holomorphic discs and moduli space of bordered stable maps}
In this section, we discuss moduli spaces of isomorphism classes
of stable maps from a genus $0$ prestable bordered Riemann
surfaces with Lagrangian boundary condition together with interior
orbifold marked points of a fixed type.
Let $(\bX, \omega, J)$ be
a symplectic orbifold with a compatible almost complex structure.
Let $L \subset \bX$ be a Lagrangian submanifold (in the smooth part of $\bX$).

An {\em orbifold Riemann surface} $\bsg$ is a Riemann surface $\Sigma$ ( with complex structure $j$)
together with the orbifold points $z_1,\cdots, z_k \in \Si$ such that  each orbifold  point $z_i \in \Sigma$ has
 a disc neighborhood $U$ of $z_i$ which is uniformized by a branched covering map $br :z \to z^{m_i}$.
 We  set $m=1$ for smooth points $z \in \Si$.
If $\Si$ has a non-trivial boundary,  we always assume that $\partial \Si$ is smooth, and that orbifold points
 lie in the interior of $\Sigma$ and such $\bsg$ will be called {\em bordered orbifold Riemann surface}.
  Hence $\bsg$ can be written as $(\Sigma, \vec{z},\vec{m})$ for short.

\begin{definition}\label{def:holo}
Let $\bsg$ be an (bordered) orbifold Riemann surface.

 A continuous map $f: \Si \to \bX$ is called pseudo-holomorphic
if for any $z_0 \in \Si$, the following holds:
\begin{enumerate}
\item There is a disc neighborhood of $z_0$ with a branched covering $br : z \to z^m$.
\item There is a local chart $(V_{f(z_0)}, G_{f(z_0)} , \pi_{f(z_0)})$ of $\bX$
at $f(z_0)$ and a local lifting $\widetilde{f}_{z_0}$ of $f$ in the sense that $f \circ br =
 \pi_{f(z_0)} \circ \widetilde{f}_{z_0} $.
\item $\widetilde{f}_{z_0}$ is pseudo-holomorphic.
\item If $\partial \Sigma \neq 0$, the map $f$  satisfies the
 boundary condition $f(\partial X ) \subset L$.
\end{enumerate}
\end{definition}

We need a few technical lemmas following \cite{CR} regarding orbifold maps, and we refer
readers to the Appendix or \cite{CR} for more details.

\begin{definition}\label{dregular}
A $C^{\infty}$ map ${\bf f}: \bsg \to \bX$ is called regular if $f^{-1} (\bX_{reg})$
is an open dense and connected subset of $X$.
\end{definition}

\begin{lemma}\label{unik}
 If $\bsg$ is an bordered orbifold Riemann surface and ${\bf f}: \bsg \to \bX$
   is regular and pseudo-holomorphic with Lagrangian boundary condition,  then it
  is the unique germ of $C^{\infty}$ liftings of $f$. Moreover ${\bf f}$ is
 good with a unique isomorphism class of compatible systems.
\end{lemma}

 Lemma \ref{unik} may be proved using the main idea of Lemma 4.4.11 in \cite{CR}
 together with the result on the local behavior of a pseudo-holomorphic map
from a Riemann surface near a singularity in the image, given in
 Lemma 2.1.4 of \cite{CR}. This latter result
yields unique continuity of a local lift of a pseudo-holomorphic map near a singularity
in the target.

\begin{lemma}\label{clcomp} Suppose $f: \bsg \to {\bX}$ is a pseudo-holomorphic map with
 interior marked points ${\vec z^+} =(z_1^+, \ldots, z_k^+)$,
such that $f(\partial \Sigma)$ does not intersect the singular set
of ${\bX}$. Then there exist a finite number of orbifold
structures on $\Sigma$ with singular set contained in ${\vec
z^+}$, for which there are good $C^{\infty}$ maps covering $f$.
Moreover for each such orbifold structure there exist a finite
number of pairs $({\bf f}, \xi)$ where ${\bf f}$ is a good map
lifting $f$ and $\xi$ is an associated isomorphism class of
compatible systems. The number of such pairs is bounded above by a
constant that depends on ${\bX}$, the genus of $\Sigma$ and $k$
only.
\end{lemma}
The proof of the above lemma is very similar to Chen-Ruan's proof in the case without boundary, see Proposition
2.2.1 in \cite{CR}. Simply note that the homomorphisms $\theta_{\xi_0, \xi_1}$ and $\theta_{\xi}$ of \cite{CR}
are well defined in our case by an application of Lemma \ref{csvb}.

 The construction of the moduli space is  a combination of the
construction of Fukaya, Oh, Ohta and Ono \cite{FOOO} regarding Lagrangian boundary condition, and  that of
Chen-Ruan \cite{CR}  regarding the interior orbifold singularities.

We recall the definition of genus 0 prestable bordered Riemann surfaces from \cite{FOOO}.
\begin{definition}
$\Sigma$ is called a {\em genus 0 prestable bordered Riemann
surface} if $\Sigma$ is a possibly singular Riemann surface with
boundary $\partial \Sigma$ such that the double $\Sigma
\cup_{\partial \Sigma} \overline{\Sigma}$ is a connected and
simply connected compact singular Riemann surface whose
singularities are only nodes. $(\Sigma, \vec{z},\vec{z}^+)$ is
called a {\em genus 0 prestable bordered Riemann surface with
$(k,l)$ marked points} if $\Sigma$ is a genus 0 prestable bordered
Riemann surface and $\vec{z} = (z_0,\cdots,z_{k-1})$ are boundary
marked points on $\partial \Sigma$ away from the nodes, and
$\vec{z}^+=(z_1^+,\cdots,z_l^+)$ are interior marked points in
$\Sigma \setminus \partial \Sigma$.
\end{definition}
A genus $0$ prestable bordered Riemann surface $(\Sigma,
\vec{z},\vec{z}^+)$ is
 said to be stable if each sphere component has three special(nodal or marked) points
and each disc component $\Sigma_\nu$ satisfies $2l_{\nu} +k_\nu \geq 3$ where $l_\nu$ is the number of
interior special points and $k_\nu$ is the number of boundary special points.
We denote by $\CM_{k,l}$ the space of isomorphism classes of genus 0 stable bordered
 Riemann surfaces with $(k,l)$ marked points. From the cyclic ordering of the boundary
  marked points, $\CM_{k,l}$ has $(k-1)!$ connected components.
  The main component $\CM_{k,l}^{main}$ is defined by considering a subset of curves in $\CM_{k,l}$
  whose the boundary marked points are ordered in a cyclic counterclockwise way
   (for some parametrization of $\partial \Sigma$ as $S^1$ for stable case).

We give the definition of genus $0$ prestable bordered orbi-curve
following \cite{CR} and \cite{FOOO}.
\begin{definition}
$(\bsg, \vec{z},\vec{z}^+)$ is called a {\em genus $0$ prestable
bordered orbi-curve (with interior singularity)} with   $(k,l)$ marked points if $(\Sigma,
\vec{z},\vec{z}^+)$ is genus $0$ prestable bordered Riemann
surface with $(k,l)$ marked points with the following properties:
\begin{enumerate}
\item Orbifold points are contained in the set of interior marked
points and interior nodal points.
 \item A disc neighborhood of an
interior orbifold marked point $z_i^+$ is uniformized by a
branched covering map $z \mapsto z^{m_i}$.
 \item A neighborhood of
interior nodal point(which is away from $\partial \Sigma$) is
uniformized by $(X(0,r_j),\Z_{n_j})$.
\end{enumerate}
\end{definition}
Recall from \cite{CR} that the local model of the interior orbifold nodal point, $(X(0,r_j),\Z_{n_j})$, is defined as follows:
 For any real number $t \geq 0, r>0$,
set $X(t,r) = \{(x,y) \in \C^2|\; ||x||,||y||<r, xy=t\}$. Fix an
action of $\Z_m$ on $X(t,r)$ for any $m>1$ by $e^{2\pi i/m} \cdot
(x,y)=(e^{2\pi i/m}x,e^{-2\pi i/m}y)$. The branched covering map
$X(t,r) \to X(t^m,r^m)$ given by $(x,y) \to (x^m,y^m)$ is
$\Z_m$-invariant. So $(X(t,r),\Z_m)$ can be regarded as a
uniformizing system of $X(t^m,r^m)$. Here $m_i, n_j$ are allowed
to take the value one, in which case the corresponding orbifold
structure is trivial.
  Hence, a data of genus 0 prestable bordered orbi-curve  includes
the numbers $m_i, n_j$ but we do not write them for simplicity. A notion of isomorphism and
the group of automorphisms of  genus 0 prestable bordered orbi-curves with interior singularity
is defined in a standard way, and omitted.

Now we define orbifold stable map to be used in this paper. We write $\Sigma = \bigcup_{\nu} \Sigma_{\nu}$ for
 each irreducible component $\Sigma_{\nu}$.
\begin{definition}
A genus 0 stable map from a bordered orbi-curve with $(k,l)$ marked points is a pair $\big((\bsg,\vec{z},\vec{z}^+),w,\xi \big)$
satisfying the following properties:
\begin{enumerate}
\item $(\bsg,\vec{z},\vec{z}^+)$ be a genus 0 prestable  bordered orbi-curve
\item $w:(\Sigma,\partial \Sigma) \to (\bX,L)$ is a pseudo-holomorphic map (see Definition \ref{def:holo}). (Here, we say that $w$ is
 pseudo-holomorphic (resp. good) if each $w_{\nu}$ is pseudo-holomorphic (resp. good) and
induce  a continuous map $w: \Sigma \to X $.) \item $w$ is a
$C^\infty$ good map with an isomorphism class $\xi$ of compatible
systems.
 \item $w$ is representable
i.e. it is injective on local groups.
\item The set of all
$\phi:\Sigma \to \Sigma$ satisfying the following properties is
finite.
\begin{enumerate}
\item $\phi$ is biholomorphic.
\item $\phi(z_i)=z_i, \phi(z_i^+) = z_i^+$
\item $w \circ \phi =w$
\end{enumerate}

\end{enumerate}

 \end{definition}

\begin{definition}\label{def:eq}
 Two  stable maps $\big( (\bsg_1, \vec{z}_1,\vec{z}_1^+),w_1 ,\xi_1 \big)$
 and  $\big( (\bsg_2, \vec{z}_2,\vec{z}_2^+), w_2 ,\xi_2 \big)$ are equivalent if there exists
an isomorphism $h: (\bsg_1, \vec{z}_1,\vec{z}_1^+) \to (\bsg_2, \vec{z}_2,\vec{z}_2^+)$ such
that $w_2 \circ h = w_1$ and $ \xi_1 = \xi_2 \circ h$ , i.e.
the isomorphism class $\xi_2$ pulls back to the class $\xi_1$ via $h$.
\end{definition}

\begin{definition}
An automorphism of a stable map $\big( (\bsg, \vec{z},\vec{z}^+),w,\xi \big)$ is a self
equivalence. The automorphism group is denoted by $Aut \big( (\bsg, \vec{z},\vec{z}^+),w,\xi \big)$.
\end{definition}

Given a stable map $\big( (\bsg, \vec{z},\vec{z}^+),w,\xi \big)$,
we associate a homology class $w_*([\Sigma]) \in H_2(X,L)$. Note
that for each interior marked point $z^+_j$ (on $\Sigma_\nu$),
$\xi_\nu$ determines by the group homomorphism at $z^+_j$,
 a conjugacy class $(g_j)$, where $g_j \in G_{w(z_j)}$.

Let $I\bX$ be the inertia orbifold of $\bX$.  Denote by $T=\{0\}
\cup T'$ the  index set of inertia components, and for $(g) \in
T$,  call the corresponding component $\bX_{(g)}$. Here
$\bX_{(0)}$ is  $\bX$ itself, and elements of $x \in \bX_{(g)}$
are written as $(x, g)$.

We thus have a map $ev_i^+$ sending each
(equivalence class of) stable map into $I\bX$ by
$$ev_i^+:\big((\bsg, \vec{z},\vec{z}^+),w,\xi \big) \to  (w(z^+_i),g_i).$$

Denote by $\UL{l} = \{1,\cdots, l\}$ and consider the map
 $\bx : \UL{l} \to T$ describing the inertia component for each (orbifold) marked point.
A stable map $\big( (\bsg, \vec{z},\vec{z}^+),w,\xi \big)$ is said to be of type $\bx$ if
for $i=1,\cdots,l$,
$$ev^+_i\big( (\bsg, \vec{z},\vec{z}^+),w,\xi \big) \in \bX_{\bx(i)}.$$

\begin{definition}\label{def:orbimark}
Given a homology class $\beta \in H_2(X,L)$, we denote by $\CM_{k,l}(L,J,\beta,\bx)$
the moduli space of isomorphism classes of  genus 0 stable maps to $\bX$ from a
bordered orbi-curve with $(k,l)$ marked points of type $\bx$ and with $w_*([\Sigma])= \beta$.
We denote by $\CM_{k,l}^{main}(L,J,\beta,\bx)$ the sub-moduli space with
$(\bsg, \vec{z},\vec{z}^+)\in \CM_{k,l}^{main}$.
\end{definition}
\begin{remark}
We follow the notations of \cite{FOOO} and denote by $\CM$ the
compactified moduli space, and by $\CM^{reg}$ the moduli space
before compactification.
\end{remark}
We can give a topology on the moduli space
$\CM_{k,l}^{main}(L,J,\beta,\bx)$ in a way similar to
\cite{FOn},\cite{FOOO} and \cite{CR2} (definition 2.3.7). As it is
standard, we omit the details. Following Proposition 2.3.8 and
Lemma 2.3.9 of \cite{CR}, we have
\begin{lemma}
The moduli space $\CM_{k,l}^{main}(L,J,\beta,\bx)$ is compact and metrizable.

The symplectic area of elements in $\CM_{k,l}^{main}(L,J,\beta,\bx)$ only
depends on the homology class $\beta$ and the symplectic form $\omega$.
\end{lemma}

The orientation issues can be dealt exactly in the same way as in
\cite{FOOO}. Theorem 2.1.30  together with Proposition 3.3.5 of
\cite{CR} show that the Kuranishi structure for orbifold case is
stably complex.
\begin{theorem}
Let $L$ be relatively spin. Then a choice of relative spin structure of $L \subset X$ canonically
induces an orientation of  $\CM_{k,l}^{main}(L,J,\beta,\bx)$.
\end{theorem}

We will consider the moduli space $\CM_{k,l}^{main}(L,\beta,\bx)$ (with $J = J_0$ the standard complex structure
of the toric orbifolds) in more detail later, but the virtual dimension of the moduli space is
given as follows.
\begin{lemma}\label{lem:dim}
The virtual dimension of the moduli space $\CM_{k,l}^{main}(L,\beta,\bx)$ is
$$n+ \mu^{de}(\beta,\bx) + k + 2l -3 \; = \; n + \mu_{CW}(\beta) + k + 2l -3 -2\iota(\bx).$$
\end{lemma}

In the next section, we will explain  $\mu^{de}$ which is the
desingularized Maslov index of $(\beta,\bx)$, and
$\mu_{CW}(\beta)$ which is the Chern-Weil Maslov index of
\cite{CS}. Let
 $\iota(\bx) =\sum_i \iota_{(g_i)}$ for $\bx=(\bX_{(g_1)},\cdots,\bX_{(g_l)})$
 where
$\iota_{(g)}$ is the degree shifting number defined by Chen-Ruan
\cite{CR2}. We remark that the desingularized Maslov index depends
on $\bx$ as we need to desingularize the pull-back tangent
orbi-bundle, which depends on $\bx$.

\section{Desingularized Maslov index}\label{maslovorbicompute}
Maslov index is related to the (virtual) dimension of moduli spaces  in Lagrangian Floer theory (Lemma \ref{lem:dim}).
 For orbifolds, the standard definition of Maslov index does not have natural extension, since the pull-back tangent
 bundle under a good map is usually an orbi-bundle which is not a trivial bundle over the bordered orbi-curve.

In this section, we define, what we call, desingularized Maslov
index, and provide computations of several examples of holomorphic
orbi-discs, which will appear in later sections.  On the other
hand, recently, the first author and H.-S. Shin \cite{CS} gave
Chern-Weil definition of Maslov index, which is given by curvature
integral of an orthogonal connection. This Chern-Weil definition
naturally extends to the orbifold setting, and the relation
between Chern-Weil and desingularized Maslov indices has been
discussed in \cite{CS}. We give a brief explanation at the end of
this section.

\subsection{Definition of desingularized Maslov index}
 Chen and Ruan \cite{CR} have shown that for orbifold holomorphic
map $u:\bsg \to \bX$ from a closed orbi-curve without boundary to
 an orbifold, the Chern number $c_1(T\bX)([\bsg])$ (defined via Chern-Weil theory) is in general a rational number
  and  by suitable subtraction  of degree shifting number for each orbifold point,
  one obtains the Chern number of a desingularized bundle which
  is an honest bundle. Hence the corresponding number is an
  integer. It is related to the  Fredholm index for the moduli spaces.

%We remark that for orbifold curves without boundary,
%Chen and Ruan \cite{CR} have shown that for orbifold holomorphic map $u:\bsg \to \bX$ from a closed orbi-curve to
% an orbifold, the Chern number $c_1(T\bX)([\bsg])$ (defined via Chern-Weil theory) is in general a rational number
%  and  by suitable subtraction  of degree shifting number for each orbifold point, one obtains the Chern number of a desingularized bundle which
%  is an honest bundle. Hence the corresponding number is an integer,
%  which is related to the  Fredholm index for the moduli spaces.

  The similar phenomenon  happens for orbi-discs (discs with interior orbifold singularities).
 We will mainly work with a Maslov index of a desingularized orbi-bundle
 and such an index will be called desingularized Maslov index for short, and this will be an integer.

Let us first recall the standard definition of Maslov index for a
smooth disk with Lagrangian boundary condition. If $w:(D^2,
\partial D^2) \to (X,L)$ is a smooth map of pairs, we can find a
unique symplectic trivialization (up to homotopy) of the pull-back
bundle $w^*TX \cong D^2 \times \C^n$. This trivialization defines
a map from $S^1 = \partial D^2$ to $(U(n)/O(n))$, the set of
Lagrangian planes in $\C^n$, and the Maslov index is
 a rotation number of this map composed with
 the map $det^2:(U(n)/O(n)) \to U(1)$ (see \cite{Ar}). For bordered Riemann surfaces
with several boundary components, one can define its Maslov index
similarly by taking the sum of Maslov indices along $\partial \Sigma$ using the fact that a symplectic vector bundle
 over a Riemann surface with boundary $\Sigma$ is always trivial.
The data of symplectic vector bundle over $\Sigma$, and Lagrangian subbundle over $\partial \Sigma$ is called a
 bundle pair, and one can define Maslov index for any bundle pair in the same way.

Next, we recall the desingularization of orbi-bundle on an orbifold Riemann surface by Chen and Ruan
(\cite{CR}) which plays a key role.

Consider   a closed (complex) Riemann surface $\Sigma$, 
  with  distinct points
$\vec{z}=(z_1,\cdots,z_k)$ paired with multiplicity $\vec{m}=(m_1,\cdots,m_k)$. 
We consider the orbifold structure at $z_i$ which is
given by the ramified covering $z_i \to z_i^{m_i}$. 
For simplicity we denote it as $\bsg = (\Sigma,\vec{z},\vec{m})$, which is a 
 closed, reduced, 2-dimensional orbifold.
 
Consider  a complex orbi-bundle $E$ be
over $\bsg$ of rank $n$. Then, at each singular point $z_i$, $E$ gives a representation
$\rho_i:\Z_{m_i} \to Aut(\C^n)$ so that over a disc neighborhood $D_i$ of $z_i$, $E$ is uniformized by
$(D_i \times \C^n, \Z_{m_i},\pi)$ where the action of $\Z_{m_i}$ on $D_i \times \C^n$ is defined as\begin{equation}
e^{2\pi i/m_i} \cdot (z,w)=\big(e^{2\pi i/m_i} z, \rho_i(e^{2\pi i/m_i})w \big)
\end{equation}
for any $w \in \C^n$. Note that  $\rho_i$ is uniquely
determined by integers $(m_{i,1},\cdots,m_{i,n})$
with $0 \leq m_{i,j} < m_i$, as it is given by the matrix
\begin{equation}\label{degshift}
\rho_i(e^{2\pi i/m_i})= diag( e^{2\pi i m_{i,1}/m_i},\cdots, e^{2\pi i m_{i,n}/m_i} \big).
\end{equation}
The sum $\sum_{j=1}^n \frac{m_{i,j}}{m_i}$ is called the degree shifting number(\cite{CR}).

Over the punctured disc $D_i \setminus \{0\}$ at $z_i$, $E$
is given a specific trivialization from $(D_i \times
\C^n,\Z_{m_i},\pi)$ as follows: consider a $\Z_{m_i}$-equivariant
map $\Psi_i:D \setminus \{0\} \times \C^n \to D \setminus \{0\}
\times \C^n$ defined by
\begin{equation}
(z,w_1,w_2,\cdots,w_n) \to (z^{m_i}, z^{-m_{i,1}}w_1,\cdots,z^{-m_{i,n}}w_n),
\end{equation}
where $\Z_{m_i}$ action on the target $D\setminus \{0\} \times \C^n$ is trivial. Hence $\Psi_i$ induces a trivialization $\Psi_i:E_{D_i \setminus \{0\}} \to D_i \setminus \{0\} \times \C^n$.
We may extend the smooth complex vector bundle $E_{\bsg \setminus \{z_1,\cdots,z_k\}}$ over $\bsg \setminus  \{z_1,\cdots,z_k\}$ to
a smooth complex vector bundle over $\bsg$ by using these trivializations $\Psi_i$ for each $i$. The resulting complex vector bundle is called the desingularization of $E$ and denoted by $|E|$.

The essential point as observed in \cite{CR} is that the sheaf of holomorphic sections of the desingularized orbi-bundle
and the orbibundle itself are the same.
\begin{prop}[\cite{CR} Proposition 4.2.2]
Let $E$ be a holomorphic orbifold bundle of rank $n$ over a compact orbicurve $(\bsg,\bz,\bm)$ of genus $g$. The $\CO(E)$ equals
$\CO(|E|)$, where
 $\CO(E)$ and  $\CO(|E|)$ are sheaves of holomorphic sections of $E$ and $|E|$.
\end{prop}
As the local group action on the fibers of the desingularized
orbi-bundle $|E|$ is trivial,
 one can think of it as a smooth vector bundle on $\Sigma$ which is analytically the same as $\bsg$
  (In other words,  there exist a canonically associated vector bundle $|E|$ over the smooth
  Riemann surface $\Sigma$). Hence, for the bundle $|E|$, the ordinary index theory can be applied, which provides the
required index theoretic tools for the orbibundle $E$.

Now we give a definition of the desingularized Maslov index, which
determines the virtual dimension
 of the moduli space of J-holomorphic orbi-discs.
\begin{definition}
Let $\bsg = (\Sigma, \vec{z},\vec{m})$ be a bordered orbi-curve with $(0,k)$ marked points.
Let $u:\bsg \to X$ be an orbifold stable map. Then, $u^*TX$ is a complex orbi-bundle over $\bsg$, with
Lagrangian subbundle $u|_{\partial \bsg}^*TL$ at $\partial \bsg$.
 Let $|u^*TX|$ be the desingularized bundle over $\bsg$( or $\Sigma$), which still have the Lagrangian subbundle
  at the boundary from $u|_{\partial \Sigma}^*TL$.
The Maslov index of the bundle pair $(|u^*TX|, u|_{\partial \Sigma}^*TL)$ over $(\Sigma,\partial \Sigma)$ is called
 the {\em desingularized Maslov index} of $u$, and denoted by
 $\mu^{de}(u)$. Note that this index is well-defined as it is
 independent of the choice of compatible system for
 $u$, within the same isomorphism class,  by Lemma \ref{csvb}.
\end{definition}

%We prove that homotopic good maps have the same desingularized
%Maslov index (see \ref{orbihtpy} for the precise notion of
%homotopy used here). For this we need to prove that the pullbacks
%of an orbifold vector bundle by homotopic good maps are
%isomorphic.

%\begin{lemma} If ${\bX} $ is compact (or paracompact) reduced orbifold, then
%${\bf f}^{\ast}(\mathcal{V} )\cong {\bf g}^{\ast}(\mathcal{V})$
%for any orbifold vector bundle $\mathcal{V} $ on
% ${\bX}^{\prime}$ and homotopic good maps ${\bf f},\, {\bf g}
% : {\bX} \to {\bX}^{\prime} $.
% \end{lemma}

%\begin{proof} Suppose $F: {\bX} \times [0,1] \to {\bX}^{\prime} $
%is the homotopy. Denote the orbifold vector bundle
%$F^{\ast}(\mathcal{V})$ on ${\bX} \times [0,1]$ by $\mathcal{E}$.
%Let $\mathcal{E}_t$ be the restriction of $\mathcal{E}$ to ${\bX}
%\times \{t\}$. It is enough to prove that  $\mathcal{E}_0 \cong
%\mathcal{E}_1$. Since ${\bX} $ is reduced, we can represent it by
%the orbifold associated to the action of a compact Lie group $G$
%on a smooth manifold $M$. The bundle $\mathcal{E}$ is represented
%by a $G$-vector bundle $E$ on $M \times [0,1]$. Then the result
%follows from Lemma 1.6.4 of \cite{At}.
% by using the trivial homotopy from $X$ to $X \times [0,1]$.
% the paracompact case follows since extension of section just
% depends on existence of partition of unity.
%\end{proof}

%\begin{corollary} Suppose $u: \bsg \to {\bX}$ is a good map.
%Then $\mu^{de}(u)$ is an invariant of the orbifold homotopy class
%of $u$.
%\end{corollary}

\subsection{Examples of computations of the index}
Here we give a few examples of computations of the desingularized Maslov indices.
Consider the orbifold disc $\bD$ with $\Z_p$ singularity at the origin,
and the orbifold complex plane $\bC$ with $\Z_p$ singularity at the origin.
 Let the unit circle $L=S^1 \in \bC$ a  Lagrangian submanifold.  Consider the natural inclusion $u:\bD \to \bC$.
\begin{lemma}
The desingularized Maslov index of  $u$ equals $0$
\end{lemma}
\begin{proof}
Consider the tangent orbibundle $T\bD$ over $\bD$, and its uniformizing chart
$D^2 \times \C = \{(z,w)| z \in D^2, w \in \C \}$ with the $\Z_p$ action given by
\begin{equation}
e^{2\pi i/p} \cdot (z,w) = \big( e^{2\pi i/p} z, e^{2\pi i/p} w \big)
\end{equation}
Then, the subbundle $TL$ at $z \in S^1$ is given by $\R \cdot iz \subset \C$.
We consider its image under the desingularization map $\Psi:D^2 \times \C \to D^2 \times \C$
defined as $\Psi(z,w) = (z^p, z^{-1}w)$. The image of $TL$ via $\Psi$
at the point $\alpha \in D^2$ with $\alpha =z^p$ is given by
$ \R \cdot z^{-1}iz = i \cdot \R  \subset \C$.

The  desingularization provides a desingularized vector bundle  over
the  orbi-disc $\bD$, which is a trivial vector bundle, and the loop of Lagrangian subspaces at the boundary
 is a constant loop. Therefore  the desingularized Maslov index is  zero.
\end{proof}

 We now consider a more general case: Consider the orbifold disc $\bD$
with $\Z_m$ singularity at the origin, and the complex plane $\bC$
with $\Z_{mn}$ singularity at the origin, and the unit circle
$L=S^1 \in \bC$ as a Lagrangian submanifold. Consider the
uniformizing cover $D^2$ of $\bD$, with coordinate $z \in D^2$.
Consider the uniformizing cover $\C$ of $\bC$, with coordinate $y
\in \C$.
\begin{lemma}
Consider the map $u:\bD \to \bC$, induced from the map $\WT{u}:D^2 \to \C$
defined by $\WT{u}(z)= z^k$. Here we assume that $k,m$ are relatively prime
to ensure that the group homomorphism is injective.
Then, the desingularized Maslov index of  $u$ equals $2[k/m]$ where $[k/m]$ is the largest integer $ \leq  k/m$.
\end{lemma}
\begin{proof}
Consider the tangent orbibundle $T\bC$ over $\bC$ whose uniformizing chart is
given by $\C \times \C = \{(y,w)| y,w \in \C \}$ with the $\Z_{mn}$ action given by
the diagonal action.
Then, the subbundle $TL$ at $y \in S^1$ is given by $\R \cdot iy \subset \C$.
We consider the pull-back orbibundle, $u^*T\bC$ whose uniformizing chart is
given by $\C \times \C = \{(z,w)| z,w \in \C \}$ with the $\Z_{m}$ action given by
\begin{equation}
e^{2\pi i/m} \cdot (z,w) = \big( e^{2\pi i/m} z, e^{2\pi k i/m} w \big)
\end{equation}
In this chart,  the subbundle $(u|_{\partial \bD})^*TL$ is given by
$(z ,\R \cdot z^k i)$ for $z \in \partial D^2$.
Now, we consider its image under the desingularization map $\Psi:D^2\setminus \{0\} \times \C \to D^2 \setminus \{0\} \times \C$
defined as $\Psi(z,w) = (z^m, z^{-k'}w)$, where $k'=k-[k/m]m$. The image of $TL$ via $\Psi$
at the point $\alpha \in D^2$ with $\alpha =z^m$ is given by
$ \R \cdot z^{-k'}i z^k = z^{[k/m]m}i \cdot \R  \subset \C$.

Hence we obtain a trivialized desingularized vector bundle $|E|$ over $\bD$( and hence $D^2$), and
from the above computation, the loop of Lagrangian subspaces along the boundary is given by
$z^{[k/m]m}i \cdot \R$. But also note that the coordinate on $D^2$ is in fact $z^m$, and
hence the desingularized  Maslov index of $u$ is $2[k/m]$.
\end{proof}
\begin{remark}
Note that in the case that $(k,m)$ are not relatively prime, say $d=gcd(k,m)$, then
instead of the map from the above orbifold disc, we consider a domain with simpler singularity,
say $\bD$ with $\Z_{m/d}$ singularity at the origin, and the map given by $x \mapsto x^{k/d}$.
The Maslov index of this orbifold holomorphic disc is still $2[k/m] = 2[(k/d)/(m/d)]$.
\end{remark}

The following computations of indices will be used later in the
paper. We compute desingularized Maslow indices for orbi-discs in
$\bX=\C^n/G$ Consider the orbifold disc $\bD$ with $\Z_m$
singularity at the origin, and the orbifold $\bX$
 defined by the complex vector space $\C^n$ with an action of a finite abelian group $G$.
Consider the uniformizing cover $D^2$ of $\bD$, with coordinate $z \in D^2$.
\begin{lemma}\label{generalorbicompute}
Consider the holomorphic orbi-disc $u:(\bD,\partial \bD) \to (\bX,L)$, induced from an equivariant
map $\WT{u}:D^2 \to \C^n$ given by
\begin{equation}\label{localorbidisc1}
(a_1z^{d_1} ,\cdots,a_kz^{d_k} ,a_{k+1} ,\cdots,a_n ),
\end{equation}
where $a_i \in U(1), d_i \geq 0$ for all $i$.
We set $d_{k+1} = \cdots = d_n =0$ and $L=(S^1)^n \in \C^n$.
Then, the desingularized Maslov index of  $u$ equals $2\sum_i[d_i/m]$.
\end{lemma}
\begin{proof}
Consider the tangent orbibundle $T\bX$ over $\bX$ whose
uniformizing chart is given by $\C^n \times \C^n =
\{(\vec{y},\vec{w})| \vec{y},\vec{w} \in \C^n \}$ with the group
$G$ acting diagonally. Then, the fiber of $TL$ at $\vec{y}$ is
given by $(\R \cdot iy_1,\cdots,\R \cdot iy_n) \in \C^n$. We
consider the pull-back orbibundle, $u^*T\bX$ whose uniformizing
chart is given by $\C^n \times \C^n = \{(z,\vec{w})| z,\vec{w} \in
\C^n \}$ with the $\Z_{m}$ action given by
\begin{equation}
e^{2\pi i/m} \cdot (z,\vec{w}) = \big( e^{2\pi i/m} z, e^{2\pi d_1
i/m} w_1,\cdots,e^{2\pi d_n i/m} w_n \big).
\end{equation}
In this chart,  the subbundle $(u|_{\partial \bD})^*TL$ is given by
$$(z ,\R \cdot a_1z^{d_1} i,\cdots,\R \cdot a_nz^{d_n} i).$$

Now, we consider its image under the desingularization map
$\Psi:D^2\setminus \{0\} \times \C^n \to D^2 \setminus \{0\}
\times \C^n$ defined by $\Psi(z,w) = (z^m,
x^{-d_1'}w_1,\cdots,x^{-d_n'}w_n)$, where $d_i'=d_i-[d_i/m]m$. We
have $d_{k+1}'=\cdots,d_n'=0$. The image of $TL$ via $\Psi$ at the
point $\alpha \in D^2$ with $\alpha =z^m$ is given by
$$ (\cdots, \prod_i \R \cdot z^{-d_i'}i z^{d_i},\cdots) = (\cdots, z^{[d_i/m]m}i \cdot \R,\cdots)  \subset \C^n. $$

Hence we obtain a trivialized desingularized vector bundle $|E|$ over $\bD$, and
the Maslov index of the loop of Lagrangian subspaces over uniformizing cover $D^2$ is  $\sum 2[d_i/m]m$ and
hence the Maslov index for the  orbi-disc $u$ is $\sum 2[d_i/m]$.
\end{proof}

\subsection{Relation to Chern-Weil Maslov index}
Now, we explain the Chern-Weil construction of Maslov index for
orbifold from \cite{CS} and its relationship with the
desingularized Maslov index defined in this section.

 By bundle pair  $(\CE,\CL)$ over $\Sigma$, we  mean a symplectic vector bundle $\CE \to \Sigma$
  equipped with compatible almost complex structure, together with Lagrangian subbundle
$\CL \to \partial \Sigma$ over the boundary of $\Sigma$.
Let $\nabla$ be a unitary connection of $E$, which is orthogonal with respect to $\CL$ :
 this  means that $\nabla$ preserves $\CL$ along the boundary $\partial \Sigma$.
 See Definition 2.3 of \cite{CS} for the precise definition.
\begin{definition}\label{def:cw}
The Maslov index of the bundle pair $(\CE,\CL)$ is defined by
$$\mu_{CW}(\CE,\CL)=\frac{ \I}{\pi}\int_\Si{tr(F_\nabla)}$$
where $F_\nabla\in\Omega^2(\Si,End(\CE))$ is the curvature induced by $\nabla$.
\end{definition}

It is proved in \cite{CS} that this Chern-Weil definition agrees
with the usual topological definition of Maslov index. But the
above definition of  Maslov index has an advantage over the
topological one in that it extends more readily to the orbifold
case, as observed in \cite{CS}. In orbifold case, $\CE$ is assumed
to be a symplectic orbibundle over orbifold Riemann surface $\bsg$
and the Maslov index is defined by considering orthogonal
connections which are, in addition,
 invariant under local group actions. Thus, the Maslov index of the
 bundle pair $(\CE, \CL)$ over orbifold Riemann surface
with boundary is defined as the curvature integral as in
Definition \ref{def:cw}. It is shown in \cite{CS} that the Maslov
index $\mu_{CW}(\CE, \CL)$ is independent of the choice of
orthogonal unitary connection $\nabla$ and also independent of the
choice of an almost complex structure.

Finally, we recall  proposition  6.10 of \cite{CS} relating Maslov
index with desingularized Maslov index:
\begin{prop}\label{maslow=de}
Suppose $\bsg$ have $k$ interior orbifold marked points of order $m_1,\cdots, m_k$,
$$\mu_{CW}(\CE,\CL) = \mu^{de}(\CE,\CL) + 2 \sum_{i=1}^k \sum_{j=1}^n \frac{m_{i,j}}{m_i},$$
where $m_{i,j}$ are defined as in \eqref{degshift}.
\end{prop}

\section{Toric orbifolds}
\label{sec:toric}

In this paper, we consider compact toric orbifolds. These are
 more general than compact simplicial toric varieties, in
that their orbifold singularities may not be fully captured by the
analytic variety structure. In fact we are mainly interested in a
subclass called symplectic toric orbifolds. These have been
studied by Lerman and Tolman \cite{LT}, and correspond to
polytopes with positive integer label on each facet. In algebraic
geometry, Borisov, Chen and Smith \cite{BCS} considered toric DM
stacks that correspond to stacky fans. The vectors of such a
stacky fan take values in a finitely generated abelian group $N$.
A toric DM stack is a toric orbifold when $N$ is free and in this
case the stabilizer of a generic point is trivial.

\subsection{Compact toric orbifolds as complex quotients}
Combinatorial data called {\em complete fan of simplicial rational
polyhedral cones}, $\Sigma$, are used to describe compact toric
manifolds (see \cite{Cox} or \cite{Au}). For the definitions of
rational simplicial polyhedral cone $\sigma$ and fan $\Sigma$, we
refer to Fulton's book \cite{Ful}. If the minimal lattice
generators of one dimensional edges of every top dimensional cone
$\sigma \in \Sigma$ form a $\Z$-basis of $N$, then the fan is
called smooth and the corresponding toric variety is nonsingular.
Otherwise, such a fan defines a simplicial toric variety (which
are orbifolds). The toric orbifolds to be considered here are more
general than simplicial toric varieties. They need an additional
data of multiplicity for each $1$-dimensional cone, or
equivalently, a choice of lattice vectors in them.

Let $N$ be the lattice $\Z^n$, and let $M=Hom_\Z(N,\Z)$ be the
dual lattice of rank $n$. Let $N_{\R} = N \otimes \R$ and $M_{\R}
= M \otimes \R$. The set of all $k$-dimensional cones in $\Sigma$
will be denoted by $\Sigma^{(k)}.$ We label the minimal lattice
generators of $1$-dimensional cones in $\Sigma^{(1)}$ as
$\{\bv_1,\cdots,\bv_m\}:= G(\Sigma)$, where  $\bv_j =
(v_{j1},\cdots,v_{jn}) \in N$. For $\bv_j$, consider a lattice
vector $\bb_j \in N$ with
 $\bb_j = c_j \bv_j$ for some positive integer $c_j$. We call $\bb_j$ a stacky vector, and
 denote $\vec{\bb} = (\bb_1,\cdots, \bb_m)$. For a simplicial rational polyhedral fan $\Sigma$,
 the stacky fan $(\Sigma, \vec{\bb})$  defines  a toric orbifold as follows.

We call a subset $\CP = \{\bv_{i_1},\cdots,\bv_{i_p}\} \subset G(\Sigma)$ a {\em primitive
collection} if $\{\bv_{i_1},\cdots, \bv_{i_p}\}$ does not generate
$p$-dimensional cone in $\Sigma$, while for all $k \, (0 \leq k <
p)$, each $k$-element subset of $\CP$ generates a $k$-dimensional
cone in $\Sigma$.

Let $\CP = \{\bv_{i_1},\cdots, \bv_{i_p}\}$ be a primitive
collection in $G(\Sigma)$. Denote
$$\BA(\CP) = \{(z_1,\cdots,z_m)\mid z_{i_1}= \cdots=z_{i_p}=0\}.$$
Define the closed algebraic subset $Z(\Sigma)$ in $\C^m$ as
$Z(\Sigma) = \cup_{\CP} \BA(\CP)$, where $\CP$ runs over all
primitive collections in $G(\Sigma)$ and we put $U(\Sigma) = \C^m
\setminus Z(\Sigma).$

Consider  the map $\pi:\Z^m \to \Z^n$ sending the basis vectors
$e_i$ to $\bb_i$ for $i=1,\cdots,m$. Note that the $\K:=Ker(\pi)$
is isomorphic to $\Z^{m-n}$ and that $\pi$ may not be surjective
for toric orbifolds. However, by tensoring with $\R$, we obtain
the following exact sequences.
\begin{equation}\label{kexact2}
 0 \to \kk \to \R^m \stackrel{\pi}{\to} \R^n \to 0.
\end{equation}
\begin{equation}\label{kexact3}
 0 \to K \to T^m \stackrel{\pi}{\to} T^n \to 0.
\end{equation}
\begin{equation}\label{kexact4}
 0 \to K_\C \to (\C^*)^m \stackrel{\pi'}{\to} (\C^*)^n \to 0.
\end{equation}
Here $T^m = \R^m/\Z^m$ and the map $\pi'$ is defined as
$$\pi'(\lambda_1,\cdots,\lambda_m) = (\prod_j \lambda_j^{b_{j1}},\cdots, \prod_j \lambda_j^{b_{jn}}).$$
Here, even though $\K$ is free, $K$ may have non-trivial torsion part.
For a complete stacky fan $(\Sigma, \bb)$, $K_\C$ acts effectively
on $U(\Sigma)$ with finite isotropy groups.
The global quotient orbifold
$$X_{\Sigma} = U(\Sigma)/K_\C$$
is called the {\em compact toric orbifold} associated to the
complete stacky fan $(\Sigma, \bb)$. We refer readers to
\cite{BCS} for more details.

There exists an open covering of $U(\Sigma)$ by affine algebraic
varieties: Let $\sigma$ be a $k$-dimensional cone in $\Sigma$
generated by $\{\bv_{i_1},\cdots, \bv_{i_k}\}.$ Define the open
subset $U(\sigma) \subset \C^m$ as
$$ U(\sigma) = \{(z_1,\cdots,z_m) \in \C^m \mid z_j \neq 0
\;\;\textrm{for all}\; j \notin \{i_1,\cdots,i_k\}\}.$$
Then the open sets $U(\sigma)$ have the following properties:
\begin{enumerate}
\item $U(\Sigma) = \cup_{\sigma \in \Sigma} U(\sigma);$
\item if $\sigma \prec \sigma'$, then $U(\sigma) \subset U(\sigma')$;
\item for any two cone $\sigma_1,\sigma_2 \in \Sigma$,
one has $U(\sigma_1) \cap U(\sigma_2) = U(\sigma_1 \cap \sigma_2)$;
in particular,
$$ U(\Sigma) = \bigcup_{\sigma \in \Sigma^{(n)}} U(\sigma).$$
\end{enumerate}
We define the open set $U_\sigma := U(\sigma)/K_\C$.
For toric orbifolds, $U_\sigma$ may not be smooth.

The following lemma is elementary (see the case of smooth toric
manifold in \cite{B1} together with the considerations of the
orbifold case in \cite{BCS}).
\begin{lemma}\label{proporbilocal}
Let $\sigma$ be a $n$-dimensional cone in $\Sigma$, with a choice
of lattice vectors $\bb_\sigma = (\bb_{i_1},\cdots, \bb_{i_n})$
from its one dimensional cones. Suppose that $\bb_\sigma$ spans
the sublattice $N_{\bb_\sigma}$ of the lattice $N$. Consider the
dual lattice $M_{\bb_\sigma} \supset M$ of $N_{\bb_\sigma}$, and
the dual $\Z$-basis $(\bu_{i_1},\cdots, \bu_{i_n})$ in
$M_{\bb_\sigma}$ defined by
$$ \langle \bb_{i_k}, \bu_{i_l} \rangle = \delta_{k,l}.$$

Recall that $\sigma$ with the lattice $N_{\bb_\sigma}$ (resp. $N$) gives rise to a space $U'_\sigma$ (resp. $U_\sigma$), and
the abelian group $G_{\bb_\sigma}=N/N_{\bb_\sigma}$ acts on $U'_\sigma$ to give
$$U'_\sigma/G_{\bb_\sigma} = U_\sigma.$$
In terms of the variables $z_1,\cdots,z_m$ of the homogeneous
coordinate $\C^m$, the coordinate functions $x_1^\sigma,\cdots,
x_n^\sigma$ of the uniformizing open set $U'_\sigma$ are given by
\begin{equation}\label{homocordeq}
\begin{cases}
x_1^\sigma  = z_1^{\langle \bb_1,\bu_{i_1} \rangle }\cdots
z_m^{\langle \bb_m, \bu_{i_1} \rangle }\\
\qquad \vdots\\
x_n^\sigma  = z_1^{\langle \bb_1, \bu_{i_n} \rangle }\cdots
z_m^{\langle \bb_m, \bu_{i_n} \rangle }
\end{cases}
\end{equation}
The $G_{\bb_\sigma}$-action on $U_\sigma'$ for $g \in N/N_{\bb_\sigma}$ is given by
\begin{equation}\label{Gactionformula}
g \cdot x_j^\sigma =  e^{2\pi i \langle g,  \bu_{i_j} \rangle} x_j^\sigma.
\end{equation}
\end{lemma}

Now, we discuss $\C^*$-action on $U_\sigma'$ and $U_\sigma$.
In what follows there is a complication because  there exist a $\C^*$-action on
the quotient (coming from the $\C^*$-action on the disc) which does not extend to
the $\C^*$-action on the uniformizing cover.

\begin{lemma} For any lattice vector $w \in N_{\bb_\sigma}$ there is an associated $\C^*$-action on $U_\sigma'$ given by
\begin{equation}\label{cstaract}
 \lambda_w(z) \cdot x_j^\sigma = z^{\langle w, \bu_{i_j} \rangle}x_j^\sigma.
\end{equation}
\end{lemma}
\begin{proof}
From the standard toric theory corresponding to the lattice
$N_{\bb_\sigma}$, for any $w \in N_{\bb_\sigma}$,  there exists an
associated $\C^*$ action: Let $z\in \C^*$, and $\bu \in
M_{\bb_\sigma}$. Toric structure provides action $\lambda_w$ of
$w$ on the function $\chi^{\bu}$ on $U_\sigma'$ by $\lambda_w(z)
\cdot (\chi^{\bu}) = z^{\langle w,\bu \rangle} \chi^{\bu}$. The
lemma follows by writing this formula in terms of coordinates
$(x_1^\sigma, \cdots, x_n^\sigma)$.
\end{proof}
\begin{lemma}\label{lem:cstaract}
For a lattice vector $v \in N$, there is an associated
$\C^*$-action on the quotient $U_\sigma$ as in \eqref{cstaract}.
Furthermore, such a $\C^*$-action induces a morphism $\C \to
U_\sigma$, if $v$ lies in the cone $\sigma$.
\end{lemma}
\begin{proof}
We write $v = \sum_{j} c_j \bb_{i_j}$ for some rational numbers
$c_j$'s. Hence, \eqref{cstaract} does not provide $\C^*$-action of
$v$ on $U_\sigma'$. But there exists a $\C^*$-action of $v \in N$
on the quotient $U_\sigma'/G$. We define the action
$\lambda_{\bv}(z)$ by the formula  \eqref{cstaract}. Then possible
values of $(z^{\langle v, \bu_{i_1} \rangle}, \cdots, z^{ \langle
v, \bu_{i_n} \rangle})$ for different choices of branch cuts
differ by multiplication of $(e^{2\pi i a \langle v, \bu_{i_1}
\rangle}, \cdots, e^{2\pi i a \langle v, \bu_{i_n} \rangle})$ for
some integer $a \in \Z$. Therefore the difference is exactly given
by the $G$-action \eqref{Gactionformula}.

% Now, from the formula \eqref{cstaract}, such an action
%extends over 0 if all the exponents are non-negative.
The $\C^*$-action corresponding to $v$ defines a map from $\C^*$
to the principal $(\C^*)^n$ orbit of the toric variety.
 If $v$ lies
in the cone $\sigma$, we have $\langle v, \bu_{i_j} \rangle \;
\geq \; 0$ for all $j$. In this case the above map extends to a
map from $\C$ to $U_\sigma$ (see \cite{Ful}, chapter 2.3).
\end{proof}

\begin{definition}\label{def:box}
Let $\sigma$ be an $d$-dimensional cone in $\Sigma$ with a choice
of lattice vectors $\bb_\sigma=(b_{i_1}, \ldots, b_{i_d})$. Let
$N_{\bb_\sigma}$ be the submodule of $N$ generated by these
lattice vectors.
 Define
$$
Box_{\bb_\sigma} =\{ \nu \in N \mid \nu=\sum_{k=1}^d c_{k} \bb_{i_k},
c_k \in [0,1)\;\}.
$$
This set has one-to-one correspondence  with  the group
\begin{equation}\label{gbsigma}
G_{\bb_\sigma} = (( N_{\bb_\sigma} \otimes_{\Z} \Q) \cap N ) /
N_{\bb_\sigma}.
\end{equation}
 This generalizes the definition of
$G_{\bb_\sigma} = N/N_{\bb_\sigma}$ given in Lemma
\ref{proporbilocal} for $n$-dimensional cones. It is easy to
observe that if $\sigma \prec \sigma'$, then $Box_{\bb_\sigma}
\subset Box_{\bb_{\sigma'}} $.

 Define
$$Box_{\bb_\sigma}^{\circ} = Box_{\bb_\sigma} - \bigcup_{\tau \prec \sigma}
Box_{\bb_\tau}.
$$
 Define
\begin{equation}
Box = \bigcup_{\sigma \in \Sigma^{(n)}} Box_{\bb_\sigma} =
\bigsqcup_{\sigma \in \Sigma} Box_{\bb_\sigma}^{\circ}.
\end{equation}
We set $Box = \{0\} \sqcup Box'$.
%Note that here even the $0$-dimensional cone is included.
 $Box$ is the index set $T$ of the components of the inertia
orbifold of the toric orbifold corresponding to $(\Sigma,
\vec{\bb})$. To every $\nu \in Box_{\bb_\sigma}^{\circ} \cap
Box'$, there corresponds a twisted sector ${\bX}_{\nu}$ which  is
isomorphic to the orbit closure $\bar{O}_\sigma$ as analytic
variety. However it has a specific orbifold structure that
includes the trivial action of $G_{\bb_\sigma}$. In particular the
fundamental class of ${\bX}_{\nu}$ is $\frac{1}{o(G_{\bb_\sigma})}
[\bar{O}_\sigma]$.
\end{definition}

\begin{remark}
We would like to point out here that there is a natural orbifold
structure on the varieties $\bar{O}_{\tau}$. This comes from
considering it as a toric orbifold with the fan $ star(\tau)$ as
described in section 3.1 of \cite{Ful}: Let $L$ be the submodule
of $N$ generated by $\tau \cap N$ and $N(\tau)= N/L$. Then
$star(\tau)$ is the set of cones containing $\tau$, realized as a
fan in $N(\tau)$. The projection of stacky lattice vectors $\bb_j$
 to $N(\tau)$ gives $\bar{O}_{\tau}$ the desired orbifold structure.
This structure induces an inclusion
 of $\bar{O}_{\tau}$ into $\bX$ as a suborbifold.

 This orbifold structure is in general different from the orbifold
 structure of $\bar{O}_{\tau}$ as an analytic variety. For
 instance when $\dim(\tau)= n-1$, the variety $\bar{O}_{\tau}$ is
 a smooth sphere whereas the above structure may involve orbifold
 singularities. On the other hand this structure also is different
 from the orbifold structure of $\bar{O}_{\tau}$ as a twisted
 sector. It precisely misses the trivial action of $G_{\bb_\tau}$
 corresponding to the group actions in the normal bundle of $\bar{O}_{\tau}$
 in $\bX$. The orbifold structure of $\bar{O}_{\tau}$ as a twisted
 sector induces a different inclusion of it into $\bX$ as a suborbifold.
\end{remark}

%\subsection{Compact toric orbifolds as symplectic quotients}
\subsection{Symplectic toric orbifolds}
\label{sec:forms}  Recall that a symplectic toric manifold is a
symplectic manifold that admits Hamiltonian action of a half
dimensional compact torus. Delzant polytopes, which are rational
simple smooth convex polytopes, classify compact symplectic toric
manifolds up to equivariant symplectomorphism. Here we review the
generalization to labeled polytope, a polytope together with a
positive integer label attached to each of its facets,
 by  Lerman and Tolman \cite{LT}. Labeled polytopes classify compact
  symplectic toric orbifolds.
We recall briefly the explicit construction of  symplectic toric
orbifold from a labeled polytope following \cite{LT} (see Audin
\cite{Au} for example in the smooth case).

\begin{definition}
A convex polytope $P$ in $M_\R$ is called {\em simple} if there
are exactly $n$ facets meeting at every vertex. A convex polytope
$P$ is called {\em rational} if a normal vector to each facet $P$
can be given by a lattice vector. A simple polytope $P$ is called
{\em smooth} if for each vertex, the $n$ normal vectors to the
facets meeting at the given vertex form a $\Z$-basis of $N$.
\end{definition}

Let $P$ be a simple rational convex polytope in $\R^n$ with $m$
facets, with a positive integer assigned to each facet of $P$.
\begin{definition}
We denote by $\bv_j$ the inward normal vector
 to $j$-th facet of $P$, which is primitive and integral, for $j=1,\cdots m$.
 Let $c_j$ be a positive integer label to the $j$-th facet of $P$ for each $j$.
   Set $\bb_j = c_j \bv_j$.
\end{definition}
The polytope $P$ may be described as follows by choosing suitable
$\lambda_j \in \R$:
\begin{equation}\label{def:polytope}
P = \bigcap_{j=1}^m \{ x \in M_\R \mid  \langle x, \bb_j \rangle
\geq \lambda_j \}.
\end{equation}

If we denote (as in \eqref{eq:lj})
$$\ell_j(u) = \langle u,\bb_j \rangle -\lambda_j,$$
then the polytope $P$ may be defined as
$$P = \{ u \in M_\R\mid \ell_j(u) \geq 0, j=1,\cdots,m\}.$$

From a polytope $P$, there is a standard procedure to get a
simplicial fan $\Sigma(P)$. Then the stacky fan $(\Sigma(P),
\vec{\bb})$ defines a toric orbifold in the sense of complex
orbifolds as explained in the last subsection. In this paper we
are only concerned with toric orbifolds derived from labeled
polytopes.

We recall a theorem by Lerman and Tolman.
\begin{theorem}\cite{LT}\label{thm:LeTo}
Let $(M,\omega)$ be a compact symplectic toric orbifold, with
moment map $\mu_T : M \to (\R^n)^\ast$. Then $P = \mu_T(M)$ is a
rational simple convex polytope. For each facet $F_j$ of $P$,
there exists a positive integer $c_j$, the label of $F_j$, such
that the structure group of every $p \in \mu_T^{-1}(int(F_j))$ is
$\Z/c_j\Z$.

Two compact symplectic toric orbifolds are equivariantly
symplectomorphic
%(with respect to a fixed torus $T$ acting on both sides)
if and only if their associated labeled polytopes are
isomorphic. Moreover, every labeled polytope $P$ arises from some
compact symplectic toric orbifold $(M_P,\omega_P)$.
\end{theorem}

Before we recall the explicit construction of symplectic toric orbifolds,
 we remark that the isotropy group of each point $p \in M_P$
can be easily seen from the polytope (Lemma 6.6 of \cite{LT}):
First, the points $p$ with $\mu_T(p) \in int(P)$ have trivial
isotropy group. If $\mu_T(p)$ lies in the interior of a facet $F$,
which has a label $c_F$, the isotropy group is $\Z/c_F\Z$. For the
points $p$ with $\mu_T(p)$ lying in the interior of a face $F$,
which is the intersection of facets, say $F_1,\cdots F_j$, the
isotropy group at $p$ is isomorphic to $A_p/A_p'$: Here, consider
the subtorus $H_p\subset T^n$ whose Lie algebra $h_p$ is generated
by $\bv_i\otimes 1 \in N_\R$ for $i=1,\cdots,j$. Let  $A_p$ be the
lattice of the circle subgroups of $H_p$. Let $A_p'$ be the
sublattice generated by $\{c_i \bv_i\}$. We remark that even when
$c_1=\cdots=c_m=1$, there can be orbifold singularities as
$\{\bv_1,\cdots, \bv_j\}$ may not a form $\Z$-basis of $N$.

Note that the face $F$ corresponds to a $j$-dimensional cone
$\sigma$ in the fan $\Sigma(P)$ with stacky vectors $\{c_i \bv_i:
i=1,\cdots,j\}$. Then the group $A_p'$ is same as the group
$N_{\bb_\sigma}$ (see Definition \ref{def:box}), and $A_p$ is same
as $ N_{\bb_\sigma} \otimes_{\Z} \Q $. Therefore the isotropy
group $A_p/A_p'$ is identical to $G_{\bb_\sigma}$.

% Here let $A_F'$
%be the sublattice generated by $\bb_1, \ldots, \bb_j$. Let $A_F =
%(A_F' \tensor_{\Z} \Q)\cap \Z^n$

%

We briefly recall the construction of the symplectic toric
orbifold $(M_P,\omega_P)$ from the labeled simple rational
polytope $P$.

Recall from \ref{kexact2} that for the standard basis $(e_1, \ldots, e_m)$ of $\R^m$,
the map $\pi$ is defined by
\begin{equation}\label{def:pi}
\pi : \R^m \to \R^n\ \ \mbox{by}\  \ \pi(e_j) = c_j \bv_j, \;
j=1,\ldots,m
\end{equation}
producing the following exact sequences:
$$ 0 \to \kk \stackrel{\iota}{\to} \R^m \stackrel{\pi}{\to}
\R^n \to 0
\ \ \ \mbox{and its dual}\ \ \
0 \to (\R^n)^\ast \stackrel{\pi^\ast}{\to} (\R^m)^\ast
\stackrel{\iota^\ast}{\to} \kk^\ast \to 0\ .
$$
Note that $\kk$ is the Lie algebra of $K$ defined in
\ref{kexact3}.

Consider $\C^{m}$ with its standard symplectic form
$$\omega_0 = \frac{i}{2} \sum dz_k \wedge d\overline{z}_k.$$
The standard action of $T^m$ on $\C^m$ is Hamiltonian whose moment map is given by
$$\mu_{\C^m}(z_1,\cdots,z_m) = \frac{1}{2}(|z_1|^2,\cdots,|z_m|^2).$$
Hence $K$ acts on $\C^m$ with the moment map
$$ \mu_K=\iota^* \circ \mu_{\C^m}: \C^m \to \kk^*.$$
For the constant vector $\lambda = (\lambda_1,\cdots,\lambda_m)$
defining the polytope (\ref{def:polytope}), define
$\pi_\lambda^*:(\R^n)^* \to (\R^m)^*$ by $\pi_\lambda^*(\xi) =
\pi^* \xi - \lambda$. Then,
\begin{eqnarray}
\pi_\lambda^*(P) &=& \{\xi \in (\R^m)^*|  \xi \in Im(\pi_\lambda^*) \; \textrm{and} \; \xi_i \geq 0 \;\textrm{for all} \;i \}\\
 &=&  \{\xi \in (\R^m)^*|  \xi \in (\iota^*)^{-1}(\iota^*)(-\lambda) \; \textrm{and} \; \xi_i \geq 0 \;\textrm{for all} \;i \}
\end{eqnarray}
Then, take $X=\mu_K^{-1}(\iota^*(-\lambda))/K$ to be the
symplectic quotient, which is the desired (K\"{a}hler) toric
orbifold. Since the action of $T^m$ commutes with $K$, there
exists an induced $T^m$ action on $X$ and the $T^m$ action
descends to $T^m/K$ action on $X$, and providing the moment map
$\mu_T = (\pi^*_\lambda)^{-1} \circ \mu_{\C^m}$ on $X$.

\section{Desingularized Maslov index formula for toric orbifolds}
We first recall the Maslov index formula of holomorphic  discs in
toric manifolds in terms of intersection numbers.
\begin{theorem}[\cite{C}\,\cite{CO}]
 For a symplectic
toric manifold $X_{\Sigma(P)}$, let $L$ be a Lagrangian $T^n$
orbit. Then the Maslov index of
 any holomorphic disc with boundary lying on $L$ is twice the
sum of intersection multiplicities of the image of the disc with
the divisors $D_j$ corresponding to $\bv_j \in \Sigma^{(1)}$, over
all $j=1,\cdots,m$.
\end{theorem}
Here the divisor $D_j$ is a complex codimension one submanifold,
which can be defined using the principal bundle $(U(\Sigma)
\stackrel{\pi}{\to} X_{\Sigma(P)})$ as $D_j = \pi(\{z_j=0\}) =
\{z_j=0\}/K_\C$. For a toric orbifold $X$, the divisor $D_j$ can
be defined similarly as a suborbifold of $X$ by $D_j=
\{z_j=0\}/K_\C$.

In this section, we find a similar formula for toric orbifolds.
Consider an  orbi-disc $\bD$  with interior marked points
$z_1^{+},\cdots, z_k^{+}$ each of which have orbifold
singularities $\Z/m_i\Z$. (Here $m_i=1$ for smooth marked points.)

Here is the desingularized Maslov index theorem for toric
orbifolds. Note that intersections of holomorphic orbi-discs with
divisors are discrete and there are only finitely many of them
because the map is holomorphic. The  multiplicity of such an
intersection is given by the ordinary intersection number in the
uniformizing cover (or  in homogeneous coordinates of
$U(\Sigma)$), divided by the order of local group of the orbi-disc
at the intersection point.

\begin{theorem}\label{thm:indexformula}
For the symplectic toric orbifold $X$ corresponding to
$(\Sigma(P),\bb)$, let $L$ be a Lagrangian $T^n$ orbit and let
$(\bD, (z_1^+,\cdots, z_k^+))$ be an orbi-disc with $\Z/m_i\Z$
singularity at $z_i^+$. Consider a holomorphic orbi-disc $w
:(\bD,\partial \bD) \to (X,L)$ intersecting the  divisor $D_j$
with multiplicity $m_{i,j}/ m_i$ at each marked point $z_i^+$, and
do not intersect divisors away from marked points. Then the
desingularized Maslov index of $w$ is given as
$$2\sum_i \sum_j (\lfloor m_{i,j}/m_i \rfloor ).$$
Here $\lfloor r \rfloor$ denotes the largest integer equal to or
less than $r$.
\end{theorem}

\begin{proof}
Recall that in \cite{C} and \cite{CO}, the Maslov index was
computed as a sum of local contributions near each intersection
with divisors. A similar scheme still works in this setting. The
local contribution at each intersection point has been computed in
the Lemma \ref{generalorbicompute}. Hence it remains to show how
to modify the general scheme in the setting of toric orbifolds.

Without loss of generality, we discuss what happens in the
neighborhood of $z_1^+$ only. The point $w(z_1^+)$ may lie in the
intersection of several divisors $D_j$'s. Suppose that
\begin{equation}
w(z_1^+) \in ( D_{i_1} \cap \cdots \cap D_{i_k}).
\end{equation}
We may assume $w(z_1^+)$ do not intersect any other toric divisor.
The fact that $D_{i_1} \cap \cdots \cap D_{i_k}) \neq 0$ implies
that $\{\bv_{i_1},\cdots,\bv_{i_k}\}$ is not a primitive
collection, hence we can choose lattice vectors
$\bv_{i_{k+1}},\cdots,\bv_{i_n}$ so that $\langle
\bv_{i_1},\cdots,\bv_{i_n} \rangle $ defines a $n$-dimensional
cone $\sigma$ in $\Sigma$.

We may consider the map $w$ in a uniformizing neighborhood
$U_\epsilon(z_1^+)$ of $z_1^+$. We consider its uniformizing cover
$D_\epsilon(z_1^+) \to U_{\epsilon}(z_1^+)$ which is the
$m_1$-fold
 branch cover branched at the origin. By the definition of orbifold
  holomorphic map, we can
consider its equivariant lift $\WT{w}: D_\epsilon(z_1^+) \to
U_\sigma'$ for the uniformizing
 chart $U_\sigma'$ as in the Lemma \ref{proporbilocal}. The intersection multiplicity
  $m_{1,j}$ can be defined as the order of zero at $z_1^+$ of the coordinate
$x_j^\sigma$ in $U_\sigma'$ for $1 \leq j \leq k$. As $w(z_1^+)$
do not intersect divisors corresponding to
$\bv_{i_{k+1}},\cdots,\bv_{i_n}$, the coordinate functions
$x_j^\sigma$ for $\WT{w}$ are non-vanishing near $z_1^+$ when
$j\ge k+1$.

 We note that this multiplicity can be
also seen in the homogeneous coordinates of $\C^m$. From Lemma
\ref{proporbilocal}, for the dual basis
$\{\bu_{i_1},\cdots,\bu_{i_n} \}$ of the linearly independent
vectors $\{\bb_{i_1},\cdots,\bb_{i_n}\}$, the affine coordinate
function $x^\sigma_j$ of $U_\sigma'$ is given as
\begin{eqnarray*}
x^\sigma_j &=& z_1^{\langle \bb_1,\bu_{i_j} \rangle }\cdots
z_m^{\langle \bb_m, \bu_{i_j} \rangle }\\
&=& C(z) \cdot z_{i_j}^{\langle \bb_{i_j}, \bu_{i_j} \rangle } =
C(z)\cdot z_{i_j}
\end{eqnarray*}
where $C(z)$ is a function nonvanishing near $\WT{w}(z_1^+)$.
Hence the order of zero of $z_{i_j}$ equals that of $x^\sigma_j$.

We write the lift  $\WT{w}:D_\epsilon(z_1^+)  \to U_\sigma'$ in
affine coordinates as
$$(a_1z^{d_1} + \CO(z^{d_1+1}),\cdots,a_kz^{d_k} + \CO(z^{d_k+1}),
a_{k+1} + \CO(z),\cdots,a_n + \CO(z)),$$ where $z=0$ corresponds
to the point $z_1^+$.

The lift $\WT{w}$ is equivariant and hence the dominating term
\begin{equation}\label{localorbidisc}
(a_1z^{d_1} ,\cdots,a_kz^{d_k} ,a_{k+1} ,\cdots,a_n ),
\end{equation} is also equivariant in $D_\epsilon(z_1^+)$.

Now we are in the similar situation as in the smooth case
\cite{C},\cite{CO} and analogously we smoothly deform the map
$\WT{w}$  in $D_\epsilon(z_1^+)$ in an equivariant way, without
changing it near the boundary of this disc,
 so that the deformed map $\widetilde{w}$
satisfies
\begin{equation}
\widetilde{w}|_{\partial D_{\epsilon/2}(z_1^+)} \subset L.
\end{equation}
We can make the deformation so that the map $\widetilde{w}$ on
$D_{\epsilon/2}(z_1^+)$ is given by
\begin{equation}\label{map}
\Big(\frac{a_1 z^{d_1}}{|a_1|(\frac{\epsilon}{2})^{d_1}},\cdots,
\frac{a_k
z^{d_k}}{|a_k|(\frac{\epsilon}{2})^{d_k}},\frac{a_{k+1}}{|a_{k+1}|},
\cdots,\frac{a_n}{|a_n|}\Big).
\end{equation}

We perform the same kind of deformations for
$z_2^+,z_3^+,\cdots,z_k^+$ inside the uniformizing neighborhoods
$D_\epsilon(z_2^+),\cdots,D_\epsilon(z_k^+)$ for sufficiently
small $\epsilon$ and write the resulting map as $\widetilde{w}'$
and the corresponding map of orbifolds as $w'$. Over the punctured
disc $$ S = \bD \setminus (U_\epsilon(z_1^+)  \cup \cdots
U_\epsilon(z_k^+)),$$ the deformed map $w'$ does not intersect
with the toric divisors, and it intersects with the Lagrangian
torus $L$ along the boundaries of the punctured disc.

\begin{lemma}
The desingularized Maslow indices of $w$ and $w'$ are equal to
each other: $\mu^{de}(w) = \mu^{de}(w')$.
\end{lemma}
\begin{proof}
As the desingularized complex vector bundle of $w$ and $w'$ will
be isomorphic as a bundle pair, hence has the same desingularized
Maslov index.
\end{proof}

Hence, it is enough to compute  $\mu(w')$. Since every
intersection with the toric divisors occurs inside the balls
$D_{\epsilon/2}$, $w'|_{S}$ does not meet the toric divisors. So
it can be considered as a map into the cotangent bundle of $L$.
Therefore we have
\begin{equation}\label{indexzero}
\mu(w'|_{S}) = 0.
\end{equation}
On the other hand, the Maslov index of the map $w'|_{S}$ is given
by the sum of the Maslov indices along $\partial S$ after fixing
the trivialization.

Now consider the map $w':\bD \to \bX$ and the pull-back bundle
${w'}^*TX$ and its desingularization  $({w'}^*TX)^{de}$. We fix a
trivialization $\Phi$ of $({w'}^*TX)^{de}$. When restricted to
$S$, $\Phi$ gives a trivialization $\Phi_{S}$ of
$((w'|_{S})^*TX)^{de}$ restricted over $S$, which does not contain
any orbifold point. In this trivialization, it is easy to see that
the Maslov index along the boundary $\partial D^2$ in $\partial S$
is the desingularized Maslow index $\mu(w)=\mu(w')$. Along the
rest of boundaries $\partial U_{\epsilon/2}(z_i)$  of $S$, which
are oriented in the opposite way, the Maslov indices equal the
negatives of the local contributions of desingularized Maslov
indices and hence is $-2\sum[m_{i,j}/m_i]$ for each $i$ by the
lemma \ref{generalorbicompute}. This proves the theorem.
\end{proof}

\section{Orbifold Holomorphic discs in toric orbifolds}\label{sec:basic}

In this section, we classify all holomorphic discs  and orbi-discs
in toric orbifolds with boundary on $L(u)$. We find one-to-one
correspondence between  non-trivial twisted sectors in $Box'$ and
orbifold holomorphic discs with a single interior orbifold
singularity (modulo $T^n$-action).
 We also find one-to-one correspondence between the stacky vectors $\bb_j$ of the fan and
 smooth holomorphic discs of Maslov index two (modulo $T^n$-action).

These two types of discs will be called {\em basic discs} for
simplicity:  Namely, Maslov index two smooth holomorphic discs and
 holomorphic orbi-discs having one
interior orbifold singularity and desingularized Maslov index
zero.
 Basic discs will be used to define  Landau-Ginzburg
   potentials $PO_{0}$ and $PO_{orb,0}^{\frak b}$, and
 will be used for computing Lagrangian Floer cohomology of torus fibers.
\subsection{Classification theorem}
We first recall the corresponding theorem for holomorphic discs in
toric manifolds.
\begin{theorem}
[Classification theorem\cite{C},\cite{CO}]
Let $\widetilde L \subset
\C^m \setminus Z(\Sigma)$ be a fixed orbit of the real $m$-torus $(S^1)^m$.
Any holomorphic map
$w:(D^2,\partial D^2) \to (X_{\Sigma(P)},L)$ can be lifted to a
holomorphic map
$$
\widetilde{w}:(D^2,\partial D^2) \to (\C^m \setminus Z(\Sigma),\widetilde L)
$$
so that each homogeneous coordinates functions
$\widetilde{w}=(\widetilde{w}_1,\cdots, \widetilde{w}_m)$ are given by
Blaschke products with constant factors.
$$i.e. \;\; \widetilde{w}_j =a_j \cdot
\prod_{s=1}^{\mu_j}\frac{z-\alpha_{j,s}}{1-\overline{\alpha}_{j,s}z}$$
for $a_j\in \C^*$ and non-negative integers $\mu_j$ for each
$j=1,\cdots, m$ and $\alpha_{j,s} \in int(D^2)$. In particular, there is no non-constant
holomorphic discs of non-positive Maslov indices.
\end{theorem}

We start by explaining the new basic factors of holomorphic orbi-discs (
in addition to the factor  $\frac{z-\alpha}{1 - \OL{\alpha}z}$ used in  the smooth cases  above).

Consider a $n$-dimensional  stacky cone $(\sigma, \bb_\sigma)$
with $\bb_\sigma= \{\bb_{i_1},\cdots, \bb_{i_n}\}$. Take an
element  $\nu = c_1\bb_{i_1} + \cdots + c_n \bb_{i_n} \in N$,
where $0 \leq c_j <1$ for $j=1,\cdots,n$. Write each $c_j$ as
rational numbers $p_j/q_j$ with relatively prime $\{p_j,q_j\}$.
Let $m_1 = g.c.d.(q_1,\cdots,q_n)$ be the greatest common divisor
of denominators, which is the order of $\nu$ in $G_{\bb_\sigma}$.

Let $\bD$ be a disc $D^2$ with orbifold marked point $z_1^+ \in
\bD$ with $\Z/m_1$ singularity. We find an explicit formula for a
holomorphic orbi-disc $w$ from $\bD$ such that the generator of
$\Z/m_1$ maps to $\nu \in G_{\bb_\sigma} =G_{w(z_1)} $. We denote
by $\phi_{z_1^+}:\Z/m_1 \to G_{\bb_\sigma}$ be an injective group
homomorphism sending the generator 1 to $\nu$.

Consider the open set $U_\sigma'$ and its coordinate functions $x_1^\sigma,\cdots, x_n^\sigma$.
In this coordinate, choose a point  $(a_1,\cdots, a_n)$ in the Lagrangian fiber $L$.
We consider the expression
\begin{equation}\label{newbasicfactor}
\big(a_1 (\frac{z-z_1^+}{1-\OL{z}_1^+ z} )^{c_1}, \cdots, a_n
(\frac{z-z_1^+}{1-\OL{z}_1^+ z} )^{c_n}\big)
\end{equation}
As $c_i$'s are rational numbers, expression such as $z^{c_i}$ for
$z \in D^2$ is not well-defined, and depends on the choice of a
branch cut. But, recall that $G_{\bb_\sigma}$ acts on $U_\sigma'$
by \eqref{Gactionformula}, and the difference from the choice of a
branch cut is given by  this action. (see the proof of  Lemma
\ref{lem:cstaract}). Hence,  the expression \eqref{newbasicfactor}
is well-defined in $U_\sigma=U'_\sigma/G_{\bb_\sigma}$. It is not
hard to check that the image of $z=z_1^+$ of
\eqref{newbasicfactor}  has $\nu$ as a stabilizer. From
\eqref{homocordeq}, one can easily lift \eqref{newbasicfactor} to
the homogeneous coordinate of toric orbifolds. This will be the
new basic factor in the classification of holomorphic
(orbi)-discs. This is a holomorphic orbi-disc, which is a good
map.

Now, we state the classification theorem of holomorphic (orbi)-discs in toric orbifolds.

\begin{theorem}\label{thm:class}
Let $\bX$ be a toric orbifold corresponding to $(\Sigma(P), \bb)$,
and $L$ be a Lagrangian torus fiber. Let $\widetilde L \subset
\C^m \setminus Z(\Sigma)$ be a fixed orbit of the real $m$-torus
$(S^1)^m$. A holomorphic map $w:(\bD,\partial \bD) \to (\bX,L)$
with orbifold singularity at marked points $z_1,\cdots,z_k$ can be
described as follows.

\begin{enumerate}
\item For each orbifold marked point $z_i^+$, the map $w$
associates to it a twisted sector $\nu^i = \sum_j c_{ij} \bb_{i_j}
\in Box$. \item For analytic coordinate $z$ of $D^2 = |\bD|$, $w$
can be written as a map
$$
\widetilde{w}:(D^2,\partial D^2) \to ((\C^m \setminus
Z(\Sigma))/K_\C ,\widetilde L / (K_\C\cap T^m) )
$$
so that each homogeneous coordinates functions (modulo $K_\C$-action)
$\widetilde{w}=(\widetilde{w}_1,\cdots, \widetilde{w}_m)$
 are given as
\begin{equation}\label{geneq}
 \widetilde{w}_j = a_j \cdot
\prod_{s=1}^{d_j}
\frac{z-\alpha_{j,s}}{1-\overline{\alpha}_{j,s}z} \prod_{i=1}^{k}
(\frac{z-z_i^+}{1-\overline{z}_i^+ z})^{c_{ij}}
\end{equation}
for $a_j\in \C^*$, non-negative integers $d_j$ for each
$j=1,\cdots, m$, $\alpha_{j,s} \in int(D^2)$ and rational numbers
$c_{ij}$ as in (1).
 \item The
desingularized Maslov index of the map $\WT{w}$ given as in
\eqref{geneq} is $\sum_{s=1}^{m} 2d_j$. The CW Maslov index of
$\WT{w}$ is $\sum_{s=1}^m 2d_j + 2 \sum_{i=1}^k \iota(\nu^i)$.
\item $\widetilde{w}$ is holomorphic in the sense of Definition
\ref{def:holo}.
\end{enumerate}
\end{theorem}

\begin{remark}
Note the expression is not well-defined as a map to $(\C^m
\setminus Z(\Sigma))$, since $c_{ij}$ are rational numbers. But it
is well-defined up to $K_\C$-action.
\end{remark}
\begin{proof}
We first claim the above expression \eqref{geneq} defines a holomorphic map
in the sense of Definition \ref{def:holo}. The first factor of \eqref{geneq} is obviously holomorphic, and
we may assume that the map $\WT{w}$ is given by
\begin{equation}\label{geneq1}
 \widetilde{w}_j = a_j \cdot \prod_{t=1}^{k} (\frac{z-z_i^+}{1-\overline{z}_i^+ z})^{c_{ij}}.
\end{equation}
Note that $(\frac{z-z_i^+}{1-\overline{z}_i^+ z})^{c_{ij}}$ is
holomorphic in $D^2$ away from $z_i^+$ in the sense of  Definition
\ref{def:holo}. Thus, it suffices to consider  the map
$(\frac{z-z_i^+}{1-\overline{z}_i^+ z})^{c_{ij}}$ near $z_i^+$.

Let $r$ be the order of $\sum_j c_{ij} \bb_{_j}$, which is  the
least common multiple of denominators of rational numbers
$c_{i1},\cdots, c_{im}$. By the automorphism $\phi_{z_i^+}:D^2 \to
D^2$, $\phi_{z_i^+} = \frac{z-z_i^+}{1-\overline{z}_i^+ z}$, and
its inverse $\phi_{-z_i^+}$, we may only consider the case that
$z_i^+=0$. Then consider the branch covering map at ${z_i^+}$,
$br: B_\epsilon(0) \to B_{\epsilon^r}(0)$, which is defined by
$br(\WT{z}) = (\WT{z})^r$. Here, we write the coordinate on the
cover by $\WT{z}$ with the relation $z = \WT{z}^r$. Thus, it is
easy to see that the map $ z^{c_{ij}} = \WT{z}^{r c_{ij}}$ is
holomorphic. Thus the lift, as a map of $\WT{z}$ is holomorphic,
as required by Definition \ref{def:holo}.

Now, we prove the classification results. The idea of the proof is
similar to that of \cite{C} and \cite{CO}. Namely, given a
holomorphic smooth or orbidisc, we consider intersection with
toric divisors, and by dividing by the basic factors, we remove
the intersection with toric divisors to obtain a map which does
not intersect any toric divisors. Then, it is easy to see that the
resulting smooth disc  whose image lies in one of the uniformizing
charts $(\C^n, (S^1)^n)$ of the toric orbifold and has vanishing
Maslov index. By classical classification of smooth holomorphic
discs, it is in fact a constant map.

Let $w:(\bD,\partial \bD) \to (\bX,L)$ be a  holomorphic good
orbidisc. Choose an interior orbifold marked point $z_i^+$ with
$\Z/m_i$ singularity. Denote by $\phi_{z_i^+}$ the injective group
homomorphism $\Z/m_i \to G_{w(z_i^+)}$ associated to the good map
$w$ at $z_i^+$. Take a toric open set $U_\sigma$ containing
$w(z_i^+)$, and denote the stacky vectors generating $\sigma$
(over $\Q$) by $\bb_{i_1},\cdots, \bb_{i_n}$. Then, the image of
generator under $\phi_{z_i^+}$ can be written as
$$\phi_{z_i^+}(1) =: \nu^i = c_{i_1}\bb_{i_1} + \cdots + c_{i_n} \bb_{i_n} \in N$$
with $0 \leq c_{i_j} <1$ for $j=1,\cdots,n$. Write each $c_{i_j}$
as rational numbers $p_{i_j}/q_{i_j}$ with relatively prime
$\{p_{i_j},q_{i_j}\}$, and observe that since $\phi_{z_i^+}$ is
injective, we have  $m_i = l.c.m.(q_{i_1},\cdots,q_{i_n})$, which
is the order of $\nu^i$ in $G_{ \bb_\sigma}$. For simplicity, we
assume that $z_i^+ =0 \in D^2$. Consider the branch cover
$br:B_\epsilon(0) \to B_{\epsilon^r}(0)$ defined by $br(\WT{z}) =
\WT{z}^{m_i}$. The map $w$ restricted on $B_{\epsilon^{m_i}}(0)$,
has a lift (by definition) $\WT{w}:B_\epsilon(0) \to U_\sigma'$,
which is holomorphic on $\WT{z}$. Note that the image of $z_i^+
=0$, $\WT{w}(0)$ has $\nu^i$ in its stabilizer. Hence, in terms of
the coordinates $(x_1^\sigma,\cdots, x_n^\sigma)$ on $ U_\sigma'$,
the $j$-th coordinate of $\WT{w}(0)$ vanishes if $c_j \neq 0$. We
denote the vanishing order (multiplicity) of $\WT{w}(0)$  at
$j$-th coordinate by $d_{i_j}$. (Here $d_{i_j}=0$ if it does not
vanish).

%As  $\WT{u}$ is  $\phi_{z_i}$-equivaraint, we observe that $m_i c_{i_j}$ divides $d_{i_j}$.

We set
$$d_{i_j} = d_{i_j}' m_i + r_{i_j}, \; {\rm where}\; 0 \le r_{i_j} < m_i. $$
By equivariance of $\WT{w}$, we have $$ \frac{r_{i_j}}{m_i} =
c_{i_j}.$$ Thus $\WT{w}$   can be written  near $0$ in these
coordinates as
$$(\WT{z}^{d_{i_1}}\WT{w}'_1, \cdots, \WT{z}^{d_{i_n}}\WT{w}'_n)$$
with $\WT{w}'_j(0) \neq 0$. Or, in the coordinate $z = \WT{z}^{m_i}$,
we have
$$ (z^{d_{i_1}'+c_{i_1}}\WT{w}'_1, \cdots, z^{d_{i_n}'+c_{i_n}}\WT{w}'_n)$$

For the general $z_i^+$ (when $z_i^+ \neq 0$), similarly we have
\begin{equation}\label{basfor}
 ((\frac{z-z_i^+}{1-\OL{z}_i^+ z} )^{d_{i_1}'+c_{i_1}}\WT{w}'_1,
  \cdots, (\frac{z-z_i^+}{1-\OL{z}_i^+ z} )^{d_{i_n}'+c_{i_n}}\WT{w}'_n).
 \end{equation}

We multiply the reciprocals $(\frac{1-\OL{z}_i^+ z}{z-z_i^+}
)^{d_{i_j}' + c_{i_j}} $ to the above to remove the intersection
with toric divisors at $z_i^+$. Such a multiplication can be done
via toric action. Namely, from the \ref{lem:cstaract}, we have a
$\C^*$-action, corresponding to the lattice vector $-\sum_j
(d_{i_j}'+c_{i_j})\bb_j \in N$ on $\bX$. More precisely, this
action corresponds to the multiplication in (homogeneous)
coordinates of $\C^m$ by the following expression
 $$(1,\cdots, (\frac{1-\OL{z}_i^+ z}{z-z_i^+} )^{d_{i_1}'+c_{i_1}}, 1, \cdots,
 (\frac{1-\OL{z}_i^+ z}{z-z_i^+} )^{d_{i_n}'+c_{i_n}},1,\cdots, 1).$$

We denote the resulting holomorphic orbi-disc by
$w_1:(\bD',\partial \bD') \to (\bX,L)$ which is obtained after
such multiplication where $\bD'$ is an orbifold disc obtained from
$\bD$ by removing the orbifold marked point $z_i^+$.

It is easy to see that the  map $w_1$ still satisfies the Lagrangian boundary condition,
 and more importantly  the intersection with toric divisor at $z_i^+$ has been removed.

The case that $w$ intersecting toric divisor at smooth point (which is not a marked point) can be
done as in \cite{CO} and the analogous modified map has less intersection with toric divisors.
By repeating this process, we may assume that we obtain a map $w_d$ which does not meet
 any toric divisor. This map is now smooth, and have Maslov index 0 from the Maslov index formula
 of the Theorem \ref{thm:indexformula}. It is easy to see that the map $w_d$ is indeed a constant map.Thus the formula of
 the original map $w$ can be written as in the statement of the theorem by tracing backwards.

 The index formula (part (3)) follows from Theorem \ref{thm:indexformula}.
 However, a more intuitive way to think about it is as follows: Note that $\mu_{CW}$ is homotopy
invariant and so is $\mu^{de}$ as long as we do not change the
twisted sector data $\bx$. Especially, when the disc splits into
several discs, the sum of $\mu_{CW}$ remains the same. Hence,
given an expression \eqref{geneq}, we consider the degeneration of
the holomorphic disc by sending each $\alpha_{j,s}$ to the
boundary $\partial D^2$. In this case, disc bubble appear, and the
component $\frac{z-\alpha_{j,s}}{1-\OL{\alpha}_{j,s}z}$ disappears
from \eqref{geneq}. Note that if $|\alpha|=1$, then
$\frac{z-\alpha}{1-\OL{\alpha}z}=-1$. The bubble is the standard
Maslov index two disc, hence has $\mu_{CW}=2$. Similarly, we can
bubble off each orbifold marked point to obtain an orbifold disc
bubble, and for each  $z_i^+$, corresponding Chern Weil
 Maslov index is $\mu_{CW} =2 \iota(\nu^i)$. By adding them up, we obtain (3).
\end{proof}

\subsection{Classification of basic discs}
In this subsection, we discuss the classification of basic discs.

Now, we find holomorphic orbi-discs of desingularized Maslov index
0 with one interior orbifold marked point and show that they are
in one-to-one correspondence with twisted sectors.

\begin{corollary}\label{cor:orbidisc}
The holomorphic orbi-discs with one interior singularity and
desingularized Maslov index 0 (modulo $T^n$-action and
automorphisms of the source disc) correspond to the twisted
sectors $\nu \in Box'$ of the toric orbifold.
\end{corollary}

\begin{proof} Let $w$ be a holomorphic orbi-discs with one orbifold
 marked point $z_1^+ \in \bD$ with
$\mu^{de}=0$. Let $\nu = \sum_j c_{j} \bb_{j}$ be the element of
$Box$ associated to the pair $(w, z_1^+)$  as in part (1) of
Theorem \ref{thm:class}. Injectivity of the homomorphism
$\phi_{z_1^+}$ implies that  $\nu \in Box'$.

 By the the classification theorem, $w$ can be written as
 $$(a_1z^{c_1}, a_2z^{c_2},\cdots, a_mz^{c_m}).$$
And this representation is unique up to $T^n$-action if we impose
the condition that $a_i=1$ whenever $c_i=0$. Conversely, given an
element of $Box'$, we can easily construct such a orbi-disc as
above.

\end{proof}

%\begin{proof}
%Let $w$ be a holomorphic orbi-discs with one orbifold marked point $z_0 \in \bD$ with
%$\mu^{de}=0$. Denote by $\nu = \sum_j c_{j} \bb_{j} \in Box'$ the element of the
%twisted sector. Then from
%the classification theorem, $w$ can be written as
% $$(a_1z^{c_1}, a_2z^{c_2},\cdots, a_mz^{c_m}).$$
%And this map is unique up to $T^n$-action. Conversely, given an element of $Box'$,
%we can easily construct such a orbi-disc as above.
%\end{proof}

We give another way to understand  the above correspondence
between  basic orbi-discs and elements of $Box'$. Such  a
holomorphic orbi-disc $w:\bD \to \bX$ (with orbifold marked point
at $0 \in \bD$) with desingularized Maslov index 0 has an image in
a open set $U_\sigma$ for some $n$-dimensional cone $\sigma$. For
its uniformizing chart $U_\sigma'\cong \C^n$, $w$ has an
equivariant lift to the uniformizing charts, $\WT{w}:D^2 \to
\C^n$, which may be written as
\begin{equation}\label{orbimap}
\WT{w}(\WT{z})=(a_1' \WT{z}^{d_1},\cdots,a_n' \WT{z}^{d_n}) =
(a_1'\WT{z}^{c_1m_{\nu}},\cdots,a_n' \WT{z}^{c_n m_{\nu}})
\end{equation}
where each $d_i $ is a nonnegative integer. Here $D^2$ is the
uniformizing chart of $\bD$ which is a branch cover of degree
 $m_{\nu}$, the order of $\nu$.

% which is in fact given by the least
%common multiple of the denominators of $\{c_1,\cdots,c_n\}$ when
%$c_j$'s are considered as a fraction $p_j/q_j$ with relatively
%prime $p_j,q_j$.

From the explicit expression of $\WT{w}$ in \eqref{orbimap},  note that
 the image of such a holomorphic orbi-disc is invariant under
$S^1$ action. More precisely, if one defines $\C^*$ action by
\begin{equation}\label{eq:vaction}
t \cdot (z_1,\cdots,z_n) = (t^{d_1}z_1,\cdots,t^{d_n}z_n), \textrm{for}\; t\in \C^*,
\end{equation}
the image of (\ref{orbimap}), equals to the the image of
$\C^*_{\leq 1}$-action to the point $(a_1',\cdots,a_n') \in L$,
where  $\C^*_{\leq 1} = \{z \in \C^*| |z| \leq 1\}$. This exactly
corresponds to the Lemma \ref{lem:cstaract} about $\C^*$-actions
on toric orbifolds, which extend to morphisms  $\C \to \bX$.

Summarizing the above discussion, we have seen that the image of
basic orbi-discs corresponds to the image of $\C^*_{\leq
1}$-action, which extends to morphisms $\C \to \bX $. Such
$\C^*$-actions are restricted to those corresponding to elements
of $Box'$.

% (ø and such $\C^*$-action corresponds to elements of $Box'$ ?).\\

Now, we consider holomorphic discs of Maslov index two without orbifold marked points.
We first note that the images of maps from smooth discs can intersect fixed loci of the orbifold.
 The definition of orbifold map requires that the map from smooth discs locally lifts to maps
 to uniformizing charts, hence can intersect the fixed loci.

We also illustrate another important point by the following example:
consider an orbifold map $w$ from orbi-disc $\bD$ with $\Z/m\Z$ singularity in the origin to $\bD'$
   with $\Z/mn\Z$ singularity in the origin, whose lift between uniformizing covers are given
    by $w(z) = z^k$. Then, if $m|k$, then
$w$ may be considered as a smooth disc $w':D^2 \to \bD'$ with the lifted map $\WT{w}':D^2 \to
D^2$ is given by $\WT{w}'(z)=z^{k/m}$.

Hence, given an orbifold holomorphic map $f :\bD \to \bX$, and a local lift $\WT{f}$, the related group homomorphism sometimes cannot be injective, if $\WT{f}$ has high multiplicities.
In such a case, the orbifold structure of $\bD$ has to be (and can be) replaced by
 a less singular or sometimes smooth ones. The correspondence below is best understood in this sense.

\begin{corollary}\label{smoothdisc}
The (smooth) Maslov index two holomorphic discs( modulo
$T^n$-action) are in one-to-one correspondence with the
 stacky vectors $\{\bb_1,\cdots,\bb_m\}$.
\end{corollary}
\begin{proof}
This follows directly from the classification theorem. Namely, let
$w:D^2 \to \bX$ be a smooth holomorphic disc of Maslov index two.
From the classification theorem,  up to automorphism of $D^2$,
such a holomorphic disc is given by
 $(a_1,\cdots, z,\cdots,a_m)$ in $\C^m$.
In the form of expression (\ref{orbimap}),
this corresponds to the case that $c_j=1,m_{\nu}=1$ and all the other $c_i=0$ for
 $i \neq j$. Hence this implies the corollary.
\end{proof}

%Given a fan $\Sigma(P)$, one can define its fan polytope as a convex hull of 1-dimensional generators of the cones, $\Sigma(P)^{(1)}$.
%The lattice point $v$ in the equation (\ref{eq:latt}) lies in the fan polytope if and only if
%$\sum_{j=1}^n \, c_j \le 1$. We remark that there may be other lattice vectors corresponding to the twisted sectors than the ones in
%the fan polytope, even in a Fano toric orbifold. Such an example can be found with a simplical fan in dimension two, whose
%one dimensional cones are generated by $\{(1,0), (-3,1), (1,-1)\}$.

\section{Areas of holomorphic orbi-discs}
In this section, we compute the area of holomorphic orbi-discs. The method to compute them is somewhat different from that of \cite{CO} and
is more elementary.

We first illustrate how the moment map measures the area of
standard orbifold disc. Let $D:= D^2 \in \C$ be the standard disc
with the standard symplectic structure. Let $\bD$ be the orbifold
disc obtained as the quotient orbifold $[D^2/(\Z_n)]$ where the
generator $1 \in \Z_n$ acts on $D^2$ by multiplication of a
primitive $n$-th root of unity. Now, both $D$ and $\bD$ has the
following $S^1$-action. Let $t \in S^1$ and $z \in \bD$. Let $w
\in D^2$  be the coordinate on the uniformizing cover of $\bD$.
Then the actions are
$$ t \cdot z = tz ,\;\; t \cdot w= t^{1/n}w.$$
Note that the $S^1$ action is not well-defined on the uniformizing
cover $D^2$, but well-defined on the quotient orbifold $\bD$. If
we compute the moment maps for $D$ and $\bD$, the length of moment
map image of $D$ is $n$-times of the length of moment map image of
$\bD$.  This is because the vector fields generated by
$S^1$-actions have such a relation. Also, we point out that the
symplectic area of $D$ is also $n$-times of the symplectic area of
$\bD$. In general, the area of a holomorphic orbi-disc $w$ with
one interior singular point can be obtained by taking the
symplectic area of the  lift $\WT{w}:D^2 \to U_\sigma'$
  and dividing it by the order of orbifold singularity of $\bD$.

Recall that symplectic areas are topological invariants. Hence, it
is enough to find symplectic areas of generators of $H_2(X,L)$.
From the Lemma \ref{lem:homgen}, it is enough to find symplectic
areas of the basic discs. We denote the homology class of a disc
corresponding to $\bb_i$ (resp. $\nu \in Box'$) by $\beta_i$
(resp. $\beta_\nu$). Note that for $\nu \in Box'$, if we have $\nu
= c_1 \bb_{i_1} + \ldots + c_n \bb_{i_n}$, then the symplectic
area for $\beta_\nu$ is given as the same linear combination of
the symplectic areas of $\beta_{i_j}$'s. Thus, it suffices to find
symplectic areas of $\beta_{i_j}$'s, which are those of smooth
holomorphic discs corresponding to stacky vectors.

Recall that symplectic form on the toric orbifold is obtained from the standard symplectic
 form of $\C^m$  via symplectic reduction. The strategy is to find a lift of the
 holomorphic map to $U(\Sigma) \subset \C^m$
and compute the area there using the standard symplectic form.

As in the classification theoreom, the smooth holomorphic discs which are basic can be obtained easily
obtained as follows. For simplicity, we state it for $\beta_1$.
Let $\WT{w}_1:D^2 \to \C^m$ be a map given by
 $$\WT{w}_1(z) = (a_1z, a_2,\cdots, a_m),$$
 where $(a_1,\cdots, a_m) \in \WT{L}$ as in the theorem \ref{thm:class}.
 Then if we compose it with the projection $\pi:U(\Sigma) \to X$,
 we have $w_1 = \pi \circ \WT{w}_1:(D^2, \partial D^2) \to (X,L)$, which
 defines a smooth holomorphic disc of homology class $\beta_1$.

 %And $|a_i|$'s are determined from the Lagrangian torus as follows.
 Consider $u = (u_1,\cdots, u_n) \in (\R^n)^*$.
If $L$ is defined by $\mu_T^{-1}(u)$, then, considering the map
$\pi_\lambda^*:(\R^n)^* \to (\R^m)^*$ defined by
$\pi_\lambda^*(\xi) = \pi^*\xi - \lambda$, the image of $\WT{L}$
under the map $\mu_{\C^m}:\C^m \to (\R^m)^*$
 corresponds to the point $\pi_\lambda^*(u)$. In fact,  $\pi_\lambda^*(u)$ is  given by
$$(\langle u,\bb_1\rangle -\lambda_1, \cdots,   \langle u,\bb_m \rangle -\lambda_m) = (\ell_1(u),\cdots, \ell_m(u)).$$

 But recall that for the standard moment map $j$-th coordinate of $\mu_{\C^m}$ is
given by $|z_j|^2/2$.
Hence, with the standard symplectic form, the symplectic area of the lift of $\WT{w}_j$ in $\C^m$ is
 just $\pi r^2$ which is  $2\pi(\ell_j(u))$.  Hence the area of $w_j $ is given by $2\pi \ell_j(u)$.

In fact, due to the difference of complex and symplectic
construction of toric orbifolds, we also need the following
argument in the above computation. Note that the holomorphic disc
$\WT{w}$ does not exactly lie on the level set $\mu_K^{-1}(
\iota^*(-\lambda))$ for the symplectic quotient. In fact, when we
say holomorphic disc $\WT{w}$ in symplectic orbifold, we mean the
following deformed disc which lies in the level set $\mu_K^{-1}(
\iota^*(-\lambda))$: From a general argument due to Kirwan
\cite{Ki}, one can consider negative gradient flow of  the
function
 $||\mu_K - \iota^*(-\lambda)||^2$ inside $U(\Sigma) = \C^m \setminus Z(\Sigma)$.  Negative
gradient flow will reach  critical points, and in this case the
only critical points are the set $\mu_K^{-1}( \iota^*(-\lambda))$.
As the torus $\WT{L}$ already lies in the level set, hence points
on $\WT{L}$ do not move under the homotopy. Thus given a
holomorphic disc in $\WT{w}$, it can be flowed into $\mu_K^{-1}(
\iota^*(-\lambda))$ with boundary image fixed, which gives the
precise holomorphic disc in the symplectic quotient. Then, simple
argument using Stoke's theorem tells us that the symplectic area
of the corresponding disc obtained by flowing to the level set
$\mu_K^{-1}( \iota^*(-\lambda))$ is the same as that of $\WT{w}$.
This proves the desired result.

By adding up homology classes, we obtain
\begin{lemma}\label{lem:area}
For a smooth holomorphic disc of homotopy class $\beta_i$,
its symplectic area is given by $2 \pi \ell_j$.

For a lattice vector $\nu = c_1 \bb_{i_1} + \ldots + c_n \bb_{i_n}$, define
\begin{equation}\label{ellnu}
\ell_\nu = \sum_{j=1}^n c_i \ell_{i_j}
\end{equation}
Then, the area of the holomorphic orbi-disc corresponding  $\nu$ is
given by  $2\pi \ell_\nu (u)$.
\end{lemma}

\section{Fredholm regularity}
In this section, we justify the use of the standard complex structure in the computation of the Floer cohomology in this paper.

\subsection{The case of smooth holomorphic discs in toric orbifolds}
 First author with Yong-Geun Oh have shown the following Fredholm regularity results for toric manifolds:
\begin{theorem}\cite{C}\cite{CO}
 Non-singular holomorphic discs of a toric manifold $M$ with boundary on $L$ are Fredholm regular, i.e.
linearization of $\OL{\partial}$ operator at each map is surjective.
\end{theorem}
This implies
that the moduli space of holomorphic discs
(before compatification) are smooth manifolds of expected dimensions.
Since the standard
complex structure is integrable, linearized operator $D_w$ for a
holomorphic disc $w$ is complex linear and exactly the Dolbeault
derivative $\overline{\partial}$.

We briefly recall the main arguments of the proof regularity in
\cite{CO}. The exact sequence \eqref{kexact2}  induces the exact
sequence of complex vector spaces
\begin{equation}\label{kexact9}
0 \to \C^\kk \to \C^m \stackrel{\pi}{\to} \C^n \to 0
\end{equation}
via tensoring with $\C$ where $\C^\kk$ is the $m-n$ dimensional
subspace of $\C^m$ spanned by $\kk \subset \R^m$. Note that this
exact sequence is equivariant under the natural actions by the
associated complex tori.

Given a holomorphic disc
$w: (D^2,\partial D^2) \to (M,L)$ and denote
$$
E = w ^*TM, \quad F = (\partial w)^*TL.
$$
By considering sheaf of local holomorphic sections of the bundle pair $(E,F)$, one can consider the sheaf cohomology group $H^q(D^2,\partial
D^2;E,F)$, and note that  the surjectivity of the linearization of the disc $w$
is equivalent to the vanishing result
\begin{equation}\label{eq:vanishing}
H^1(D^2,\partial D^2;E,F) = \{0\}.
\end{equation}

Denote by $\widetilde w: (D^2,\partial D^2) \to (\C^m, \widetilde
L)$ be the lifting of $w$, whose boundary lies on
$$
\widetilde L = (S^1)^m \cdot (c_1, \cdots, c_m) \subset
\pi^{-1}(L) \subset \C^m.
$$
We denote by
\begin{eqnarray*}
(E,F) & = &(w^*TM,(\partial w)^*TL) \\
(\widetilde E,\widetilde F) & = & (D^2 \times \C^m,
(\partial \widetilde w)^*(T\widetilde L))) \\
(E_\kk,F_\kk) & = &((\widetilde w)^*(T Orb_{\C^\kk}),
(\partial \widetilde w)^*(T Orb_{\kk}))
\end{eqnarray*}
and by
$$
(\CE, \CF), \quad (\widetilde \CE,\widetilde \CF), \quad
(\CE_\kk,\CF_\kk)
$$
the corresponding sheaves of local holomorphic sections

\begin{lemma}[\cite{CO} Lemma 6.3]
The natural complex of sheaves
\begin{equation}\label{eq:exact}
0 \to (\CE_\kk,\CF_\kk) \to (\widetilde \CE,\widetilde \CF) \to
(\CE,\CF) \to 0
\end{equation}
is exact.
\end{lemma}

In \cite{CO}, Lemma 6.4, the vanishing  $H^1(\widetilde \CE,\widetilde \CF)=0$ is
proved by checking the Fredholm regularity of the trivial bundle pair. And the above exact
sequence then proves the desired  Fredholm regularity for holomorphic discs for the case of toric manifolds.

Now, consider the case of smooth holomorphic discs in toric orbifolds. Note that
the exact sequence \eqref{kexact9} remains true in the case of toric orbifolds.
For smooth discs in orbifolds, the pull-back bundle is smooth vector bundle and also
 we have shown in section \ref{sec:basic} that smooth holomorphic discs admit holomorphic liftings to $\C^m$.
Thus, exactly the same argument as in the case of manifolds proves
the following:
\begin{prop}\label{fred:smooth}
Smooth(non-singular) holomorphic discs of a toric orbifold with boundary on $L$ are Fredholm regular.
\end{prop}

\subsection{The case of orbidiscs}
We only discuss the case of holomorphic orbidiscs with one interior orbifold marked point.
We conjecture that all the holomorphic orbidiscs obtained in the classification theorem are indeed Fredholm regular,
 but we do not know how to prove it in this generality.

Suppose $\bD$ is an orbifold disk $D^2$ with
the $\Z_m$ orbifold singularity at the origin, with
boundary $\bbD$. From a good orbifold map $w: (\bD,\bbD) \to (\bX, L)$ to a toric orbifold,
$ \bE = w ^*T\bX$ defines an orbifold holomorphic vector bundle with
$ F = (\partial w)^*TL$ Lagrangian subbundle at the boundary.
Namely, if we let $\pi:D^2 \to \bD$ be its uniformizing chart, then the vector
bundle $\bE$ may be understood as a holomorphic vector bundle $E \to D^2$ with
 effective $\Z_m$ action on $E$, which acts linearly on fibers.
In addition,  $F|_{\partial D^2} \subset E|_{\partial D^2}$
have induced  $\Z_m$ action from $E$.

Denote by $\CE$  the sheaf of local holomorphic sections of $E$ over $D^2$ and
denote by $(\CE,\CF)$ the sheaf of local holomorphic sections of $E$ over $D^2$ with
 values in $F$ on $\partial D^2$.
 Denote by $\CE^{inv}$ (resp. $(\CE,\CF)^{inv}$) the sheaf of local holomorphic sections of $\bE$ over $\bD$
 (resp. $(\bE,F)$ over $(\bD,\bbD)$ ), which, by definition is the sheaf of local holomorphic invariant
  sections of $E \to D^2$ (resp. $(E,F) \to (D^2,\partial D^2)$) under $\Z_m$-action.

\begin{lemma} Suppose $\CE$ has a fine resolution
$$0 \to \CE \to \mathcal{H}_0 \stackrel{h}{\to} \mathcal{H}_1 \to 0,$$
where $\mathcal{H}_i$ ($i=0,1$) are given an effective $\Z_m$ action so that all arrows are equivariant maps.

Then, $\CE^{inv}$ also admits a fine resolution
$$0 \to \CE^{inv} \to \mathcal{H}_0^{inv} \stackrel{h}{\to} \mathcal{H}_1^{inv} \to 0.$$

Analogous statements for $(\CE,\CF)$ also hold true.
\end{lemma}
\begin{proof}
This is standard fact, since taking invariants is an exact functor up to torsion. But we give
a proof of it for readers convenience for the case of $\CE$.
First we recall that any open cover of an orbifold consisting of uniformized open subsets
admits a partition of unity on $X$ subordinate to it (\cite{CR} Lemma 4.2.1). Hence, if $\mathcal{H}_i$ is a fine sheaf, then
 $\mathcal{H}_i^{inv}$ is also a fine sheaf.
The resulting complex is exact: the injectivity of the first arrow
is obvious. To prove the surjectivity of the last arrow, first
take a preimage in $\mathcal{H}_0$, and its average over $\Z_m$
action still maps to the same element due to equivariance of the
map. The exactness in the middle can be proved similarly.
\end{proof}

Now, sheaf cohomology of $\CE^{inv}$ over $\bD$, or $(\CE, \CF)^{inv}$ over $(\bD, \bbD)$
can be introduced by taking a global section functor as before.
Then the above lemma on taking invariants functor, implies the following lemma:
\begin{lemma}
We have
$$H^{0}(\bD,\CE^{inv}) = H^{0}(D^2, \CE)^{inv},
H^{1}(\bD,\CE^{inv}) = H^{1}(D^2, \CE)^{inv}.$$
$$H^{0}(\bD, \partial \bD;(\CE, \CF)^{inv}) = H^{0}(D^2,\partial D^2;\CE,\CF)^{inv},$$
$$H^{1}(\bD, \partial \bD; (\CE, \CF)^{inv}) = H^{1}(D^2,\partial D^2; \CE, \CF )^{inv}.$$
\end{lemma}
In particular, if $H^{1}(D^2, \CE)=0$  then, $H^{1}(\bD,\CE^{inv})=0$ also.

Now,  this enables us to prove the basic orbifold discs with only one singular point in the interior,
by using the results of the first author and Oh on the Fredholm regularity of holomorphic discs.
Namely, given an orbifold holomorphic disc $w: (\bD,\bbD) \to (\bX, L)$,
by definition, we have a lift $\WT{w}:(D^2,\partial D^2) \to (\bX,L)$, which defines a smooth
holomorphic disc to a toric orbifold. From the Fredholm regularity of smooth holomorphic discs in the
previous section,  we thus have the vanishing of $H^{1}(D^2, \partial D^2; \CE, \CF)$, which
implies $H^{1}(\bD, \partial \bD; (\CE, \CF)^{inv})=0$. This proves:
\begin{prop}
Basic holomorphic (orbi)-discs are Fredholm regular.
\end{prop}

\section{Moduli spaces of basic holomorphic discs in toric orbifolds}
In this section, we find properties of moduli spaces of basic
holomorphic (orbi)-discs.

\subsection{Homology class $H_2(X,L;\Z)$}
For toric manifold $M$ and a Lagrangian torus fiber $L$, recall
that we have the exact sequence
$$0 \to Ker(\pi)\to \Z^m \stackrel{\pi}{\to} \Z^n \to 0,$$
where $\pi$ sends the standard generator $e_i$ to $v_i$. This
exact sequence is isomorphic to the homotopy (or homology) exact
sequence (\cite{FOOO2})
\begin{eqnarray}
0  \to \pi_2(M)  \to \pi_2(M,L) \to \pi_1(L)  \to 0 \\
0  \to H_2(M;\Z)  \to H_2(M,L;\Z)  \to H_1(L;\Z)  \to 0
\end{eqnarray}

For a toric orbifold $X$, the situation is more complicated. For
example, the natural map $\pi:\Z^m \to \Z^n$ sending $e_i$ to
$\bb_i$ is not surjective in general but only $\pi \otimes \Q:\Q^m
\to \Q^n$ is surjective, and also  $\pi_2(X,L)$ has additional
classes corresponding to orbifold discs.

First, we consider the case of a stacky $n$-dimensional cone. Let
$(\sigma, \bb_\sigma)$ an $n$-dimensional stacky cone  with stacky
vectors $\bb_\sigma=(\bb_{i_1},\cdots, \bb_{i_n})$ which lies on
one-dimensional cones of $\sigma$. Denote by $N_{\bb_\sigma}$
the sublattice of $N$ generated by stacky vectors $\bb_\sigma$.
Denote $N/N_{\bb_\sigma}$ by $G_{\bb_\sigma}$ as before. Denote by
$L$ a non-singular torus fiber.

We compute $H_2(X_{\sigma,\bb_\sigma}, L;\Z)$ for
$X_{\sigma,\bb_\sigma}$, the underlying quotient space. Here, $L$
may be replaced by $(\C^*)^n$, which is the non-fixed loci of
$X_{\sigma,\bb_\sigma}$. Since $\sigma$ is a cone, it is easy to
observe that
$$\pi_1(X_{\sigma,\bb_\sigma})= \pi_2(X_{\sigma,\bb_\sigma})=0, \; \pi_1(X_{\sigma,\bb_\sigma}, (\C^*)^n)=0.$$
Thus in this case,
$$H_1(X_{\sigma,\bb_\sigma})=H_2(X_{\sigma, \bb_\sigma}) =0, \; H_1(X_{\sigma,\bb_\sigma}, (\C^*)^n)=0.$$
From the homotopy exact sequence and Hurewicz theorem, we have
$$\pi_2(X_{\sigma,\bb_\sigma}, (\C^*)^n) \cong \pi_1((\C^*)^n) \cong \Z^n
\cong H_2(X_{\sigma,\bb_\sigma}, (\C^*)^n;\Z) \cong H_1(L;\Z).$$

In fact, we can find generators of the above explicitly. Elements
of $\Z^n$ above correspond to points of the lattice $N$. Finding
generators of $\Z^n$ corresponds to finding that of the lattice
$N$.

In the previous sections, we have found holomorphic discs
corresponding to the stacky vectors $\bb_\sigma=(\bb_{i_1},\cdots,
\bb_{i_n})$. We denote the homology class of a disc corresponding
to $\bb_i$ by $\beta_i$. Also, we have found holomorphic orbidiscs
corresponding to elements of $Box'$, and we denote the homology
class of a disc corresponding to $\nu \in Box'$ by $\beta_\nu$.

The lattice $N$ is generated by stacky vectors in $\bb_\sigma$
together with $Box_{\bb_\sigma}$.
 Thus  $H_2(X_{\sigma,\bb_\sigma}, L:\Z)$ is generated by $\beta_i$'s and $\beta_\nu$'s.
These correspond to the basic discs explained earlier.

In the general case of toric orbifolds,  by applying the
Mayer-Vietoris sequence of a pair, we obtain the following result.
\begin{lemma}\label{lem:homgen}
For toric a orbifold $X$, and a Lagrangian torus fiber $L$,
 $H_2(X,L; \Z)$ is generated by  the homology classes of basic discs,
 $\beta_i$ for $i=1,\cdots,m$ together with $\beta_\nu$ for $\nu \in Box'$.
\end{lemma}
We have the following short exact sequence
$$0 \to \pi_2(X_{\Sigma,\bb}) \to \pi_2(X_{\Sigma,\bb}, L) \to \pi_1(L) \to 0,$$
and from the fact that the map $H_2(L) \to H_2(X)$ is trivial, the
five lemma gives
$$\pi_2(X,L) \cong H_2(X,L;\Z).$$
Thus,  $ \pi_2(X_{\Sigma,\bb}, L)$
 is generated by homotopy classes of smooth and
orbifold holomorphic discs (or that of basic discs) and elements
of $\pi_2(X_{\Sigma,\bb})$ correspond to homotopy classes of
orbi-spheres in toric orbifolds.

The following lemma (based on ideas in page 48 of \cite{Ful})
shows that for an $n$-dimensional stacky cone, we can choose
exactly $n$ holomorphic (orbi)discs which generate
$H_2(X_{\sigma,\bb_\sigma},L;\Z)$.

\begin{lemma}
Let $\sigma$ be any  $n$-dimensional simplicial rational
polyhedral cone in $\R^n$. Then we can find an integral basis of
the lattice $N=\Z^n$, all of whose vectors lie in $\sigma$.
\end{lemma}

\begin{proof}
 Let $\sigma$ be an $n$-dimensional simplicial cone with primitive integral generators
 $v_1, \ldots, v_n$ of its one dimensional faces. Let $N_\sigma$ be the
 submodule of $N$ generated by $v_1, \ldots, v_n$.
  Let $G_\sigma = N/N_\sigma$. Since $\sigma$ is simplicial, $N_\sigma$ has
  rank $n$ and $G_\sigma$ is finite. Let $ {\rm mult}(\sigma)
 = o(G_\sigma)$.

  Let $B=[v_1 \ldots v_n]$ be the matrix with the $v_i$'s as columns.
 Consider $B$ as
 a linear operator $B: N_\sigma \to N$ and $G_\sigma$ as the cokernel of $B$. Then
 from the Smith normal form of $B$ and the corresponding decomposition of the finite
 abelian group $G_\sigma$ into a direct product of cyclic groups, we conclude that $ {\rm
 mult}(\sigma) = |\det (B)|$.

If $ {\rm mult}(\sigma) = 1$ then we are done as $v_1, \ldots,
v_n$ form a basis of $N$ in this case.
 Assume $ {\rm mult}(\sigma) > 1$. Then there exists an integral vector $v \in N$
 which does not
  belong to $N_\sigma$. Therefore  $v = \sum_{i=1}^n t_i v_i$ where not
 every $t_i$ is an integer. By
 adding suitable integral multiples of the $v_i$¥s to $v$, we may assume that each
 $ t_i \in [0,1)$ and not
 every $t_i$ is zero. Without loss of generality suppose that $1,\ldots,k$
 are the values of $i$ for which $t_i \neq 0$.
 Then $v$ belongs to the relative interior of the face of $\sigma$ generated
 by $v_1, \ldots, v_k$. Suppose that $v/d$ is a primitive integral vector, where $d$
 is a positive integer.

 We subdivide the cone $ \sigma$ into  $n$-dimensional cones $\sigma_i $,
 $1\le i \le k$.
 Here $\sigma_i$ is generated by $\{ v_1, \ldots, \WH{v}_i, \ldots,
 v_n, v/d \}$. It is easy to check using determinants
  that ${\rm mult}(\sigma_i)= \frac{t_i}{d} {\rm mult}(\sigma)$.
 Therefore ${\rm mult}(\sigma_i) < {\rm mult}(\sigma)$.  Note that
 the generators of one dimensional
 faces of $\sigma_i$ belong to $\sigma \cap N$.

  Iterating the above process (if necessary) we
  obtain an $n$-dimensional cone $\tau$ having multiplicity one whose one
 dimensional generators
 belong to $\sigma \cap N$. These generators give the required basis of $N$.
\end{proof}

Here it is important that the basis lattice vectors lie in the
cone
 $\sigma$, since then they correspond
to holomorphic (orbi)-discs in $X_{\sigma,\bb_\sigma}$.

\subsection{Moduli spaces of smooth holomorphic discs}
In this subsection, we discuss the moduli spaces of holomorphic
discs {\em without interior orbifold marked points}.

Recall from the corollary \ref{smoothdisc} that we have a
one-to-one correspondence between  stacky vectors $\{\bb_1,
\cdots, \bb_m\}$ and smooth holomorphic discs of Maslov index two
(modulo torus $T^n$-action).

We denote by $\beta_i \in
H_2(X,L(u);\Z) \,(i=1,\cdots,m)$, the homology class of discs
corresponding to $\bb_i$. Note that  we have $\mu(\beta_i)=2$, and the
intersection number of $\beta_i$ with $j$-th toric divisor is $1$  if $i =j$ and
$0$ otherwise. (Here the intersection number is measured in the uniformizing
chart or $\C^m$).

For each $\beta \in H_2(X,L;\Z)$, consider the moduli space $\CM^{main}_{k+1,0}(L(u),\beta)$ of stable maps from bordered genus zero Riemann surfaces with $k+1$
boundary marked points 
of homotopy class $\beta$. We denote by $\CM^{main,reg}_{k+1,0}(L(u),\beta)$, its subset
whose domain is a single disc.
For the orientation of the moduli spaces, we use the spin
structure of $L(u)$ which is induced from the torus $T^n$-action,
it is the same as the case of toric manifolds (see \cite{C}, \cite{CO} and \cite{FOOO} for more details).

In the following proposition, we do not consider interior marked
points, hence only  holomorphic discs without orbifold
 marked points are allowed, and its Maslov index $\mu$ can be defined as usual.
We also emphasize that the moduli spaces discussed here are not
perturbed.

\begin{prop}\label{prop:discmoduli}
Let $\beta$ be a homology class in  $ H_2(X,L(u),\Z)$.
\begin{enumerate}
\item The moduli space
$\CMBR$ is Fredholm regular for any $\beta$.
Moreover, evaluation map
\begin{equation}\label{ev:reg}
ev_0 :\CMBR \to L(u)
\end{equation} is submersion.
\item For $\beta$ with $\mu(\beta)<0$, or $\mu(\beta)=0, \beta \neq 0$, 
$\CMBR$ is empty.
\item $\CMBR$ is empty if
$\mu(\beta)=2$, and  $\beta \neq  \beta_1,\cdots,\beta_m$

\item If $\CM_{k+1,0}^{main}(L(u),\beta)$ is non-empty, then there
are $k_i \in \Z_{\geq 0}$ and $\alpha_j \in H_2(X;\Z)$ such
that 
\begin{equation}
\beta = \sum_i k_i \beta_i + \sum_j \alpha_j
\end{equation}
 and $\alpha_j$ is a homology class of a holomorphic sphere. 
 If $\beta \neq 0$,
at least one $k_i$ is non-zero.
 \item For each $i =1,\cdots,m$, we have
\begin{equation}
\CM^{main,reg}_{1,0}(L(u),\beta_i) =\CM^{main}_{1,0}(L(u),\beta_i).
\end{equation}
Hence, the moduli space $\CM^{main}_{1,0}(L(u),\beta_i)$ is Fredholm
regular and the evaluation map $ev_0$ becomes diffeomorphism preserving
orientation.
\end{enumerate}
\end{prop}
\begin{proof}
The proof follows from the classification theorem in  section
\ref{sec:basic}, similarly as theorem 11.1 of \cite{FOOO2} follows
from  the classification theorem of \cite{CO}. For (1), Fredholm
regularity for holomorphic discs was proved in the proposition
\ref{fred:smooth}.
Evaluation map $ev_0$ becomes a submersion since  $T^n$ acts 
on $L(u)$ and  moduli spaces such that  $ev_0$ becomes
$T^n$-equivariant map. Since $L(u)$ is a $T^n$-orbit, it is a submersion.

For (2), if  $\CM_{1,0}^{main,reg}(L(u),\beta)$ is
non-empty, then since $ev$ is a submersion, we have 
$$\dim  \CM^{main,reg}_{1,0}(L(u),\beta) = n+ \mu(\beta) -2 \geq n.$$
for $\beta \neq 0$. This implies that $\mu(\beta) \geq 2$.

(3) is a direct consequence of the classification theorem.

For (4), consider a map $[h] \in \CM^{main}_{k+1,0}(L(u),\beta)$. If the domain
of $h$ is a single disc, then the statement follows from the classification theorem
in which case $\alpha_j=0$. In general, the domain of $h$ is decomposed into irreducible components,
which are discs or sphere components. As holomorphic discs are
already classified, the claim follows.

For (5), the first statement can be proved as in \cite{FOOO2}. Let
$[h] \in \CM^{main}_{1,0}(L(u),\beta_{i_0})$. By (4), we can write
\begin{equation}\label{eq:split}
\beta_{i_0} =\sum_i k_i \beta_i + \sum_j \alpha_j, \;\; \partial
\beta_{i_0} = \sum_i k_i \partial \beta_i.
\end{equation}
We need to show that there exists no sphere bubble $\alpha_j$ and
$k_i =0$ if $i \neq i_0$ and $k_{i_0}=1$. Since the symplectic
area $\alpha_j \cap \omega >0$, it follows that $$\beta_{i_0} \cap
\omega \geq \sum_i k_i \beta_i \cap \omega.$$ It suffices to show
that
$$\beta_{i_0} \cap \omega \leq \sum_i  k_i \beta_i \cap \omega,$$
and that equality holds only if $k_i =0$ if $i \neq i_0$ and
$k_{i_0}=1$.

From \eqref{eq:lj} and from the second equation of
\eqref{eq:split}, we have
$$\ell_{i_0}(u) = \sum_{i=1}^m k_i \ell_i(u) + c$$
for some constant $c$.  This is because $\partial (\beta_{i_0} -
\sum_i k_i \beta_i)=0$, and hence its symplectic area $2\pi c$ is
independent of $u$.

By evaluating at $u \in \partial_{i_0}P$, we have $c \le 0$, since
$\ell_i \geq 0$ on $P$. But since $\ell_i(u) = \beta_i \cap
\omega$, this implies the desired inequality. Let us assume that
equality  $\ell_{i_0} = \sum_i k_i \ell_i $ holds. If there exists
$i \neq j$ with $k_i, k_j >0$, then since $u' \in P$ satisfies $u'
\in \partial_j P$ if  $\ell_j(u') =0$, the above equality implies
that $\partial_{i_0}P \subset \partial_i P \cap \partial_j P$,
which is a contradiction since $P_{i_0}$ is codimension one.

The second statement of (5) follows from the torus action, and the
orientation analysis of \cite{C} as in the case of smooth toric
manifolds. But there is a little subtlety, which is different from
the manifold case, which we now explain.

Given a smooth holomorphic disc $w:(D^2,\partial D^2) \to
(\bX,L)$, with marked point $z_0 \in \partial D^2$, the
equivalence relation (Definition \ref{def:eq}) implies that if an
automorphism of the disc $\rho:D^2 \to D^2$ satisfies $w \circ
\rho = w$, then the holomorphic disc $((D^2,z_0), w)$ is
identified with $((D^2, \rho(z_0)), w)$.

We illustrate this phenomenon by an example, which explains what
happens for a basic smooth disc. Consider a map $w:(D^2, \partial
D^2) \to \C^m/G$ given by $z \mapsto (z,1,\cdots,1)$, where $G$ is
a finite group $G=\Z/k\Z$ acting by rotation on the first
coordinate of $\C^m$ (so that the image of $w$ is invariant under
$G$-action). Denote by $\rho$ the multiplication of $k$-th root of
unity on $D^2$. Then, clearly, $w \circ \rho = w$ as $w$ is a map
to the quotient space. Hence, the marked point $z_0$ and
$\rho(z_0)$ is identified.

Hence in the moduli space of smooth holomorphic discs containing
the above map $w$, we may regard that the marked point $z_0$ moves
only along the arc from 1 to $e^{2\pi i /k}$ of $\partial D^2$.
The smooth disc wraps around the orbifold point with multiplicity
$k$ in the above example, but due to the identification of the
boundary marked point as above, the evaluation image of $ev_0$ of
the moduli space of discs only covers the boundary once. The rest
is the same of \cite{FOOO2} and we leave the details to the
reader.
\end{proof}

\subsection{Moduli space of holomorphic orbi-discs}
In this subsection, we allow interior marked points, and in
particular, interior orbifold marked points. Let $\beta \in
H_2(X,L;\Z)$ and let $\CM^{main}_{k+1,l}(L(u),\beta,\bx)$ be the
moduli space of good representable stable maps from bordered
orbifold Riemann surfaces of genus zero with $k+1$ boundary marked
points,
 $l$ interior (orbifold) marked points  in the homology class $\beta$ of type $\bx$ where $\bx= (\bX_{(g_{1})},\cdots, \bX_{(g_{l})})$.
   We denote
$$\CM^{main}_{k+1,l}(L(u),\beta) = \bigsqcup_{\bx} \CM^{main}_{k+1,l}(L(u),\beta,\bx).$$
The problem of orientation of the moduli spaces is similar to that
of the smooth discs, and we omit the details.

In corollary \ref{cor:orbidisc}, we have found the one-to-one
correspondence between an element of $\nu \in Box'$, and
holomorphic orbi-discs with one orbifold marked point that
satisfies $\mu^{de}=0$ (modulo torus $T^n$-action). We have
denoted the homotopy class of such orbi-disc by $\beta_v \in
H_2(X,L;\Z)$. In particular, such $\nu \in Box'$ can be written as
$\nu = c_{i_1} \bb_{i_1} + \ldots + c_{i_n} \bb_{i_n} \in N $ with
$0 \le c_{i_j} < 1$.
% and if $j \notin \{i_1,\cdots, i_n\}, c_j=0.$
Then, it is easy to see that $\beta_\nu$ satisfies the following:
$$\partial \beta_\nu = \nu \in N \cong \Z^n, \; \mu^{de}(\beta_\nu, \bX_{\nu})=0,\;\;
\beta_\nu \cap [\pi^{-1}(\partial P_j)] =  c_{j}.$$

\begin{prop}\label{prop:orbimoduli}
\begin{enumerate}
\item  Suppose
 $\mu^{de}(\beta,\bx) < 0$. Then, $\CM^{main,reg}_{k+1,l}(L(u),\beta,\bx)$ is empty.
 
\item If $\mu^{de}(\beta, \bx) =0$, and if $\beta \neq \beta_\nu$
for any $\nu \in Box$, then
$\CM^{main,reg}_{k+1,1}(L(u),\beta,\bx)$ is empty.

\item The moduli space
$\CM^{main,reg}_{k+1,1}(L(u),\beta)$ is Fredholm regular for any $\beta$.
Moreover, evaluation map $ev_0:\CM_{k+1,1}^{main,reg}(L(u),\beta) \to L(u)$ is
submersion.

\item If $\CM^{main}_{k+1,l}(L(u),\beta)$ is non-empty, then there
are
  $k_\nu, k_i \in \N$,  $\alpha_j \in H_2(X;\Z)$ such
 that $$\beta = \sum_{\nu \in Box'} k_\nu \beta_{\nu} + \sum_i k_i \beta_i + \sum_j \alpha_j$$
and $\alpha_j$ is realized by a holomorphic orbi-sphere, and at
least one   $k_\nu$ or $k_i$ is non-zero.

If $\CM^{main}_{1,1}(L(u),\beta)$ is not empty and if $\partial
\beta \notin N_{\bb} := \Z \langle \bb_1, \ldots,\bb_m \rangle$,
then there exists $\nu \in Box'$ such that
$$\beta = \beta_\nu + \sum_i k_i \beta_i + \sum_j \alpha_j.$$

\item For $\nu \in Box'$, we have
$$\CM^{main,reg}_{1,1}(L(u),\beta_\nu) =\CM^{main}_{1,1}(L(u),\beta_\nu).$$
The moduli space $\CM^{main}_{1,1}(L(u),\beta_\nu)$ is Fredholm
regular and the evaluation map $ev_0$ becomes a
diffeomorphism preserving orientation.
\end{enumerate}
\end{prop}
\begin{proof}
For (1), this follows from the desingularized Maslov index formula for
holomorphic orbi-discs. And (2) follows from the classification results in section
\ref{sec:basic}.

For (3), Fredholm regularity is already proved. The complex
structure is invariant under $T^n$-action and $L(u)$ is
$T^n$-orbit, it follows that $T^n$ acts on the moduli
space $\CM_{k,1}^{main,reg}(L(u),\beta)$ and $ev_0$ becomes $T^n$-equivariant map.
Hence $ev_0$ is submersion.

For (4), the first statement follows from the structure of the
stable map, and for the second statement, consider a map 
$h \in \CM_{1,1}(L(u),\beta)$. If the domain of $h$ is a single (orbi)-disc,
then the theorem follows from the classification theorem, in which case $\alpha_j=0$.
Otherwise,  domain of $h$ has several irreducible components, which are (orbi)-discs and
(orbi)-spheres.  Since $\partial \beta \notin
N_{\bb}$, one of the disc component has to be a holomorphic
orbi-disc, and as we allow only one interior marked point, there
cannot be any other orbifold disc. Then the claim follows from the
classification theorem.

For (5), let $h \in \CM_1^{main}(L(u),\beta_{\nu})$. By (4), we
can write
\begin{equation}\label{eq:split2}
\beta_{\nu} = \beta_{\nu'} +\sum_i k_i \beta_i + \sum_j \alpha_j
,\end{equation} for some $\nu' \in Box$. By considering their
boundaries, we have
$$\partial \beta_{\nu} = \partial \beta_{\nu'} +  \sum_i k_i  \partial \beta_i $$
or equivalently, $$\nu = \nu' + \sum k_i \bb_i.$$ By the
definition of $Box$, this implies that $\nu=\nu'$ since the
coefficients of $\nu$ as a linear combination of $\bb_i$'s should
lie in the interval $[0,1)$ and since $k_i \in \Z_{\geq 0}$.

Thus, we have $\sum_i k_i \beta_i + \sum_j \alpha_j =0$. As their
symplectic areas are positive unless trivial, hence this implies
that $k_i=0$ for all $i$, and $\alpha_j=0$ for all $j$. This
proves the first statement and the second statement follows as in
the proof of Proposition \ref{prop:discmoduli}.
\end{proof}

\section{Moduli spaces and their  Kuranishi structures}
In this section, we discuss the $T^n$-equivariant Kuranishi structures of  moduli spaces $\CM_{k+1,l}(L(u),\beta)$ of holomorphic (orbi)-discs.
Recall that $T^n$-equivariant Kuranishi
structure of the moduli spaces  in smooth toric manifolds has been constructed  in \cite{FOOO2}.
And also recall that Kuranishi structure of the moduli space of  stable maps from orbi-curves(without boundary)  has been established in the work of Chen and Ruan \cite{CR}. We also
recall that the Fredholm setup and gluing analysis for $J$-holomorphic discs has been carefully discussed in the foundational work of \cite{FOOO}, and the case with bulk insertion  is discussed in \cite{FOOO3}.

For our case of toric orbifolds, the moduli spaces  $\CM_{k+1,l}(L(u),\beta)$ of holomorphic (orbi)-discs also have $T^n$-equivariant Kuranishi structure, as most of the construction of
\cite{FOOO} and \cite{FOOO2} can be easily extended to these cases in a straightforward way by combining the work of
Chen and Ruan \cite{CR} regarding interior orbifold marked points.
But we give brief explanations on some of the issues for readers who are not familiar with them.

\subsection{Fredholm index}\label{sec:index}
Let us explain the virtual dimension of the moduli spaces. First, we recall the case of closed $J$-holomorphic orbi-curves from Chen-Ruan \cite{CR}.
Let $\Sigma$ be a closed Riemann surface, with complex vector bundle $E$ on it. The index of
the first order elliptic operator $\dbar$ is given by Riemann-Roch formula
$$index(\dbar) = 2c_1(E)[\Sigma] + 2n(1-g_{\Sigma}),$$
where $2n$ is the dimension of $E$, and $g_{\Sigma}$ the genus of $\Sigma$.

Let $\bsg$ is a closed orbi-curve with orbifold marked points $z_1,\cdots, z_k$ (with underlying Riemann surface $\Sigma$, and $E$ is orbifold complex vector bundle, with degree shifting number $\iota_i$ at $i$-th marked point.
Then  the index
of $\dbar$ is given as (Lemma 3.2.4 of \cite{CR})
$$index(\dbar) = 2c_1(|E|)[\Sigma] + 2n(1-g_{\Sigma}) = 2c_1(E)[\Sigma]
+ 2n(1-g_{\Sigma}) - \sum_{i=1}^k 2\iota_i.$$
Here, $|E|$ is the desingularization of $E$ explained in section \ref{maslovorbicompute},
and the second identity is from Proposition 4.1.4 of \cite{CR2} which follows from the curvature computation in Chern-Weil theory.  The desingularized bundle $|E|$
can be used for index computations, as local holomorphic sections of $E$ and $|E|$ can be identified
(see Proposition 4.2.2 of \cite{CR2}),
and hence, they have the same indices.  Note that the desingularized orbi-bundle over orbi-curve has trivial fiber-wise action near orbifold point. Hence $|E|$ gives an honest vector bundle over $\Sigma$ and we can apply the usual index theorem, and obtain the above equality.

The moduli space of stable maps from genus g orbi-curves with
$k$ marked points mapping to $\bx$, of class $A \in H_2(X)$, is denoted as   $\CM_{g,k}(X,J,A,\bx)$. Applying the above index formula to the pull back orbi-bundle, the dimension of $\CM_{g,k}(X,J,A,\bx)$ is given as
(Lemma 3.2.4 of \cite{CR})
$$ 2c_1(TX)[A]
+ 2n(1-g_{\Sigma})-6 - \sum_{i=1}^k 2\iota(x(i)).$$

Exactly the same argument applies to our cases.
Let $\Sigma$ be a bordered Riemann surface, with complex vector bundle $E \to \Sigma$,
a Lagrangian subbundle $\CL \to \partial \Sigma$.  Recall that (see \cite{KL} for example), the index of
$\dbar$ is given by Riemann-Roch formula
$$index(\dbar) = \mu(E,\CL) + n \cdot e(\Sigma),$$
where $2n$ is the dimension of $E$, and $e(\Sigma)$ the euler characteristic of $\Sigma$.

Let $\Sigma$ is a bordered orbi-curve with interior orbifold marked points $z_1,\cdots, z_k$, and $E$ is orbifold complex vector bundle, with degree shifting number $\iota_i$ at $i$-th marked point, and with a Lagrangian subbundle $\CL \to \partial \Sigma$. Then we have
$$index(\dbar) = \mu(|E|,\CL) + n \cdot e(\Sigma)= \mu(E,\CL) + n \cdot e(\Sigma) - \sum_{i=1}^k 2\iota_i.$$
The second equality follows from Proposition \ref{maslow=de}(Proposition  6.10 of \cite{CS}).

Applying the above index formula to the pull back orbi-bundle of holomorphic orbi-discs (note that $e(\Sigma)=1$),
we obtain the virtual dimension of  the moduli space of bordered stable maps $\CM_{k,l}(L,\beta,\bx)$,
which proves the lemma \ref{lem:dim}:
 $$n+ \mu^{de}(\beta,\bx) + k + 2l -3 \; = \; n + \mu_{CW}(\beta) + k + 2l -3 -2\iota(\bx).$$
Here we have subtracted $ Aut(D^2) = 3$ as we consider the moduli space, and $k$, $2l$ accounts for
 the freedom of boundary, and interior marked points.

\subsection{Construction of Kuranishi structures}
We recall a definition of a Kuranishi neighborhood (chart) $(V,E,\Gamma,\psi,s)$ of a moduli
space $\CM$:
 $V$ is a smooth manifold
and $E$ is a vector bundle  over $V$,  with a group $\Gamma$ acting
on $V$ and $E$ in a compatible way, and $s:V\to E$ is  $\Gamma$-equivariant section such that $\psi: s^{-1}(0)/\Gamma \to \CM$ is homeomorphic to an open set of the moduli space $\CM$.
We refer readers to \cite{FOOO} for the definition of compatibilities between Kuranishi charts, and
for more details.

The general scheme to construct a Kuranishi structure of a moduli
space is as follows: first, one constructs a Kuranishi
neighborhood of each point in the interior of the moduli space.
The proper Fredholm setting for this construction and the
application of the implicit function theorem to it is by now
standard.  Then, one also construct a Kuranishi neighborhood of
each point in the boundary of the moduli space or for the stable
map. For this, Taubes type gluing argument is needed, and the
gluing construction for interior node \cite{FOn} (and orbifold
interior node \cite{CR}), or boundary node \cite{FOOO},  has been
established. Once,  local Kuranishi neighborhoods are constructed,
there is a standard procedure to construct the global Kuranishi
charts, which we refer readers to \cite{FOn} or \cite{FOOO}.

We explain the construction of  a local Kuranishi neighborhood of
\begin{equation}\label{eq:element}
((\bsg, \vec{z}, \vec{z}^+), w,\xi) \in \CM_{k+1,l}(L(u),\beta,\bx).
\end{equation}
First, we consider the case that the domain $\bsg = \bD$ is an orbi-disc $\bD$,  in which
case the element \eqref{eq:element} lies in the interior of the moduli space $\CM_{k+1,l}(L(u),\beta,\bx)$. Then, the linearized $\dbar$-operator
at $w$ is given as
\[
D_w\dbar:W^{1,p}(\bD,w^*T\bX, L) \to L^p(\bD, w^*T\bX \otimes \Lambda^{0,1})
\]
and obstruction space $E$ can be chosen so that
elements of $E$ are smooth, and supported away from marked points and from $\partial \bD$ and also that
$$\textrm{Image}(D_w\dbar) + E =  L^p(\bD, w^*T\bX \otimes \Lambda^{0,1}).$$
Then the kernel of
$D_w\dbar:W^{1,p}(\bD,w^*T\bX, L) \to L^p(\bD, w^*T\bX \otimes \Lambda^{0,1})/E$
is denoted as $V^{map}$, and the section $s = D_w\dbar$. One takes $V = V^{map} \times V^{dom}$
where $V^{dom}$ parametrizes the deformation of the domain $(\bD, \vec{z}, \vec{z}^+)$. In this case the automorphism $\Gamma$
is trivial since the boundary of the disc maps to $L$, only intersects toric divisors at finitely many points. Non-trivial $\Gamma$ appears if $\bsg$ has sphere component.

In fact, to consider  $D_w\dbar$ properly, instead of $\bD$, one identifies $\bD$ with
a bordered Riemann surface $\Sigma'$ of
genus 0, with strip-like end (near boundary marked points) and with cylindrical end (near interior marked points).
Then over this domain $\Sigma'$, we have a Fredholm problem by considering $D_w\dbar$ problem with suitable exponential weights as in \cite{FOn}, \cite{CR} (for cylindrical end) and \cite{FOOO} (for strip-like end). In the case of orbifold marked points, we follow Chen-Ruan's construction that
such Riemann surface $\Sigma'$ still has  ``orbifold'' data near orbifold marked points.
Namely, consider an interior marked point $z_1 \in \Sigma$ which has $\Z/m\Z$ singularity. Let $\rho$ be the generator of
$\Z/m\Z$. Suppose the equivariancy data $\xi$ of the map $w$ in
\eqref{eq:element} gives a homomorphism $\phi:\Z/m\Z \to G_{w(z_1)}$ where $G_{w(z_1)}$ is a local group of $z_1$.
Then a cylindrical end (for $z_1$) is considered to have a covering cylinder with $\Z/m\Z$ action,
and the pull-back bundle over it is considered as an orbifold bundle on it.
Hence the change is only for analytical purposes and orbifold data is not lost during the process.
Then in setting up the Fredholm problem, one adds the description of the points $p_i$ where these infinite ends are exponentially converging to.
For orbifold marked point $z_1$ as above, Chen-Ruan required that the end of holomorphic cylinder limits to  a point
$p_i \in \chi_{\phi(\rho)}$ in the twisted sector. We refer readers to  Lemma 3.2.3 of \cite{CR} for more details. The construction of $\psi:s^{-1}(0) \to \CM_{k+1,l}(L(u),\beta,\bx)$ involves
implicit function theorem following \cite{FOOO} (and \cite{CR} regarding interior orbifold marked points) and  is standard and omitted.

Now, we consider the construction when
$\sigma:=((\bsg, \vec{z}, \vec{z}^+), w,\xi)$ is in the boundary(or corner) of the
moduli space $\CM_{k+1,l}(L(u),\beta,\bx)$. We first write the domain $\Sigma = \cup_\nu \pi_\nu(\Sigma_\nu)$ as
the union of irreducible components, which are (orbi)-discs and (orbi)-spheres.
We  recall an important ingredient  of Chen-Ruan's construction of Kuranishi structure when
the image of some irreducible components of a domain entirely maps into the (orbifold) singular locus of $\bX$.

 Note that if $\Sigma_\nu$ is a disc, it cannot map entirely into the singular locus of $\bX$, due to
 Lagrangian boundary condition. So, let us suppose that  component $\Sigma_\nu$ is an (orbifold) sphere which  maps  entirely into the singular locus of $\bX$ via $w$.
 We denote by  $G_\nu$ the group whose elements are  stabilizers of all but finitely many points of the image of $\Sigma_\nu$. Namely, after deleting finitely many points $\vec{z}\,' \supset \vec{z}^+\cap \Sigma_\nu$ of $\Sigma_\nu$,  for any
 points $p \in w(\Sigma_\nu \setminus \vec{z}\,')$, local group $G_p$ is isomorphic to a fixed group
 $G_\nu$ (from the properties of orbifold $J$-holomorphic maps). Then, define
 $$G_\sigma = \{(g_\nu) \in \prod_\nu G_\nu \mid g_\nu(z_\nu) = g_\omega(z_\omega)
 \;\textrm{if}\;\; \pi_\nu(z_\nu) = \pi_\omega(z_\omega) \}.$$

 This $G_\sigma$ will be added to the $\Gamma$ of the Kuranishi structure in the following way.
 The automorphism group $Aut(\sigma)$ of $\sigma$ acts on $G_\sigma$ by pull-backs. Hence
 we get a short exact sequence
 $$1 \to G_\sigma \to \Gamma_\sigma \to Aut(\sigma) \to 1.$$
$\Gamma_\sigma$ is the finite group $\Gamma$ of the local Kuranishi neighborhood and the action of
$\Gamma_\sigma$
on $V$ and $E$ is defined from that of $Aut(\sigma)$ by setting $G_\sigma$ to act
trivially on them. The rest of the construction is carried out in an $\Gamma_\sigma$-equivariant way.

  We remark that the  general discussion in Chen-Ruan \cite{CR} is more
  complicated  as the groups may not be abelian. In the general case of \cite{CR}, among $G_\nu$, one  should take the elements which form a global section on $\Sigma_\nu \setminus \vec{z}\,'$, so that $G_\nu$ do not change the local group at $w(z^+_i)$'s by conjugation (Then such elements of $G_\nu$  commute with local groups at $w(z^+)$). In our case, toric orbifolds, the local groups are abelian, and we can take $G_\nu$ as above.

So, in our case of $\sigma \in \CM_{k+1,l}(L(u),\beta,\bx) \setminus \CM_{k+1,l}^{reg}(L(u),\beta,\bx)$, we need Taubes' type gluing construction.  Namely, one first replaces $\Sigma_\nu$ (equipped  with marked points) with the associated Riemann surface with cyclindrial, and strip like ends, and apply the construction of the above for each $\Sigma_\nu$. Then, as in section 7.1.3 of \cite{FOOO}, one can apply gluing construction (of constructing
approximate solution and applying Newton-type iteration arguments to find actual holomorphic curves), where the gluing near boundary nodal point  is carried out in \cite{FOOO} and gluing near interior
nodal point is carried out in \cite{FOn} (and the orbifold nodal points in \cite{CR}). We omit the
details, and refer readers to the above references.
We can construct  global Kuranish structure from
 local Kuranishi charts as explained in section 7.1.4 of \cite{FOOO}.

\subsection{$T^n$-equivariant perturbations}
We briefly recall  $T^n$-equivariant Kuranishi structure of the moduli spaces in \cite{FOOO2}, and show that the moduli space of stable orbi-discs  $\CM_{k+1,l}(L(u),\beta)$ in this paper, also
has such a structure analogously.

Consider the following family $\frak{A}$ of compatible Kuranishi charts of the moduli space $\CM$: $$\{(V_\alpha,E_\alpha,\Gamma_\alpha,\psi_\alpha,s_\alpha)| \alpha \in \frak{A}\}.$$
Here $\pi:E_\alpha \to V_\alpha$ is a vector bundle with equivariant $\Gamma_\alpha$-action,
and $s_\alpha$ is a $\Gamma_\alpha$ equivariant section of $E_\alpha$ and $\psi_\alpha$
is a homeomorphism $\psi:s^{-1}(0)/\Gamma_\alpha \to \CM$.

Recall from \cite{FOOO2} Appendix 2, Definition 15.4 that such a Kuranishi structure is said to be
$T^n$-equivariant in the strong sense, if
\begin{enumerate}
\item $V_\alpha$ has $T^n$-action and it commutes with $\Gamma_\alpha$-action.
\item $E_\alpha$ is $T^n$-equivariant vector bundle.
\item The maps $s_\alpha$, $\psi_\alpha$ are $T^n$-equivariant.
\item Coordinate change maps for embeddings of Kuranishi charts are $T^n$-equivariant.
\end{enumerate}

Recall that a strongly continuous smooth map $ev:M \to L$ is a family of $\Gamma_\alpha$-invariant smooth maps $ev_\alpha: V_\alpha \to L$ which induces
 $ev_\alpha: V_\alpha/\Gamma_\alpha \to L$  and compatible with coordinate changes.
$ev$ is said to be  weakly submersive if each  $ev_\alpha$ is a submersion.

\begin{prop}[c.f. Prop. 15.7 \cite{FOOO2}]\label{prop:tnequi}
The moduli space $\CM_{k+1,l}(L(u),\beta)$ has a $T^n$-equivariant
Kuranishi structure such that $ev_0: \CM_{k+1,l}(L(u),\beta) \to L$ is $T^n$-equivariant, strongly continuous, weakly submersive map
\end{prop}
\begin{proof}
The same line of proof as in that of Prop. 15.7 \cite{FOOO2} can be used to prove the existence of $T^n$-equivariant Kuranishi structure in our case too:
The standard complex structure $J$ of $X$ is $T^n$-invariant, and Lagrangian submanifold $L(u)$
is a free $T^n$-orbit. Note that as the torus action on ambient toric variety carries over to the tangent bundles, and Cauchy-Riemann equations in a natural way, and the main new ingredient is how to choose an obstruction bundle in a $T^n$-equivariant way.

We have a free $T^n$ action on the Kuranishi neighborhood  since the $T^n$ action on the Lagrangian submanifold $L(u)$ is free and also that evaluation maps $ev$ are $T^n$ equivariant as explained in \cite{FOOO2}. We can take a  multivalued perturbation of the Kuranishi structure that is $T^n$ equivariant.
Such a multisection which is also transversal to 0 is constructed by  taking the quotient of Kuranishi neighborhood, obstruction bundles and so on by $T^n$ action  to obtain a space with Kuranishi structure. Then  take a transversal multisection of the quotient Kuranishi structure and lift it to a multisection of the Kuranishi neighborhood. The evaluation map becomes submersion from the $T^n$-equivariance.
(The existence of $T^n$-action simplifies the general construction of \cite{FOOO} because the fiber products appearing in the inductive construction are automatically transversal).
\end{proof}

Now, we focus attention to the moduli space of holomorphic discs {\em without (orbifold) interior
marked points}. (The case of orbi-discs will be considered in section \ref{sec:bulkmoduli}).
Consider the following map which forgets the $(1,\cdots, k)$-th marked points
 $$forget_0:\CM^{main}_{k+1,0}(L(u),\beta) \to \CM^{main}_{1,0}(L(u),\beta).$$
 As in \cite{FOOO2}
we can construct our Kuranishi structure so that it is compatible with $forget_0$.

\begin{lemma}[c.f. Lemma 11.2 \cite{FOOO2}]\label{forgetlemma}
For each given $E>0$, we can take a system of multisections $\fs_{\beta,k+1}$ on
$\CM_{k+1,0}^{main}(L(u),\beta)$ for $\omega(\beta) < E$ satisfying the following
properties:
\begin{enumerate}
\item They are transversal at zero section, and invariant under $T^n$-action.
\item The multisection $\fs_{\beta,k+1}$ is obtained as  the pull-back of the multisection $\fs_{\beta,1}$ by
the forgetful map.
\item The restriction of the multisection $\fs_{\beta,1}$ to the boundary of $\CM^{main}_{1,0}(L(u),\beta)$ given as the fiber product of the multisection $\fs_{\beta',k'}$ from the following
$$ \partial \CM_{1,0}^{main}(L(u),\beta) = \bigcup_{\beta_1+\beta_2=\beta}
\CM_{1,0}^{main}(L(u),\beta_1) \,_{ev_0} \times_{ev_1} \CM_{2,0}^{main}(L(u),\beta_2).$$
\item  $\CM^{main}_{1,0}(L(u),\beta_i)$ for $i=1,\cdots,m$ are not perturbed
\end{enumerate}
\end{lemma}
The proof of the lemma is the same as \cite{FOOO} and is omitted.
We obtain the following corollary from the dimension arguments:
\begin{corollary}\label{corollary:mu2}
The moduli space $\CM_{1,0}^{main}(L(u),\beta)^{\fs_\beta}$ is empty
if the Maslov index $\mu(\beta) <0$ or $\beta \neq 0$ and $\mu(\beta)=0$. 
\end{corollary}

These $T^n$-equivariant perturbations define the following open Gromov-Witten invariants
for toric orbifolds as in  Lemma 11.7 of \cite{FOOO2}. This is because the virtual fundamental chain of 
$\CM^{main}_{1,0}(L(u),\beta)$ is now a cycle due to corollary \ref{corollary:mu2}.
A homology class
$c_\beta [L(u)] \in H_n(L(u);\Q)$ can be defined by the pushforward
\begin{equation}
c_\beta [L(u)] = ev_*([\CM_{1,0}^{main}(L(u),\beta)^{\fs_\beta}]).
\end{equation}

\begin{lemma}[Lemma 11.7 \cite{FOOO2}]\label{invari}
The number $c_\beta$ is well-defined,  independent of the choice 
$s_{\beta,k+1}$ in  Lemma \ref{forgetlemma}.
\end{lemma}

From the classification results (Prop. \ref{prop:discmoduli}), we have $c_{\beta_i}=1$ for $i=1,\cdots,m$,
where the sign can be computed from \cite{C}. If $\bX$ is Fano, then we also have $c_\beta =0$ for $\beta \neq \beta_i$.

\section{Filtered $\AI$-algebra and its potential function}
\subsection{Filtered $\AI$-algebras and its deformation theory}
We provide a quick summary of the deformation and obstruction theory of \cite{FOOO} just to set the
notations. We refer readers to \cite{FOOO}, \cite{FOOO2} for full details.

For a graded $R$-module $C$,  its suspension $C[1]$ is defined as
$C[1]^k = C^{k+1}$. For $x \in C$, we denote by $deg(x)$ and $deg'(x)$ denote
its original, and shifted degree of $x$.
Bar complex $B(C[1])$, which is a graded coalgebra,  is defined as  $B(C[1]) =
\bigoplus^\infty_{k=0} B_k(C[1])$ with
\begin{equation}
B_k(C[1]) = \underbrace{ C[1]\otimes \cdots \otimes C[1]}_{k}. 
\end{equation}
We have $B_0(C[1]) = R$ by definition.

\begin{definition}
 $A_\infty$ algebra structure on $C$ is given by a
sequence of degree one $R$-module homomorphisms
$  m_k: B_k(C[1]) \to C[1]$ for $ k = 1, 2, \cdots$
such that 
\begin{equation}\label{aineq}
\sum^{n-1}_{k=1}\sum^{k-i+1}_{i=1} (-1)^{\epsilon}  m_{n-k+1}(
x_1 \otimes \cdots \otimes  m_k(x_i,\cdots,x_{i+k-1})
\otimes \cdots \otimes x_n) = 0,
\end{equation}
which
are called the {\em $A_\infty$-equations}.
Here $\epsilon = \deg' x_1+\cdots+\deg' x_{i-1}$.
\end{definition}
This can be written using coderivations as follows.
The map  $m_k$ can be extended to a coderivation
$\WH{
m}_k: B(C[1]) \to B(C[1])$
by
\begin{equation}
\WH{m}_k(x_1 \otimes \cdots \otimes x_n) =
\sum^{k-i+1}_{i=1} (-1)^{\epsilon}
x_1 \otimes \cdots \otimes  m_k(x_i,\cdots,x_{i+k-1})
\otimes \cdots \otimes x_n
\end{equation}
If we set
$\WH{d} = \sum_{k=1}^\infty \WH{m}_k$, then
$\AI$-equation is equivalent to  $ \WH{d} \circ \WH{d}= 0$.

Since $m_1\circ m_1=0$, the complex $(C, m_1)$ defines the homology of $\AI$-algebra.
In a filtered case, $\AI$-algebra is similarly defined but has $ m_0: R \to C[1]$ term and
we have  $ m_1 \circ m_1 \neq 0$ in general filtered case.

\begin{definition} An element $\be \in C^0$ is called a {\it unit}
if it satisfies
\begin{enumerate}
\item $ m_{k+1}(x_1, \cdots,\be , \cdots, x_k) = 0$ for $k \geq 2$ or $k=0$.
\item $ m_2(\be , x) = (-1)^{\deg x} m_2(x,\be ) = x$ for all $x$.
\end{enumerate}
\end{definition}
If   $m_0(1)$ is a constant multiple of a unit ( i.e. $m_0(1) = c \be $ for some $c \in R$), then $m_1 \circ m_1 = 0$, Therefore, one can consider the homology
of $m_1$.

To consider filtered $\AI$-algebra,  consider $\bigoplus_{m \in \Z} C^m$,
a  free graded $\Lambda_{0,nov}$-module.
Filtration $F^\lambda C^m$ given by a submodule of elements with coefficients with  $T$-exponent $\geq \lambda$,
gives  natural energy filtration and we consider a completion with respect to this filtration
to define $C$. And consider similar completions to $B_kC$ and $BC$.
A  structure of filtered $A_\infty$ algebra on $C$ is given by a sequence of 
$\Lambda_{0,nov}$-homomorphisms $\{ m_k \}$ satisfying $\AI$-equation \eqref{aineq}
with $k \geq 0$, and additionally satisfying the following properties.
\begin{enumerate}
\item $ m_0(1) \in F^{\lambda}C^1$ with $\lambda >0$,
\item $ m_k$ respect the energy filtration,
\item $m_k$ is induced from  $\overline{m}_k : B_k \OL{C}[1] \to \OL{C}$ which
is an $R$-module homomorphism, where  $\OL{C}$ is the free $R$-module over the same basis elements as $C$.
\end{enumerate}
In this paper, we follow \cite{FOOO2} to work with $\Lambda_0$ rather than $\Lambda_{0,nov}$ by forgetting $e$  and we work with $\Z_2$ graded complex( see \eqref{de:nov1} for Novikov rings).

For $b \in F^{\lambda}C^1$ with $\lambda > 0$,
consider the following exponential
$$e^b  = 1 + b + b\otimes b + \cdots  \in BC$$
Then, deformed $\AI$-algebra $(C, \{m_k^b\})$ is defined by setting $m_k^b$ as
\begin{equation} m^k_b(x_1,\cdots, x_k) = m(e^b,x_1, e^b,x_2, e^b,x_3, \cdots,
x_k,e^b).
\end{equation}
If $m(e^b)=m_0^b$ is a multiple of unit $\be$, then $m_1^b$ defines a complex.
\begin{definition}
 An element $b \in F^{\lambda}C^1$ with $\lambda>0$ is called
a {\em weak bounding cochain} if $m(e^b)$ is a multiple of unit $\be$.
A filtered $\AI$-algebra is called weakly unobstructed if a weak bounding cochain $b$ exists.
\end{definition}
We denote by $\WH{\CM}_{weak}(L)$
the set of weak bounding cochains of $L$. The moduli space $\CM_{weak}(L)$ is then defined to be the
quotient space of $\WH{\CM}_{weak}(L)$ by suitable gauge
equivalence(see section 4.3 \cite{FOOO}).
In fact, when $C=H(L,\Lambda_0)$, one can also consider  $b$ in $F^0C[1]$ by introducing
non-unitary flat complex line bundle over the Lagrangian submanifold (see \cite{C3}, \cite{FOOO2}
for more details).

\subsection{Construction of filtered$\AI$-algebra of Lagrangian submanifold}
We construct a filtered $\AI$-algebra on $H(L(u);\Lambda_0^\R)$ using the (perturbed) moduli
space of holomorphic discs as in \cite{FOOO2}. We emphasize that we do {\em not} use orbi-discs to
construst the filtered $\AI$-algebra. Orbi-discs will be used for bulk deformations in later sections.

For a $T^n$-invariant metric on $L(u)$, a differential form $x$ on $L(u)$ becomes harmonic if and only if $x$ is $T^n$-equivariant, and we  identify $H(L(u),\R)$ with
the set of $T^n$-equivariant forms, on which we construct the $\AI$-structure.

Consider evaluation maps
\begin{equation}
ev=(ev_1,\cdots,ev_k,ev_0):\CM^{main}_{k+1,0}(L(u),\beta)^{\fs_\beta} \to L(u)^{k+1}.
\end{equation}
For $\omega_1, \cdots, \omega_k \in H(L(u),\R)$, we define
\begin{equation}\label{def:mk}
m_{k,\beta}(\omega_1, \cdots, \omega_k)= (ev_0)_{!}(ev_1,\cdots,ev_k)^*(\omega_1 \wedge \cdots
\wedge \omega_k).
\end{equation}
Here $(ev_0)_{!}$ is an integration along fiber and it is well-defined as $ev_0$ is a submersion.
(See the appendix C of \cite{FOOO2} for details on smooth correspondences).

The resulting differential form is again $T^n$-equivariant since $\fs_\beta$ and all other maps are $T^n$-equivariant. As in \cite{FOOO2} (and using  Lemma \ref{forgetlemma}),  we obtain the $\AI$-formula:

\begin{equation}
\sum_{\beta_1+\beta_2 = \beta} \sum_{k_1+k_2 = k+1} \sum_{l=1}^{k_1}
(-1)^\epsilon m_{k_1,\beta_1}(\omega_1,\cdots,m_{k_2,\beta_2}(\omega_l,\cdots,),\cdots,\omega_k)=0.
\end{equation}

Here $\epsilon = \sum_{i=1}^{l-1}(deg' \omega_i)$. We put $m_k= \sum_{\beta} T^{\omega(\beta)/2\pi} m_{k,\beta}.$
We extend the above to $\omega$ with coefficients in $\Lambda_0^\R$  multi-linearly. This defines an $\AI$-structure on $H(L(u),\Lambda_0^\R)$. The constructed filtered $\AI$-algebra
is unital with the unit $\be$ being the constant $1 \in H^0(L,\R)$, which is the Poincar\'e dual $PD([L(u)])$ of the fundamental class, and this follows from the definition \eqref{def:mk}.
Note that constructed $\AI$-algebra is already a canonical model, since  we define them on harmonic forms $H(L;\Lambda_0)$ in this case of a toric fiber $L=L(u)$.

As in \cite{FOOO2}, for $r \in H^1(L,\R)$,  the $\AI$-structure can be
explicitly computed:
\begin{lemma}[c.f. Lemma 11.8 \cite{FOOO2}]\label{fooocal}
For $r \in H^1(L(u),\Lambda_0^\R)$  and $\beta \in \pi_2(X,L)$ with $\mu(\beta)=2$,
and for $c_\beta$ defined in Lemma \ref{invari}, we have
 $$m_{k,\beta}(r,\cdots, r) = \frac{c_{\beta}}{k!}(r(\partial \beta))^k \cdot PD([L(u)]).$$
\end{lemma}
The proof is based on the fact that the intersection number of $r$ and $\partial \beta$ is  determined by the cap product $\partial \beta \cap r = r(\partial \beta)$, and it is the same as that of Lemma 11.8 of \cite{FOOO2}, and omitted.

From this computation, we have
\begin{prop}[c.f.\cite{FOOO2} Prop. 4.3]\label{unobstruct} 
We have an inclusion
\begin{equation}
H^1(L(u);\Lambda_{+}) \hookrightarrow \CM_{weak}(L(u)).
\end{equation}
Hence, toric fiber $L(u)$ is weakly unobstructed for any $u \in Int(P)$.

Moreover, one can take $b \in H^1(L(u);\Lambda_{0})$, and it is  contained in
$\mathcal M_{weak}(L(u);\Lambda_0)$.
\end{prop}
\begin{proof}
First, take $b_+ \in H^1(L(u),\Lambda_+)$. We have
$$\sum^\infty_{k=0}  m_k(b_+,\cdots,b_+) 
 =  \sum_{\beta} \sum_{k=0}^\infty \frac{c_{\beta}}{k!}(b_+(\partial \beta)
)^k T^{\omega(\beta)/2\pi} \cdot PD([L(u)]).$$
By the degree reason, the sum is over $\beta$ with $\mu(\beta) =2$.
Hence $b_+ \in \widehat \CM_{weak}(L(u))$  and the gauge equivalence
relation  is trivial on
$H^1(L(u);\Lambda_0)$ and this proves the inclusion.

One can take $b \in H^1(L(u);\Lambda_{0})$ in the definition of
weak Maurer-Cartan elements as in \cite{FOOO2} as follows:
For, $b = b_0 + b_+$, with $b_0 \in H^1(L(u),\C)$ and $b_+ \in H^1(L(u),\Lambda_+)$,
we introduce a representation $\rho: \pi_1(L) \to \C^*$ such that
$\rho(\gamma) = exp(\int_\gamma b_0)$. We define a
non-unitary flat line bundle $\CL_\rho$ on $L$ with holonomy given by $\rho$, and modify the
$\AI$-structure  by
$$
m_k^\rho = \sum_{\beta \in \pi_2(M,L)} \rho(\partial \beta) m_{k,\beta} \otimes T^{\omega(\beta)/2\pi}
$$
If the resulting $\AI$-structure $\{m_k^\rho\}$ is weakly unobstructed
with weak bounding cochain $b_+$, then  the set  of such $b$'s
are denoted by $\CM_{weak}(L(u);\Lambda_0)$, and
again called weak bounding cochains. We refer readers to \cite{C3}, \cite{FOOO2} for
more details.
\end{proof}

For $b \in \WH{\CM}_{weak}(L)$, we have $m_0^b = m(e^b)= 0$ and the $\AI$-equation tells us that $m_1^b$ is a differential.  Hence, for $b \in \WH{\CM}_{weak}(L)$, we define the Bott-Morse Floer cohomology of $L$ as
\begin{equation}\label{de:fhdefine}
HF((L;b),(L;b)) = \frac{Ker\;m_1^b}
{Im \; m_1^b},
\end{equation}
We call it {\em smooth Floer cohomology of L} to emphasize that it does not
use the data of bulk deformation by twisted sectors of toric orbifolds.

Recall that for weakly unobstructed $L$, the potential function $PO$, 
as a function from $ \WH{\CM}_{weak}(C)$ to $\Lambda_+$
is defined by the equation
\begin{equation}
m(e^b)=PO(b)\cdot PD([L]).
\end{equation}

\subsection{Smooth potential for toric orbifolds}\label{sec:smoothpotential}
Given a toric orbifold, the above construction gives filtered $\AI$-algebra for $L(u)$, 
which uses only smooth holomorphic (stable) discs. 
The potential $PO$ above, may be called
{\em smooth potential} for $L(u)$ since it does not use information
on orbi-discs. (The bulk deformed potential will be defined using orbi-discs in later sections).

 In this subsection, we discuss the properties of a smooth potential and
define leading order smooth potential, which can be explicitly computed. As in the manifold
case, if $\bX$ is not Fano, smooth potential $PO$ need to also consider stable disc
contributions, which is not readily computable.

We choose an integral basis $\TE_i \in H^1(L(u);\Z)$, which can be done by identifying  $L(u) =T^n = (S^1)^n = (\R/\Z)^n$ by free torus action on $L(u)$. (Here we may use  $dt_i$ in de Rham cohomology, where
$t_i$ is the coordinate of the $i$-th factor of $(\R/\Z)^n$).

We choose a weak bounding cochain $b$   as 
$$
b = \sum x_i \TE_i \in H^1(L(u);\Lambda_{0}).$$
Then, $PO(b)$ depends on  $(x_1,\cdots,x_n) \in (\Lambda_{0})^n$ and
$(u_1,\cdots,u_n) \in Int (P)$, and hence to emphasize its dependence on $u$,
we may write $PO(b)$ as  $PO(x;u) := PO(x_1, \cdots,x_n;u_1,
\cdots,u_n)$. But for simplicity, most of the time we omit $u$, and write $PO$ and $PO(b)$.
( $PO^u$ is used in \cite{FOOO3}).

As in \cite{FOOO2}, it is convenient to introduce $y_1,\cdots, y_n$ as follows
(also because  holonomy is defined up to $2\pi \sqrt{-1} \Z$): We define
$$
y_i = e^{x_i} = e^{x_{i,0}} \sum_{k=0}^{\infty} x_{i,+} ^k/k!,
$$
where we write $x_i = x_{i,0} + x_{i,+}$ with
$x_{i,0} \in \C$ and $x_{i,+} \in \Lambda_+$.

Consider a toric orbifold $X$ with moment polytope $P$ and stacky vectors $\vec{\bb}$.
From \eqref{eq:lj}( Lemma \ref{lem:area}),  the following affine function measures the
area of smooth discs corresponding to stacky vectors  $\bb_j=(b_{j1},\cdots,b_{jn}) \in \Z^n$ for $j=1,\cdots,m$:
$$
\ell_j(u) = \langle u, \bb_j \rangle -\lambda_j.$$

We define the leading order smooth potential function $PO_0(b)$ of toric orbifold:
\begin{equation}
PO_0(b) := \sum_{j=1}^m T^{\ell_j(u)}(y_1)^{b_{j1}}\dots (y_n)^{b_{jn}},
\end{equation}
whose the $j$-th term corresponds to stacky vector $\bb_j$ (Corollary \ref{smoothdisc}).
Remaining terms $PO(b) - PO_0(b)$ corresponds to the contributions of stable discs.

We introduce variables $z_j$ as follows: (which simplifies $PO_0(b) = z_1 + \cdots + z_m$)
\begin{equation}
z_j = T^{\ell_j(u)}(y_1)^{b_{j1}}\dots (y_n)^{b_{jn}}.
\end{equation}
\begin{theorem}[c.f. Theorem 5.2 \cite{FOOO5}]
\begin{enumerate}
\item
$PO(b)$ can be written as
\begin{equation}\label{POformula1}
PO(b)
=  \sum_{i=1}^m z_{j} + \sum_{k=1}^{N} T^{\lambda_k}P_k(z_1,\dots,z_m).
\end{equation}
for $N \in \Z_{\geq 0} \cup \{\infty\}$ and $ \lambda_k \in \R_{>0}$.
If $N=\infty$, then $\lim_{k \to \infty} \lambda_k = \infty$.
Here $P_k(z_1, \dots, z_m)$ are monomials of $z_1,\dots,z_m$ with
$\Lambda_0$ coefficient.
\item If $\bX$ is Fano then $P_k=0$.
\item The above formula \eqref{POformula1} is independent of $u$
and depends only on $\bX$.
\end{enumerate}
\end{theorem}
\begin{proof}
If $\bX$ is Fano, then, the usual dimension counting shows that
only Maslov index two discs of class $\beta_i$'s for $i=1, \cdots, m$ contribute to 
$m_{k,\beta}(b,\cdots,b)$.  Hence, to show (2), it is enough to show that each $\beta_i$ 
contribution for $\sum_{k} m_{k,\beta_i}(b,\cdots,b)$ is given as $z_i$.

 Denote  $b = \sum_{i=1}^n x_i \TE_i$ with
$b = x_0 + x_+$ as before  and consider flat line bundle $\CL$ on $L$ whose
holonomy $\rho$ along $e_i^*$ is $exp(x_{i,0})$.
Then, we have
\begin{eqnarray}\label{eq:shape2}
\sum_{k=0}^\infty m_{k}^{\rho}(x_+,\cdots,x_+) &=& \sum_{i=1}^m \sum_{k=0}^{\infty}T^{\omega( \beta_i)/2\pi} \rho(\partial \beta_i) m_{k,\beta_i}(x_+,\cdots,x_+) \\
 & = &
\sum_{i=1}^m \sum_{k=0}^\infty  e^{ \langle \bb_i, x_0 \rangle} \frac{1}{k!}( x_+(\partial \beta_i))^k
T^{\ell_i(u)} \cdot  PD([L])\\
&=& \sum_{i=1}^m e^{\langle \bb_i, x \rangle} T^{\ell_i(u)} \cdot PD([L])
\end{eqnarray}
where the third inequality follows by writing $x_+(\partial \beta_i) = <\bb_i, x_+>$. Since $y_i = e^{x_i}$, we obtain $e^{\langle \bb_i, x \rangle} = y_1^{b_{i1}}\cdots y_n^{b_{in}}$
and thus, in the Fano case, we have
$$PO(x;u)=PO_0(b)= \sum_{i=1}^m y_1^{b_{i1}}\cdots y_n^{b_{in}} T^{\ell_i(u)}. $$

Hence, to prove (1), let us assume that $\bX$ is not Fano, and find a general expression for
stable map contributions. If $\bX$ is not Fano, and for $\beta \neq \beta_j$, $\beta$ is the homotopy class of stable discs (still with $\mu(\beta) =2$) and from the
Theorem \ref{prop:discmoduli} (4), we have that
\begin{equation}
\partial \beta = \sum k_i \partial \beta_i,
\quad \beta = \sum_i k_i \beta_i + \sum_j \alpha_j.
\nonumber\end{equation}
Thus, by computing
$$
\sum_k T^{\omega(\beta)/2\pi}  m_{k,\beta}^\rho(b,\cdots,b)
$$
we note that it is a constant multiple of the expression $ T^{ (\sum_j \omega(\alpha_j)/2\pi)} \prod_{i} z_i^{k_i}$, which proves the theorem.
The proof of (3) are similar to \cite{FOOO5} and omitted.
\end{proof}

The rest of the procedure to compute smooth Floer cohomology from the smooth potential function is
analogous to \cite{FOOO2} or \cite{FOOO5} of the manifold case.
Hence, we only summarize the main results and refer readers to the above references for full details.
The following criterion reduces the computation of smooth Floer cohomology to the critical point theory
of the potential function.
\begin{theorem}[c.f. Theorem 5.5 \cite{FOOO5}]\label{compHF:thm}
Let $b = \sum x_i \TE_i$. The following are equivalent
\begin{enumerate}
\item For each of $i=1,\cdots, n$, we have
$$ \left. \frac{\partial PO}{\partial x_i} \right|_b =0.$$
\item We have an isomorphism as modules
$$HF\big( (L(u),b), (L(u),b);\Lambda_0 \big) \cong H(T^n;\Lambda_0).$$
\item
$$HF\big( (L(u),b), (L(u),b); \Lambda) \neq 0.$$
\end{enumerate}
\end{theorem}
\begin{proof}
This is obtained by taking a derivative: note that 
$\partial b/\partial x_i =\TE_i$, and hence
$$ \left. \frac{\partial PO}{\partial x_i} \right|_b = \sum^\infty_{k_1 =0} \sum^\infty_{k_2=0}
m_{k_1+k_2+1}(\underbrace{b,\cdots,b}_{k_1},\TE_i, \underbrace{b,\cdots,b}_{k_2}) =m_1^b(\TE_i).$$
This shows that (1) is equivalent to the condition $m_1^b(\TE_i)=0$ for all $i=1,\cdots, n$.
For the equivalence between the latter condition and (2), we refer readers to
section 4.1 of \cite{C3} or  Lemma 13.1 of \cite{FOOO}, where one uses the product
structure of Floer cohomology classes to show that $\TE_i$'s are non-trivial classes. The
rest is left as an exercise.
\end{proof}

In practice, we use derivatives with respect to $y_i$, and $\frac{\partial}{\partial x_i}$
is the same as $y_i \frac{\partial}{\partial y_i}$.
In fact, the variable $y$ depends on $u$ and written as $y^{\bf u}$ in \cite{FOOO5}, but
potential function given as \eqref{POformula1} is independent of $u$.
Thus, we may take $u=0$, and write $y$ for $y^{\bf 0}$ as in \cite{FOOO5} and write
$z_j = T^{\ell_j(0)}(y_1)^{b_{j1}}\dots (y_n)^{b_{jn}}$.
In \cite{FOOO5},
they introduce $(\eta_1,\cdots \eta_n) \in (\Lambda \setminus \{0\})^n$ as
possible domain for $(y_1,\cdots, y_n)$ and considered $$A(Int(P)) = \{(\eta_1,\cdots \eta_n) \in (\Lambda \setminus \{0\})^n \mid
(\frak v_T(\eta_1), \cdots, \frak v_T(\eta_n)) \in Int(P)\}.$$
Then, the relevant information of $u$ from $y$ variable can be read off from the
valuation $\frak v_T$ of $y$ variables, and  $PO$ can be considered as a function on $A(Int(P))$.
\begin{theorem}[c.f. Theorem 5.9 \cite{FOOO5}]\label{thm:computeHF}
For $u \in Int (P)$, the following two conditions are equivalent.
\begin{enumerate}
\item There exists $b \in H^1(L(u);\Lambda_0)$ such that
we have an isomorphism as modules
$$HF\big( (L(u),b), (L(u),b);\Lambda_0 \big) \cong H(T^n;\Lambda_0).$$
\item There exists $\eta = (\eta_1,\cdots \eta_n)  \in A(Int(P))$ such that
$$\eta_i \frac{\partial PO}{\partial y_i}(\eta) =0 $$
for $i=1,\cdots, n$ and that
$$(\frak v_T(\eta_1), \cdots, \frak v_T(\eta_n))=u.$$
\end{enumerate}
\end{theorem}
The proof is analogous to \cite{FOOO5} and omitted. 
We discuss examples of the smooth Floer cohomology
of Lagrangian torus fibers for  teardrop orbifolds and weighted projective spaces in section \ref{sec:ex}.

\section{Bulk deformations of Floer cohomology and bulk orbi-potential}\label{sec:bulkmoduli}
Bulk deformations were introduced in \cite{FOOO} as a way to deform  $\AI$-algebra
of a Lagrangian submanifold by an ambient cycle of the symplectic manifold. It gives further ways
to deform the Floer theory, which found to be very effective way of locating non-displaceable
torus fibers in toric manifolds(\cite{FOOO3}).

 For an orbifold $\bX$ and a smooth Lagrangian submanifold $L$, bulk deformations
from inertia components of $\bX$, play much more important role, because, $J$-holomorphic orbi-discs come into Floer theory only via bulk deformations. This is because, holomorphic orbi-discs has a domain, which
has an interior orbifold singularity, and we have used interior orbifold marked point to record the
orbifold structure of such a domain. Thus, description of $J$-holomorphic orbi-discs always requires
at least one interior orbifold marked point.

We will see in examples in section \ref{sec:ex} that these bulk deformations are very important to understand symplectic geometry of orbifolds, because the
 very rigid feature of Hamiltonian dynamics of orbifolds, are detected by bulk deformations via
 twisted sectors.

In this section, we first  explain our setting of bulk deformations for toric orbifolds, and set up bulk deformed $\AI$-
algebras, and analyze the bulk potentials of them.

\subsection{Bulk deformation}
We follow \cite{FOOO}, and \cite{FOOO3} to set up bulk deformations of
$\AI$-algebras as follows. The new feature is that  for toric orbifold $\bX$, we consider bulk deformation via fundamental class of twisted sectors.

\begin{definition}\label{def:bu}
For each $\nu \in Box'$, consider fundamental cycles $1_{\bX_\nu} \in H^{0}(\bX_\nu;R)$ of inertia component $\bX_\nu$, and regarded it as an element with degree $2 \iota_\nu$ ( i.e. $deg(1_{\bX_\nu})  =2 \iota_\nu$)
as in \cite{CR}.
Also, consider toric divisor $D_i$ of $\bX$. We take a finite dimensional graded
$R$-vector space $H$ generated by these $1_{\bX_\nu}$'s and $D_i$'s:
\begin{equation}\label{def:H}
H = \oplus_{\nu \in Box'} R <1_{\bX_\nu}> \oplus_{i=1}^m R<D_i>.
\end{equation}
\end{definition}
Note that we do {\em not} consider more general bulk deformations by $H^\ast_{orb}(\bX)$
in this paper. To simplify notation, we label elements of $Box'$ as
\begin{equation}\label{boxlist}
Box' = \{ \nu_{m+1},\cdots, \nu_B\}.
\end{equation}
We define
 \begin{equation}\label{hlist}
 H_a =  \begin{cases}   D_a & \textrm{for}\;\; 1 \leq a \leq m  \\
 1_{\bX_{\nu_a}}& \textrm{for}\;\; m+1 \leq a \leq B. \end{cases}
 \end{equation}
 These $H_a$'s  for $a=1,\ldots,B$ form a basis of $H$.

 For $\frak b_a \in \Lambda_+$  for each $a$,
we consider an element
\begin{equation}\label{fa}
\frak b = \sum_{a} \frak b_a H_a \in
H \otimes \Lambda_+.
\end{equation}

Bulk deformations uses  following  family of operators
\begin{equation}
\fq_{\beta;\ell,k}:
E_{\ell}  (H[2])   \otimes   B_k(H^*(L;R)[1]) \to  H^*(L;R)[1].
\end{equation}
Here, degree shiftings  $H[2]$ and $H^*(L;R)[1]$ are introduced to
make the degree of the map $\fq_{\beta;\ell,k}$ to be $1 - \mu(\beta)$,
where
$2, 1$ corresponds to the degree of freedom of interior and boundary marked points in $D^2$.

The symmetrization $E_{\ell}C$ of $B_\ell C$ can be defined as invariant elements of $B_\ell C$
under symmetric group action. Consider the standard coproduct $\Delta:BC \to BC$ and  $\Delta^{n-1}: BC \to (BC)^{\otimes n}$ or $EC \to
(EC)^{\otimes n}$ which is defined by
\begin{equation}
\Delta^{n-1} = (\Delta \otimes  \underbrace{id \otimes \cdots
\otimes id}_{n-2}) \circ (\Delta \otimes  \underbrace{id \otimes
\cdots \otimes id}_{n-3}) \circ \cdots \circ \Delta.
\end{equation}
An element $\bfx \in BC$ under $\Delta^{n-1}$ can be written as
\begin{equation}
\Delta^{n-1}(\bfx) = \sum_c \bfx^{n;1}_c \otimes
\cdots \otimes \bfx^{n;n}_c,
\end{equation}
for $c$ running over some index set  for each $\bfx$. 
Shifted degree of the element 
$\bfx = x_1 \otimes \cdots \otimes x_k$
is given by $ \deg' \bfx = \sum \deg'x_i$.
\begin{theorem}[c.f. Theorem 3.8.32 \cite{FOOO}]\label{bulkopdef} For toric orbifold $\bX$ and
Lagrangian torus fiber $L$, the operators $\fq_{\beta;l,k}$ can be constructed to have the following properties.
\begin{enumerate}
\item
For  $\beta$ and $\bfx \in B_k(H(L;R)[1])$,
$\bfy \in E_l(H[2])$, we have
\begin{equation}\label{qeq1}
0 = \sum_{\beta_1+\beta_2=\beta}\sum_{c_1,c_2} (-1)^\epsilon \fq_{\beta_1}(\bfy^{2;1}_{c_1};
\bfx^{3;1}_{c_2} \otimes \fq_{\beta_2}(\bfy^{2;2}_{c_1};\bfx^{3;2}_{c_2})
\otimes \bfx^{3;3}_{c_2}) 
\end{equation}
where
\begin{equation}
\epsilon = \deg'\bfx^{3;1}_{c_2}(1 +  \deg \bfy^{2;2}_{c_1}) +\deg \bfy^{2;1}_{c_1}.
\end{equation}
Here, we write $\fq_{\beta}(\bfy;\bfx)$ for
$\fq_{\beta;l,k}(\bfy;\bfx)$.

\item For $1 \in E_0(H[2])$ and $\bfx \in B_k(H(L;R)[1])$, we have
\begin{equation}
\fq_{\beta;0,k}(1;\bfx) =  m_{k,\beta}(\bfx),
\end{equation}
where $ m_{k,\beta}$ is the filtered $A_{\infty}$ structure on
$H(L;R)$ constructed in \eqref{def:mk}.

\item  Consider
$\bfx = \bfx_1 \otimes \TE \otimes \bfx_2
\in B(H(L;R)[1])$. Then
\begin{equation}
\fq_{\beta}(\bfy;\bfx) = 0
\end{equation}
except
\begin{equation}
\fq_{\beta_0}(1;\TE \otimes x) =
(-1)^{\deg x}\fq_{\beta_0}(1;x \otimes \TE) = x,
\end{equation}
where  we have $\beta_0 = 0 \in H_2(X,L;\Z)$ and $x \in H(L;R)[1]$.
\end{enumerate}
\end{theorem}
We explain the construction of $\fq$ in the next
subsection \ref{sec:constq}. The proof of the theorem
then follows from that of  Theorem 3.8.32 \cite{FOOO} and Theorem 2.1 of \cite{FOOO3},
and we omit the further details.

Using the notation in \eqref{fa}, we define
\begin{equation}
 m_{k}^{\frak b}(x_1,\ldots,x_k) = \sum_{\beta}\sum_{l=0}^\infty
T^{\omega(\beta)/2\pi} \fq_{\beta;l,k}(\frak b^{\otimes l};x_1,\ldots,x_k).
\end{equation}
The above theorem implies that
\begin{lemma}[Lemma 2.2 \cite{FOOO3}] The operations
$\{ m^{\frak b}_{k}\}_{k=0}^{\infty}$ define a structure of filtered $\AI$-algebra
 on $H(L;\Lambda_0)$.
\end{lemma}

The element  $b \in H^1(L;\Lambda_{+})$ is called a {\em weak bounding cochain} of the filtered $A_{\infty}$
algebra $(H(L;\Lambda_{0}),\{ m_{k}^{\frak b}\})$ if 
\begin{equation}\label{eqc}
m_k^{\frak b}(e^b) = \sum_{k=0}^{\infty}  m_{k}^{\frak b}(b,\ldots,b) = c PD([L]),
\end{equation}
for some constant $c \in \Lambda_+$.
In fact, one can extend it for $b \in H^1(L;\Lambda_{0})$ exactly the same way as in
Proposition \ref{unobstruct}, and we omit the details.
We define the potential $PO(\frak b,b)$ by the equation \eqref{eqc}
\begin{equation}
PO(\frak b,b) = c \in \Lambda_+.
\end{equation}

\begin{definition}\label{pobulkdef}
The set of the pairs $(\frak b,b)$ such that $b$ is a  weak bounding cochain  of $(H(L;\Lambda_0),\{\frak m_{k}^{\frak b}\})$, is denoted as $\WH{\CM}_{weak,def}(L;\Lambda_0)$.
$PO(\frak b, b)$ defines  {\it the potential function}:
$
PO: \WH{\CM}_{ weak,def}(L;\Lambda_0) \to \Lambda_+$. 

We also use the notation $PO^{\frak b}(b)$, and $PO^{\frak b}(b,u)$ sometimes for $PO(\frak b, b)$.
\end{definition}

For $(\frak b,b) \in \WH{\CM}_{ weak,def}(L;\Lambda_0)$,
we have differential satisfying $m_1^{\frak b,b} \circ m_1^{\frak b,b} =0$:
\begin{equation}
m_1^{\frak b,b} =\sum_{k=0}^\infty \sum_{\ell=0}^\infty
m_{k+\ell +1}^{\frak b}(b^{\otimes k},x,b^{\otimes \ell}),
\end{equation}

\begin{definition}[\cite{FOOO} Definition 3.8.61]\label{defbulkfl}
For  $(\frak b,b) \in \WH{\CM}_{ weak,def}$,
we define Floer cohomology with  deformation $(\frak b, b)$ by
\begin{equation}
HF((L,\frak b,b),(L,\frak b,b);\Lambda_0) = \frac{\text{\rm Ker}(m_1^{\frak b,b} )}
{\text{\rm Im}(m_1^{\frak b,b} )}.
\end{equation}
\end{definition}

\subsection{Construction of $\fq$ for toric orbifolds}\label{sec:constq}
In this section, we construct  the operator $\fq$ using the moduli space of holomorphic (orbi)-discs to prove theorem \ref{bulkopdef}.
Recall from Definition \ref{def:bu} that bulk deformation by elements of $H$ are considered, where $H$ is generated by fundamental class $[1_{\bX_\nu}]$'s for  $\nu \in Box'$ and by divisors $D_i$'s (for $i=1,\cdots,m$).

First, we consider the relevant moduli spaces.
Recall that (Definition  \ref{def:orbimark}) we denote by $\UL{l} = \{1,\cdots, l\}$ and consider the map
 $\bx : \UL{l} \to Box$, where
 a stable map $\big( (\Sigma, \vec{z},\vec{z}^+),w,\xi \big)$ is said to be of type $\bx$ if
for $i=1,\cdots,l$,
$$ev^+_i\big( (\Sigma, \vec{z},\vec{z}^+),w,\xi \big) \in \bX_{\bx(i)}.$$

To include the interior intersection condition with toric divisors, we introduce
the following notations. Consider a function
$$\bfp:\UL{l} \to \{1,\cdots, B\}$$
to describe bulk intersection. We write $|\bfp|=l$. The set of all such $\bfp$ are
denoted as $Map(l,\UL{B})$. From $\bfp$, we define
$\bx: \UL{l} \to Box$ as follows.
$$\bx(j) =  \begin{cases}   \nu_j & \textrm{if}\;\; m+1 \leq j \leq B \\
0 & \textrm{if}\;\; \bfp(j) \in \{1,\cdots, m\} \end{cases} $$
We enumerate the set of all  $j \in \UL{l}$ with $\bx(j)=0$ as
$\{j_1,\cdots,j_{l_1}\}$.

We define a fiber product
\begin{equation}
\mathcal M^{main}_{k+1,l}(L(u),\beta;\bfp)
=
\mathcal M^{main}_{k+1,l}(L(u),\beta,\bx)
{}_{(ev^+_{j_1},\cdots,ev^+_{j_{l_1}})}
\times_{X^{l_1}} \prod_{i=1}^{l_1} D_{\bfp(j_i)}.
\end{equation}
The virtual dimension of the above fiber product is (see section \ref{sec:index})
\begin{equation}\label{dim:bulk}
 n + \mu(\beta) + k+2l +1 -3 -\sum_{j=1}^l 2\iota(\bx(j)).
\end{equation}
We also remark that we take fiber product only at smooth interior marked points,
and hence the above fiber product is the usual fiber product, not the orbifold one.

The following lemma is an analogue of Lemma 6.3 in \cite{FOOO3} and part of
it is already discussed in the Proposition \ref{prop:tnequi}. We omit its proof and refer readers to \cite{FOOO3}.
\begin{lemma}
Moduli space $\mathcal M^{main}_{k+1,l}(L(u),\beta;\bfp)$
has a $T^n$-equivariant Kuranishi structure, and  the evaluation map
\begin{equation}
ev = (ev_0,ev_1,\ldots,ev_k):
\mathcal M^{main}_{k+1,l}(L(u),\beta;\bfp)
\to L(u)^{k+1}
\end{equation}
is weakly submersive and $T^n$-equivariant. It is oriented and has  a tangent bundle. 
\end{lemma}

The consideration of boundary of a moduli space is by now standard, can be
done as in \cite{FOOO3} Lemma 6.4. (We skip the details and refer to \cite{FOOO3}).

 Lemma 6.5 \cite{FOOO3} also generalizes to our situation.
Let
\begin{equation}
\mathfrak{forget}_0:  \mathcal M^{\text{\rm
main}}_{k+1,l}(L(u),\beta;\bfp) \to  \mathcal
M^{main}_{1,l}(L(u),\beta;\bfp)
\label{forgetmap1}\end{equation} be the forgetful map which forgets all
the boundary marked points except the $0$-th one. We may choose our
Kuranishi structures so that (\ref{forgetmap1}) is compatible with
$forget_0$ of Lemma \ref{forgetlemma}.

\begin{lemma}[c.f. Lemma 6.5 \cite{FOOO3}]\label{bulkpertsection}
Fix $E > 0$. Then there exists a system of
multisections $\mathfrak s_{\!\beta,k+1,l,\bfp}$ on
$\CM^{main}_{k+1,l}(L(u),\beta;\bfp)$
for $\omega(\beta) < E$, $\bfp \in
Map(l,\underline B)$, satisfying the following properties.
\begin{enumerate}
\item They are transversal to zero section and invariant under $T^n$-action
\item The multisection $\mathfrak s_{\beta,k+1,l,\bfp}$
is given by the pull-back of the multisection $\mathfrak s_{\beta,1,l,\bfp}$
via the forgetful map \eqref{forgetmap1}.
\item multisection at the boundary is compatible with the fiber product
as in Lemma 6.5 of \cite{FOOO3}.
\item For $l=0$ the multisection
$\mathfrak s_{\beta,k+1,0,\emptyset}$ is the same
as the one defined in Lemma \ref{forgetlemma}.
\item multisection $\mathfrak s_{\!\beta,k+1,l,\bfp}$ is invariant under 
permutation of the interior marked points.
\end{enumerate}
\end{lemma}
\begin{proof}
The proof is similar to the proof of Lemma $6.5$ \cite{FOOO3} and omitted.
\end{proof}

We use the above moduli spaces to define the operators $\mathfrak
q_{\beta;k,l}$ as follows.
We put
$$H(\bfp) = H_{\bfp(1)}\otimes \cdots \otimes H_{\bfp(l)}
\in H^{\otimes l},$$

Then, $\fq$ is defined as in \eqref{def:mk} by pulling back differential forms
and pushing forward:
\begin{equation}
\fq_{\beta;l,k}(H(\bfp);h_1,\ldots,h_k)
= \frac{1}{l !}(ev_0)_! (ev_1,\ldots,ev_k)^*(h_1 \wedge \cdots \wedge h_k).
\label{eqbulkpull}\end{equation}

We define $\fq_{\beta;l,k}$ for $(\beta,l,k) \neq (0,0,0), (0,0,1)$ by
the above  and   put
\begin{equation}
\fq_{0;0,1}(h) = (-1)^{n+\deg h+1} dh, \;\;\;\; \fq_{0;0,2}(h_1,h_2)
= (-1)^{\deg h_1(\deg h_2+1)} h_1 \wedge h_2.
\end{equation}

\begin{remark}
One needs to  fix $E_0$ and construct $\mathfrak
q_{\beta;k,l}$ for $\beta\cap\omega < E_0$ and to take inductive limit, 
due to Kuranishi perturbation.
(see Sections 7.2 and 7.4 of \cite{FOOO}).
As in \cite{FOOO3},
we can use $A_{n,K}$ structure in place of $A_{\infty}$ structure, 
and we omit the details.
\end{remark}

We put
$
\fq_{l,k} = \sum_{\beta} T^{\omega \cap \beta/2\pi}
\fq_{\beta;l,k},
$ and by extending linearly to $H\otimes \Lambda_+^\R$, we obtain
an operator $\fq_{l,k}$ for the Theorem \ref{bulkopdef}
The proof of  \eqref{qeq1}  is the same as that of Theorem 2.1 of \cite{FOOO3} and omitted.
By taking $T^n$-invariant differential forms on $L$, we in fact obtain a canonical model $(H(L(u);\Lambda_0(\R)),\{
m_{k}^{\mathfrak b,{\text{\rm can}}}\}_{k=0}^{\infty})$ as before.

\subsection{Bulk orbi-potential of toric orbifolds}
Recall that in section \ref{sec:smoothpotential}, we have discussed smooth potential $PO$ for toric orbifolds. In this subsection, we discuss  the bulk (orbi)-potential $PO^{\mathfrak b}$( Definition \ref{pobulkdef}) of toric orbifolds, which should be considered
as a bulk deformation of the smooth potential $PO$.

Even for Fano orbifolds, it is very difficult to compute the bulk potential when we take
$\frak b$ from inertia components. The reason is related to the fact that constant orbi-spheres
with several orbifold marked points are in general obstructed, and Chen and Ruan \cite{CR2} introduced
Chen-Ruan cohomology ring of an orbifold from it.
We have found holomorphic orbidiscs with one orbifold marked point, and proved its Fredholm
regularity.  But to consider bulk deformations $\frak b$,
we need to consider several insertions of $\frak b$'s, and in general, even the constant orbi-spheres
will make the relevant compactified moduli spaces obstructed. Hence, it is hard to compute them
directly.
We remark that in the forthcoming work of the first author with K. Chan, S.C. Lau and H.H. Tseng, we
find a way to compute these bulk orbi-potential for some cases, and show that these gives rise to
geometric understanding of the (open) Crepant resolution conjecture and change of variable formulas.

We will define a notion of leading order potential for toric orbifold, which we can compute explicitly
using the classification of basic orbi-discs. This will be enough to determine Floer cohomology
deformed by $(b,\frak b)$.

First we consider the dimension restrictions. From \eqref{dim:bulk}, the moduli space
$\mathcal M^{main}_{1,l}(L(u),\beta;\bfp)$  contributes to the bulk potential
if the following equality holds.
\begin{equation}\label{dimensionf2}
 n + \mu(\beta) + 1+2l -3 -\sum_{j=1}^l 2\iota(\bx(j)) = n, \;\; \textrm{or} \;\;
 \mu(\beta) = 2 + \sum_{j=1}^l (2\iota(\bx(j)) -2)
 \end{equation}
and $\beta\ne 0$.
In such a case, note that the moduli space defined in Lemma \ref{bulkpertsection}
$$
\mathcal M^{main}_{1,l}(L(u),\beta;\bfp)
^{\mathfrak s_{\beta,1,l,\bfp}}
$$
has a virtual fundamental {\it cycle}, because boundary strata involve
moduli spaces of lower dimension, but due to  $T^n$-equivariant condition, such boundary contribution vanishes as the expected dimension is less than $n$ as in Lemma \ref{invari}. Hence we can define the following
orbifold open Gromov-Witten invariants.
\begin{definition}
The number $c(\beta;\bfp) \in \Q$ is defined by
$$
c(\beta;\bfp)[L(u)]
= ev_{0*}([\CM^{main}_{1;l}(L(u),\beta;\bfp)
^{\mathfrak s_{\beta,1,l,\bfp}}]).
$$
\end{definition}

\begin{lemma}
The number $c(\beta;\bfp)$ is well-defined,  independent of the choice of $\mathfrak s_{\beta,k+1}$ in  Proposition $\ref{bulkpertsection}$.
\end{lemma}
The proof is the same as the proof of  Lemma 11.7 \cite{FOOO2} and
so is omitted.

From the classification results, Prop. \ref{prop:orbimoduli} (5), we know the one point
orbifold disc invariants.
\begin{lemma}\label{lem:p=1}
For $|\bfp|=1$, we have
$$c(\beta_a; \bfp)  =
 \begin{cases}   0 & \textrm{if}\;\; \bfp(1) \neq a \\
1 & \textrm{if}\;\; \bfp(1) =a \end{cases}. $$
\end{lemma}

\begin{lemma}\label{lem:shape1}
For $r \in H^1(L(u);\Lambda_{+})$, $\beta \in
H_2(X,L;\Z)$, and $\bfp \in Map(l,\underline B)$ satisfying the dimension condition \eqref{dimensionf2},
we have
$$\fq_{\beta;l,k}(H(\bfp);r,\ldots,r)
= \frac{c(\beta;\bfp)}{ l! k!}
(r(\partial \beta))^k \cdot PD([L(u)]),$$
\end{lemma}
Note that  \eqref{dimensionf2} is needed to have non-zero value
by dimension counting. The calculation is the same as that of
Lemma \ref{fooocal}, and omitted.

From this calculation, as in Proposition \ref{unobstruct}, we have
\begin{prop} We have the canonical inclusion
$$
( H\otimes \Lambda_+) \times  H^1(L(u);\Lambda_0) \hookrightarrow
\WH{\CM}_{ weak,def}(L(u)).
$$
\end{prop}
\begin{remark}
We do not know how to extend above to $H \otimes \Lambda_0$.
Namely, it is desirable in several cases to have a  bulk insertion with energy zero,
but it is hard to make it rigorously defined in the orbifold case. On the contrary, for toric manifolds,
bulk deformation can be extended over $\Lambda_0$ for degree 2 classes of
ambient symplectic manifold, because the related open Gromov-Witten invariants can
be readily computed using divisor equation.
\end{remark}

We choose
$\frak b \in H\otimes \Lambda_+$
and  $ b \in H^1(L(u);\Lambda_0)$. Thus, we have a weak bounding cochain
$(\frak b, b)$, and
Definition \ref{pobulkdef} defines the bulk potential  $PO(\mathfrak b,b)$.
If we set $\frak b=0$,  we get $PO(0,b) = PO(b)$, the smooth potential discussed in section \ref{sec:smoothpotential}.

Next, we describe the leading order bulk potential for toric orbifolds. Leading order bulk potential is
a part of the full potential, and can be explicitly computed by the classification of basic (orbi)-discs.
Furthermore, we show in the next section, that non-displaceability of a Lagrangian torus fiber can
be obtained by studying the leading term equation, which will be derived from leading order bulk potential.

Let us write $$\frak b = \frak b_{sm} + \frak b_{orb}$$
where $$\begin{cases} \frak b_{sm} = \sum_{i=1}^m \frak b_{i} D_i  & \frak b_i \in \Lambda_+ \\
 \frak b_{orb} = \sum_{\nu \in Box'} \frak b_{\nu} 1_{\bX_\nu}  & \frak b_\nu \in \Lambda_+
 \end{cases}$$

 Recall that for each $\nu_a \in Box'$, we denoted the corresponding lattice vector as $\bb_a$.
%Also, we sometimes denote $\bb_a = \nu_a$ for $g\geq m+1$ for convenience, and denote $\beta_a \in H_2(\bX, L)$ the homology class of the basic disc
%corresponding to $\bb_a$.
% we further denote
 %$\frak b_\nu = \frak b_{\nu,0}( 1 + \frak b_{\nu,+})$, where $\frak b_{\nu,+} \in \Lambda_+$
 %and $\frak b_{\nu,0} = c_\nu T^{\lambda_\nu}$ for some $c_\nu \in \C$ and
% $\lambda_\nu >0$.
%Recall that $\bb_j = (b_{j1},\cdots, b_{jn}) \in N$.
% Since  $\nu_a \in Box'$ is a lattice vector in $N$, we denote $\nu _a= (b_{a1},\cdots, b_{an}) \in N%$.%

\begin{definition}\label{leadingorderorb}
 We define the leading order potential $PO^{\frak b}_{orb,0}(b)$ as
 \begin{equation}
 PO^{\frak b}_{orb,0}(b) =  \sum_{j=1}^m T^{\ell_j(u)}(y_1)^{b_{j1}}\dots (y_n)^{b_{jn}}
 + \sum_{\nu_a \in Box'} \frak b_{\nu_a}  T^{\ell_a(u)}(y_1)^{b_{a1}}\cdots (y_n)^{b_{an}}
 \end{equation}
 \end{definition}
 Note that the first summations are the leading order terms $PO_0(b)$ of the smooth potential $PO(b)$
and the second summations are contributions
 from $Box'$.  More precisely, in the classification of holomorphic orbi-discs (Corollary \ref{cor:orbidisc}), we have found
 one-to-one correspondence between basic holomorphic orbi-disc (modulo $T^n$-action) and  twisted sectors $Box'$ of the toric orbifold.
  These basic orbi-disc contributions are the new terms in
 $PO^{\frak b}_{orb,0}(b) - PO_0(b)$, since  $PO_0(b)$ are contributions from basic smooth holomorphic
 discs.

We remark that for the case of toric manifolds in \cite{FOOO3}, leading order bulk potential $PO^{\frak b}_0(b)$ is the same as the leading order potential for the  potential $PO_0(b)$ (without bulk), since every bulk contributions come from holomorphic discs (by adding interior marked points). But in our case of toric orbifolds, addition of $\frak b_\nu$ allows holomorphic orbi-discs into the theory, and provides
new terms in the leading order potential also. Hence it is quite different from the case of manifolds.

  It is important to note that the smooth potential $PO_0$ is independent of $\frak b$, but
 $PO^{\frak b}_{orb,0}(b)$ {\em depends }on the choice of $\frak b_{\nu}$.
 In particular, we will see that the freedom to choose this coefficient $\frak b_{\nu}$ allows us to find much more non-displaceable Lagrangian torus fibers in toric orbifolds.

 In our applications ( in the next section and  in examples), we will choose simpler type of  bulk deformations  as $\frak b_{\nu} = c_\nu T^{\lambda_\nu}$ for some $c_\nu \in \C$ and
 $\lambda_\nu >0$, but in general,  one can work with more general cases. In fact, we may define
 leading order potential by just taking the term of $\frak b_{\nu}$ with the smallest $T$-exponent for each $\nu$,  as it will give rise to the same leading term equation later on.

To discuss the general form of the bulk potential, we need a notion of $G$-gappedness 
for a discrete monoid $G$, which we refer readers to Definition 3.3 of \cite{FOOO3}.
The discrete monoid $G$ in this setting is defined as in \cite{FOOO3}.
\begin{equation}\label{gap:1}
G(X) = \langle \{\omega (\beta) / 2\pi \mid  \beta\in
\pi_2(X) \;\text{ is realized by a holomorphic sphere}\} \rangle.
\end{equation}
The actual discrete monoid to be used, $G_{bulk}$ will be defined
defined in Definition \ref{Gbulk} and  $G(X)$ is a subset of $G_{bulk}$.

We discuss the general form of the bulk potential for toric orbifolds, roughly given by the leading order bulk potential  with additional
higher order terms.
\begin{theorem}[c.f. \cite{FOOO3} Theorem 3.5]\label{thm:podiff}
Let $X$ be a compact symplectic toric orbifold  and let $\frak b \in  H(\Lambda_+)$ be a $G_{bulk}$-gapped element.
Then  the difference of the bulk orbipotential and its leading order potential can be
written as follows.
\begin{equation}\label{podiff:eq}
PO(\frak b;b) - PO^{\frak b}_{orb,0}(b)
= \sum^{\infty}_{\sigma=1} c_\sigma
y_1^{v'_{\sigma,1}}\cdots y_n^{v'_{\sigma,n}}T^{\ell'_\sigma+\rho_\sigma},
\end{equation}
for some  $c_\sigma \in \Q$, $e_\sigma^i \in \Z_{\geq 0}$,
$\rho_\sigma \in G_{bulk}$ and $\rho_\sigma > 0$, such that
$\sum_{i=1}^B e_\sigma^i > 0$.
Here
\begin{equation}
v'_{\sigma,k} = \sum_{a=1}^B e_\sigma^a b_{a,k} , \quad \ell'_\sigma = \sum_{a=1}^B e_\sigma^a \ell_a.
\end{equation}
If we have infinitely many non-zero $c_\sigma$'s, we have
\begin{equation}\label{rtoinfty}
\lim_{\sigma\to\infty}
\rho_\sigma = \infty.
\end{equation}
\end{theorem}
\begin{proof}
The proof is in the same line as that of \cite{FOOO3} theorem 3.5.
Let $\frak b = \sum_{a=1}^{B}\frak b_a H_a$ with
$\frak b_a \in \Lambda_+$, where $\frak b_a$ is $G_{bulk}$-gapped.
Note that $c(\beta:\bfp)$ determines $\fq_{\beta;\vert\bfp\vert,k}(H(\bfp);b,\ldots,b)$
from the formula \eqref{lem:shape1}. Hence, proceeding as in \eqref{eq:shape2}
we obtain that
\begin{eqnarray}
PO(\frak b;b)
&= &\sum_{\beta,\bfp,k} \frak b^{\bfp} T^{\omega(\beta)/2\pi}\frac{c(\beta;\bfp)}{k!| \bfp |!}
(b(\partial \beta))^k
\\
&=& \sum_{\beta,\bfp}\frac{1}{|\bfp| !}\frak b^{\bfp}T^{\omega(\beta)/2\pi}
c(\beta;\bfp) \exp(b(\partial \beta)).
\end{eqnarray}

Now, we consider the cases of $\vert\bfp\vert = 0, 1$ or $\vert\bfp\vert \geq 2$.
If $\vert\bfp\vert = 0$, there is no interior marked point, and hence there is
no orbifold disc contributions. The statement in this case follows from \eqref{POformula1}.

When $\vert\bfp\vert = 1$,  the case that the interior marked point is smooth is similar
to the case of smooth manifold, and it is enough to consider the case that the interior marked
point is orbifold marked point.
In this case, additional orbi-disc contributions for basic orbi-disc classes are computed
from Lemma \ref{lem:p=1}. For other homology classes, the
statement follows from Prop. \ref{prop:orbimoduli}.

We next study terms for $\vert\bfp\vert \geq 2$.
We first consider the case $\beta = \beta_a$  for $a=1,\cdots, B$.
In this case we obtain the following term
\begin{equation}\label{sumterm}
c T^{\ell_a(u) + \rho} y^{\bb_a},
\end{equation}
where  $c\in \Q$ and $\rho$ is obtained by summing over the exponents of $\frak b_{
\bfp(j)}$
for various $j$. As $l \neq 0$ and  $\frak b_{
\bfp(j)} \in \Lambda_+$, this is non-zero. 
Hence $\rho \in G_{bulk} \setminus\{0\}$. Therefore the form of \eqref{sumterm} equals
the right hand side of \eqref{podiff:eq}.

Now, we consider  $\beta \neq \beta_a$
$(a=1,\ldots,B)$. We may assume that $c(\beta;\bfp) \neq 0$. Then
by Proposition \ref{prop:orbimoduli} (4) we have $e^i$ and $\rho$ satisfying
$$
\frac{\omega(\beta)}{2\pi} = \sum^B_{a=1} e^a \ell_a(u) + \rho.
$$
Here $e^a \in \Z_{\geq 0}$ and $\sum e^a> 0$ and $\rho$ is from a sum of symplectic
areas of holomorphic spheres( divided by $2\pi$). Hence these gives rise to an expression
$$
c T^{\sum_{a} e^a\ell_a(u) + \rho+\rho'} y^{\sum e^a\bb_a},
$$
where $c \in \Q$ and $\rho'$ 
is obtained by summing over the exponents of $\frak b_{
\bfp(j)}$ for various $j$.
This form agrees with the right hand side of (\ref{podiff:eq}).

The proof of  (\ref{rtoinfty}) is based on the idea that
to have infinitely many terms, either infinitely many bulk insertions contributing to the potential, or the contribution of energy from the sphere component should go to infinity, and the proof is the similar to that of \cite{FOOO3}, and omitted.
\end{proof}

Theorem \ref{compHF:thm}, Theorem \ref{thm:computeHF} can be easily generalized into bulk setting.
\begin{theorem}\label{computefloer2}
Let $b = \sum x_i \TE_i$, and $\frak b \in H(\Lambda_+)$.  
Theorem \ref{compHF:thm} and Theorem \ref{thm:computeHF} holds in the bulk case too
by replacing $PO$, $(L(u),b)$ with $ PO^{\frak b}$ 
 $(L(u),(\frak b,b))$ respectively.
\end{theorem}
The proof is analogous to \cite{FOOO3} and omitted.
We also remark that the above can be extended to the following general form as in \cite{FOOO3}:
If $(\mathfrak b,b)$ satisfies
\begin{equation}\label{crit2}
\mathfrak y_i \frac{\partial PO^u}{\partial y_i}(\mathfrak b,\mathfrak y) \equiv 0 \mod T^{\mathcal N},
\end{equation}
then we have
\begin{equation}\label{eq:homoiso2}
HF((L(u_0),\mathfrak b,b),(L(u_0),\mathfrak b,b);\Lambda_0/T^{\mathcal N})
\cong H(T^n;\Lambda_0/T^{\mathcal N}).
\end{equation}

\section{Leading term equation and bulk deformation}
From the theorem \ref{computefloer2}, if the (bulk) potential function is known, then  Floer cohomology is determined by considering the critical points of the potential function. But for toric orbifolds, the full bulk potential is
very difficult to compute even for Fano orbifolds.  For toric manifolds,  recall that the notion of leading term equation was introduced in section 4 of \cite{FOOO3}, so that one can determine the Floer cohomology only from
the knowledge of leading order potential, which is explicitly calculable.
Namely, given the solutions of leading term equations, they show that there exist a bulk deformation $\frak b$ such that  potential $PO(\frak b, b)$ has an actual critical point. The bulk-balanced Lagrangian fibers can be located in this method.

In this section, we define leading term equation for toric orbifolds.
The construction is similar to that of section 4 of \cite{FOOO3}. Instead of repeating their
construction, we make our construction in the same form as that of \cite{FOOO3}, so that
once we prove the Proposition \ref{prop:LTEind} in this paper, which
play the role of Proposition 4.14 of \cite{FOOO3},  the rest of construction,
which is rather long, becomes the same and can be omitted.

Note that leading term equations are determined from leading order equation (defined in definition \ref{leadingorderorb}), and for toric orbifolds, they depends on $\frak b_{\nu}$.
Thus, the crucial difference  from \cite{FOOO3} is that
\cite{FOOO3} deals with leading term equation of $PO_0(b)$(which is independent of $\frak b$) whereas we deal with
that of  $PO^{\frak b}_{orb, 0}$, where the latter depends on the choice of $\frak b_{orb}$.

We first set up some notations.
Recall from \eqref{bv} that bulk deformation terms corresponding to twisted sectors are
$$\frak b_{orb} = \sum_{\nu \in Box'} \frak b_\nu 1_{\bX_\nu}$$
and we denote
\begin{equation}\label{bv2}
\frak b_\nu = \frak b_{\nu,0} + \frak b_{\nu,+},
\end{equation}
 where
 and $\frak b_{\nu,0} = c_\nu T^{\lambda_\nu}$ for some $c_\nu \in \C$ and
 $\lambda_\nu >0$ and
  $\frak b_{\nu,+}$
satisfies $\frak v_T(\frak b_{\nu, +}) > \lambda_\nu$.

To define leading term equation in our case,
we fix $u \in P$, and bulk deformation $\frak b$ as above.
 We start with relabeling the indices $a=1,\cdots,B$.
Recall that for $1 \leq a \leq m$, $\bb_a \in N$ is the stacky vector corresponding to $a$-th facet
of the polytope, and $\ell_a$ is the symplectic area of the corresponding basic disc (intersecting
that facet). For $m+1 \leq a \leq B$, $a$ labels elements of $Box'$, corresponding
to the lattice vector $\bb_a$ in $N$, and the area of the corresponding basic orbi-disc is $\ell_a$
 (see \ref{ellnu} for $\ell_a$ in this case).

We compare the areas $\ell_a$ for $a=1,\cdots, m$ and $\ell_a + \lambda_{\nu_a}$ for $a=m+1,\cdots, B$,  because the
orbi-disc for $\nu_j$ has an additional energy coming from
$\frak b_{\nu_j,0} = c_\nu T^{\lambda_{\nu_j}}$.
We enumerate energy levels as
\begin{equation}
\{S_l \mid l = 1,2,\ldots, \mathcal L\} = \{\ell_i(u), \ell_j(u)+ \lambda_{\nu_j} \mid i=1,2,\ldots, m
, j = m+1,\cdots, B\},
\end{equation}
so that $S_i < S_{i+1}$ and $S_i \in \R_+$.
Note that these are the exponents of $T$ of the terms in $PO^{\frak b}_{orb, 0}$.
The indices of $\bb_k$'s can be re-enumerated:
We denote $\{\bb_{l,1},\ldots,\bb_{l,a(l)}\}$ for all $\bb_k$'s satisfying
$$ \ell_k(u) = S_l ,  k\leq m, \;\; \textrm{or},\;\; \ell_k(u)+ \lambda_{\nu_k}= S_l, k \geq m.$$

By the following procedure, we determine an optimal energy level (so that $\bb_a$ is with smaller
or equal $T$-exponent spans $N_{\R}$).
We denote the $\R$-vector space generated by $\bb_{l',r}$ for $l' \leq l$, $r=1,\cdots,a(l')$
as  $A_l^\perp  \subset N_\R$.  The smallest integer $l$ such that
$A_l^\perp = N_\R$ is denoted as $K$.
The difference of the dimension of  $A_l^\perp$, and $A_{l-1}^\perp$ is denoted as $d(l)$
and we set $d(1)$ to be the dimension of$A_1^\perp$.
Then,  $\sum_{l=1}^K d(l)=n$.
We will only consider $l \leq K$ for $\bb_{l,r}$.

To relate to the original indices, integer $i(l,r) \in \{1,\ldots,B\}$
are defined  by $\bb_{l,r} = \bb_{i(l,r)}$.
We renumber the set of $\bb_i$'s for $i=1,\cdots B$ as
$$\{\bb_{l,r}
\mid
l=1,\cdots,K, r=1,\cdots,a(l)\}
\cup \{ \bb_{i} \mid i = \CK+1 ,\ldots, B\}
$$
where $\CK$ is given by 
$\CK = \sum_{l =1}^K a(l)$.

The rest of the procedure to define leading term equation  is similar to that of \cite{FOOO2} section 4,
and hence we only briefly sketch the construction. 
Now, at each energy level $S_l$, the vectors $\bb_{l,1},\cdots, \bb_{l,a(l)}$ may
not be linearly dependent (if $a(l) \neq d(l)$), and the next procedure chooses
a suitable basis of the subspace spanned by these vectors.
We denote the dual  basis of  $\TE_i$ of $H^1(L(u);\Z) \cong M$ by $\TE^*_i$, which becomes basis of  $N_\R$. Then,
basis  $\TE^*_{l,s}$ of $N_{\R}$ can be chosen so that
$\TE^*_{1,1},\ldots,\TE^*_{l,d(l)}$ becomes a 
$\Q$-basis of $A_l^\perp$ and also that all lattice vectors of $A_l^\perp$
is given as a integer linear combination of $\TE^*_{l,s}$'s. In particular
$\TE_{i}^* $ can be written as a integer linear combination of $\TE_{l,s}^*$'s.

By considering $\TE_i^*$ and $\TE_{l,s}^*$ as functions on $M_\R$, we may
write them as $x_i$ and $x_{l,s}$, and define $y_{l,s}$ as  $e^{x_{l,s}}$.
Then, $y_i$ is given as a monomial of $y_{l,s}$'s and hence so is
$y^{\bb_{l,r}}$ (see Lemma 13.1 \cite{FOOO3}).

In this way, we can define  $T^{S_l}$ part of the potential $PO^{\frak b}_{orb, 0}$:
\begin{equation}
(PO^{\frak b}_{orb, 0})_l = \sum_{r=1, i(l,r) \leq m}^{a(l)} y^{\bb_{l,r}} +
\sum_{r=1, i(l,r) > m}^{a(l)} c_{\nu_{i(l,r)}} y^{\bb_{l,r}} ,
\;\; \textrm{for}\;\; l=1,\dots ,K,
\end{equation}
where $c_{\nu_{i(l,r)}}$ is given in \eqref{bv2}.
Then, $(PO^{\frak b}_{orb, 0})_l$ can be
written as a Laurent polynomial of $y_{l',s}$ with  $s \leq
d(l')$ and $l' \leq l$.

\begin{definition}
The {\em leading term equation} of   $PO^{\frak b}_{orb, 0}$( or  that of $ PO^{\frak b}(b)$) is the system of equations
\begin{equation}\label{LTE}
\frac{\partial (PO^{\frak b}_{orb, 0})_l }{\partial y_{l,s}} = 0, \;\; \textrm{for}\;\;
l=1,\ldots,K;\, s=1,\ldots,d(l)
\end{equation}
with $y_{l,s} \in R\setminus  \{0\}$.
\end{definition}
Note that we only take derivatives of $(PO^{\frak b}_{orb, 0})_l$
with respect to the variables $y_{l,s}$ for  $s=1,\ldots, d(l)$, but not
with respect to the variables $y_{l',s'}$ for $l' < l$.

In view of Theorem \ref{computefloer2}, we need critical points of the actual bulk potential
$ PO^{\frak b}(b)$, but the solutions of leading term equation \eqref{LTE},
equals that of $ PO^{\frak b}(b)$ from the Lemma 4.4 of \cite{FOOO3}.
(The solutions of equation $y_k\frac{\partial PO^{\frak b}(b)}{\partial y_k} = 0$
correspondes to the solutions of equation 
$y_{l,s}\frac{\partial PO^{\frak b}(b)}{\partial y_{l,s}} = 0$
by Lemma 4.2 of \cite{FOOO}).

The following proposition \ref{prop:LTEind} on
the shape of bulk orbi-potential, is  analogous to Proposition 4.14 of \cite{FOOO3}.
We first define the monoid $G_{bulk}$.  Recall the definition of $G(X)$ from \eqref{gap:1}.
We define
\begin{equation}
G(L(u)) =
\langle \{\omega(\beta)/2\pi \mid  
\beta \in H_2(X,L(u)) \; \textrm{is from that of a holomorphic  orbi-disc}\} \rangle.
\end{equation}
\begin{definition}\label{Gbulk} $G_{bulk}$ is a discrete submonoid of $\R$ which is generated by
$G(X)$ and the subset
$$
\{\lambda - S_l \mid \lambda > S_l, \lambda \in G(L(u)), \quad l=1,\ldots,K,
 \} \subset \R_+.
$$
\end{definition}
Note that that $G(L(u)) \subset G_{bulk}$.

\begin{conds}\label{gap2}
Bulk deformation $\frak b$ is of the form
\begin{equation}
\frak b = \sum_{l=1}^{K}\sum^{a(l)} _{r=1}\frak b_{i(l,r)}
H_{i(l,r)} \in H(\Lambda_+)
\end{equation}
such that each $\frak b_{i(l,r)}$ is $G_{bulk}$-gapped.
Here $\frak b_{i(l,r)}$ means $\frak b_{\nu_{i(l,r)}}$ in case $i(l,r) >m$.
\end{conds}
The main proposition to prove in our orbifold case is the following.
\begin{prop}[c.f. Proposition 4.14 \cite{FOOO3}]\label{prop:LTEind}
Assume that $\frak b$ satisfies Condition \eqref{gap2} and
consider
\begin{equation}
\frak b^{\prime} = \frak b + cT^{\lambda} H_{i(l,r)},
\end{equation}
for $c \in R$, $\lambda \in G_{bulk}+ \frak b_{\nu_{i(l,r)},0}$, $l \leq K$.

Then the difference of the corresponding bulk orbipotentials is given as
\begin{eqnarray}\label{eq:appr1}
PO^{\frak b^{\prime}}(b) - PO^{\frak b}(b)
&= & c T^{\lambda + \ell_{i(l,r)}(u)} y^{\bb_{i(l,r)}}  +
\sum_{h=2}^{\infty} c_h T^{h\lambda + \ell_{i(l,r)}(u)} y^{\bb_{i(l,r)}}\\
& \,&  + \sum_{h=1}^{\infty}\sum_\sigma
c_{h,\sigma} T^{h\lambda + \ell'_\sigma(u) + \rho_\sigma}
y^{\bb_\sigma}.
\end{eqnarray}
Here $c_h, c_{h,\sigma} \in R, \rho_\sigma \in G_{bulk}$ and 
there exist $e_\sigma^i \in \Z_{\geq 0}$ such that $\bb_\sigma = \sum e_\sigma^i \bb_i$,
$\ell'_\sigma = \sum e_\sigma^i \ell_i$ and $\sum_i e_\sigma^i > 0$. Also in the third summand of right hand side of \eqref{eq:appr1}, the case that $h=1$ and $\rho_\sigma=0$, we have
$c_{h,\sigma}=0$.
\end{prop}
\begin{remark}
In Proposition 4.14 \cite{FOOO}, the last assertion is not written, but is shown in their proof and it is needed in  the induction for the theorem.
\end{remark}
\begin{proof}
We first remark that the proof in the toric orbifold case is somewhat different than that of toric manifolds.
For toric manifolds, the bulk deformation contribution when $\frak b$ is from toric divisors
are explicitly computable for the basic disc classes by homology arguments(similar to divisor equation).
(see for example Proposition 4.7 of \cite{FOOO3}.  But, for toric orbifolds,
such arguments does not work for basic disc classes, as there is no divisor type equation for
orbifold marked points.  But, as we will see, the proposition does not require this explicit computation.

The potential term has a contribution if the dimension satisfies \eqref{dimensionf2}
We will divide the contribution  of
\begin{equation}\label{diff}
PO^{\mathfrak b^{\prime}}(b)
- PO^{\mathfrak b}(b)
\end{equation}
 into the several cases, and subcases.
First, we first argue with the number of interior marked points.
Consider the terms corresponding to the case with no interior marked points.
Then, as they do not see $\frak b$, they give $0$ in \eqref{diff}.

Next, the case of one interior marked point. Recall that
one point orbifold Gromov-Witten invariant is computed in Lemma \ref{lem:p=1},
and it is non-zero only if the disc class is of $\beta_{i(l,r)}$, in which case
we get the first term of \eqref{eq:appr1}.

Thus, from now on, we consider the case with two interior marked points or more.
We remark that if the bulk insertion $\frak b$ is used in all of the interior insertions,
then, obviously, such term contribute $0$ to \eqref{diff}. So we
assume that at least one of the interior marked point is used for the insertion of
$T^\lambda H_{i(l,r)}$.
We divide it into three cases as follows.
\begin{enumerate}
\item $\beta= \beta_{i(l,r)}$.

We consider the following two subcases:
\begin{enumerate}
\item All bulk inputs are $T^\lambda H_{i(l,r)}$:
In this case,  it is easy to see that the contribution is of 2nd term of RHS of \eqref{eq:appr1}.

\item Both bulk inputs $T^\lambda H_{i(l,r)}$, and $\frak b$ are used at least once:
In this case, it contributes to the 3rd term of RHS of  \eqref{eq:appr1}, with
$h \geq 1$, $\ell_\sigma'(u) = \ell_{i(l,r)}$, and $\rho_\sigma >0$ since
it receives non-trivial contribution from $\frak b$.
\end{enumerate}

\item $\beta$  equals the basic disc class
$\beta_a$ for $a=1,\cdots, B$, and $a \neq i(l,r)$.
\begin{enumerate}
\item All bulk inputs are $T^\lambda H_{i(l,r)}$:
the possible contribution is of 3rd term of RHS of \eqref{eq:appr1}, with
$h \geq 2$, $\ell_\sigma'(u) = \ell_a$, and $\rho_\sigma =0$.

\item Both bulk inputs $T^\lambda H_{i(l,r)}$, and $\frak b$ are used at least once:
In this case, it contributes to the 3rd term of RHS of  \eqref{eq:appr1}, with
$h \geq 1$, $\ell_\sigma'(u) = \ell_a$, and $\rho_\sigma >0$ since
it receives non-trivial contribution from $\frak b$.
\end{enumerate}

\item $\beta \neq \beta_a$ for $a=1,\cdots, B$.

We may write
$$\beta = \sum_{i=1}^B e^i_\beta \beta_i + \sum_j \alpha_{\beta,j},$$
then we have
\begin{eqnarray*}
\frac{\omega(\beta)}{2\pi} &=&
\sum_{i=1}^B e^i_\beta \ell_i(u) + \sum_j \frac{\omega(\alpha_{\beta,j})}{2\pi} \\
exp(b ( \partial \beta)) &=& y^{\sum_{i=1}^B e^i_\beta \bb_i}.
\end{eqnarray*}
We have $e^i \geq 0$ and $\sum_i e^i_\beta >0$.
Thus the contributions belong to the third term of \eqref{eq:appr1} with
$\ell'_\sigma(u) = \sum_i e^i_\beta \ell_i(u)$ and $\rho_\sigma$ is the sum
of the contribution from the sphere class $\omega(\alpha_{\beta,j})$
together with contributions from $\frak b$.
Now we consider the following subcases:
\begin{enumerate}
\item All bulk inputs are $T^\lambda H_{i(l,r)}$:
the possible contribution is of 3rd term of RHS of \eqref{eq:appr1}, with
$h \geq 2$, $\ell_\sigma'(u)$ and  $\rho_\sigma$ as described above.
\item Both bulk inputs $T^\lambda H_{i(l,r)}$, and $\frak b$ are used at least once:
in this case, it contributes to the 3rd term of RHS of  \eqref{eq:appr1}, with
$h \geq 1$, $\ell_\sigma'(u)$ and  $\rho_\sigma$ as described above, and we have $\rho_\sigma >0$ since
it receives non-trivial contribution from $\frak b$.
\end{enumerate}
\end{enumerate}
This proves the proposition.
\end{proof}

For convenience, given $\frak b$ as in \eqref{bv}, \eqref{bv2},
we denote (by taking the least exponent terms)
$$\frak b_{orb,0} = \sum_{\nu \in Box'} \frak b_{\nu,0} 1_{\bX_\nu}$$
The leading order equation of $PO^{\frak b}_{orb,0}$ only depend on $\frak b_{orb,0}$, not the whole $\frak b_{orb}$, and so does
its leading term equation.

\begin{theorem}[c.f. Theorem 4.5 \cite{FOOO3}]\label{equivLTE}
The following conditions on $u$ are equivalent:
\begin{enumerate}
\item The leading term equation of $ PO^{\frak b}_{orb,0}(u)$ has a solution
$y_{l,s} \in R \setminus \{0\}$
$(l=1,\dots , K, s=1,\dots ,d(l))$.
\item There exists $\WT{\frak b} \in H(\Lambda_+)$ such that $\WT{\frak b}_{orb,0} =
\frak b_{orb,0}$
 and $ PO^{\WT{\frak b}}(u)$
has a critical point on $(\Lambda_{0}\setminus \Lambda_+)^n$.
\item There exists $\WT{\frak b} \in H(\Lambda_+)$ such that
$\WT{\frak b}_{orb,0} =
\frak b_{orb,0}$  and
 $y_{l,s} \in R \setminus \{0\}$
$(l=1,\dots , K, s=1,\dots ,d(l))$ in the item (1) above is a critical point of
$ PO^{\WT{\frak b}}(u)$.
\end{enumerate}
\end{theorem}
\begin{proof}[Proof of Theorem \ref{equivLTE}]
We have set up our case in the similar form to that of \cite{FOOO3} so that
 the same proof of
Theorem 4.5 \cite{FOOO3} should work in our case too, as we have proved
Proposition \ref{prop:LTEind} which plays the role of Proposition 4.14 of \cite{FOOO3}.
We refer readers to \cite{FOOO3} for its full proof and  briefly explain the rest of the procedure to prove Theorem  \ref{equivLTE}. The argument is exactly the same, except
the point regarding $\frak b_{orb,0}$.

Given a solution $\eta_{l,s}$ of the leading term equation, we need to find $\frak b$
such that $\eta_{l,s}$ satisfies the actual critical point equation:
\begin{equation}\label{criteq}
\eta_{l,s} \frac{\partial PO^{\WT{\frak b}}}{\partial y_{l,s}}(\eta) =0.
\end{equation}

We first enumerate elements of $G_{bulk}$ so that
$$G_{bulk}=\{\lambda_j^b \mid j=0,1,2 \ldots\}$$
where $0 =\lambda_0^b < \lambda_1^b< \cdots.$

Then we define $\WT{\frak b}$ inductively by choosing $\WT{\frak b}(k)$ for each $k$ for the terms with energy $S_l + \lambda_k^b$
(and also for $ 1 \leq l \leq K$.) (see Definition 4.15 of \cite{FOOO3})

First, we take
$$\WT{\frak b}(0) = \frak b_{orb,0}.$$
If the critical point equation \eqref{criteq}  is  satisfied up to the level $k$,
then, we introduce the bulk deformation $\WT{\frak b}(k)$ to make the
equation \eqref{criteq} satisfied up to level $k+1$ (see Proposition 4.18).

In this process, the equation \eqref{criteq} is satisfied up to $S_l + \lambda_k^b$,
hence we need to kill the error terms with $T$-exponent $S_l + \lambda_{k+1}^b$.
This can be done by choosing appropriate $\WT{\frak b}(k+1)$, and using
the Proposition 13.6 to cancel the error term with the first term of the RHS of (13.13).
As two other terms of RHS of (13.13) has higher $T$-exponent, it does not introduce
any other error terms on the level $k+1$.
Note that the we need to choose $\lambda$ of the proposition 13.6 so that $\lambda + \ell_{i(l,r)}(u)$
equals  $S_l + \lambda_{k+1}^b$, but as $S_l$ equals $\ell_{i(l,r)}(u) + \lambda_{\nu_{i(l,r)}}$.
Hence, we choose $\lambda$ here to be $\lambda_{k+1}^b+ \lambda_{\nu_{i(l,r)}}$.
Thus, we still have that the leading term $\WT{\frak b}_{orb,0}$ is not
changed and equals $\frak b_{orb,0}$.

Then one takes the limit as $k \to \infty$ to define $\WT{\frak b}$
such that  the equation \eqref{criteq} is satisfied.
We refer readers to  section 4 of \cite{FOOO3} for details.
\end{proof}

\section{Floer homology of Lagrangian intersections in toric orbifolds}
So far, we have discussed the Bott-Morse version of Floer cohomology of Lagrangian submanifolds.
In this section, we discuss the Lagrangian Floer cohomology between two Lagrangian submanifolds $L$ and $\psi_1(L)$, for a Hamiltonian isotopy $\psi_1$, and its relationship to $\AI$-algebra and
bulk deformed $\AI$-algebra which are constructed in the previous sections. We note that
the Lagrangian submanifold $L$ lies in the smooth part of the
orbifold $\bX$.

There are two versions of Floer cohomology of Lagrangian intersections, as
we have two versions of $\AI$-algebras.  Namely,  there is a smooth Lagrangian Floer cohomology
where we consider $J$-holomorphic strips and discs from a smooth domain. Here {\em smooth domain} means that it is not an orbifold domain, and smooth domains thus include nodal (smooth) Riemann surfaces.  We emphasize that the maps from the smooth domain can intersect orbifold points, as
we have seen in the case of smooth holomorphic discs corresponding to stacky vectors $\bb_i$
for $i=1,\cdots, m$.

And there
is a version, which includes orbifold $J$-holomorphic strips and discs, which are maps from
an orbifold domain. To denote the orbifold structure of the domain, we have introduced orbifold marked points in the interior of the Riemann surface, and their deformation theory is entirely analogous to
that of a Riemann surface with interior marked points.

Namely, orbifold marked points cannot disappear, created, nor combined  when we consider sequences
of orbifold $J$-holomorphic maps of a given type. Thus, when we consider only maps from smooth domains (into an orbifold), a degeneration which appear in the compactification of the moduli spaces of such maps is still from a smooth domain.  Hence, we have a smooth Floer theory for orbifolds.
Such theory still is non-trivial. Namely, we show in subsection \ref{sec:wps}, that smooth
Floer theory finds a central fiber to be non-displaceable in weighted projective spaces.

But, Lagrangian Floer theory involving orbifold strips and discs provide much more information
as we will see in several examples (actually the example of a teardrop already shows such phenomenon). But to have an orbifold structure
in the domain strip or disc, we need an orbifold marked point to record the orbifold structure of
the domain. Hence this  always requires interior marked points, and hence
they appear as a {\em bulk deformation theory} of the smooth theory.

\subsection{Smooth Lagrangian Floer homology}
First, we consider Hamiltonian vector fields in an effective orbifold $\bX$.
By definition, a smooth function $H:  \bX \to \R$ is a function $H:X \to \R$, which  locally has its lifting
$\tilde{H}_V:= H \circ \pi $ in any uniformizing chart $(V,G,\pi)$ such that
$\tilde{H}_V$ is smooth.  Note that $\tilde{H}_V$ is invariant under $G$-action:
$\tilde{H}_V(g \cdot x) = \tilde{H}_V(x)$.
Hamiltonian vector field $X_H$ can be defined by $i_{X_{\tilde{H}}} \omega  = d\tilde{H}$, and
$X_{\tilde{H}}$ is preserved by $G$-action because  the symplectic form $\omega$ (
on the chart $V$) is also invariant.

Hamiltonian isotopy $\psi^{H}_t$ of the flow $X_H$ is well-defined without much difficulty as we consider effective orbifolds: It is well-known that effective orbifolds can be always considered as a global quotient of a manifold, say $M$, by a
compact Lie group action, and one can use this presentation to define the flow of a vector field,
by integrating the flow after pull-back to the manifold $M$.

One can also consider time-dependent Hamiltonian functions (we still denote it by $H$ for simplicity), and define time-dependent Hamiltonian isotopy. The resulting Hamiltonian isotopies are regular (in the sense explained in the appendix), hence is a good map and the related group homomorphisms are isomorphisms as its inverse is also good.
This implies the following simple lemma.
\begin{lemma}
For  any Hamiltonian isotopy $\psi^H_t$, the isotropy group of the point $x$ and $\phi^H_t(x)$ are isomorphic.  This in particular implies that by an Hamiltonian isotopy, smooth points always move to smooth points of an orbifold.
\end{lemma}
 In particular, for our Lagrangian torus fiber, which lies away from the singular set  $\Sigma \bX$ of toric orbifold $\bX$, $\psi^H_t(L)$ also does not intersect $\Sigma \bX$.

For smooth Lagrangian Floer theory, we  only consider $J$-holomorphic
strips and discs (not orbifold ones). We may additionally consider smooth interior marked points to consider smooth bulk deformations, but we will not consider interior marked points here.

 Lagrangian intersection Floer cohomology between $L$ and $\psi^H_1(L)$, is constructed
from the $\AI$-bimodule, and  it is easy to see that
 the construction of such an $\AI$-bimodule for the smooth Floer theory in our toric orbifold case is entirely analogous to that of \cite{FOOO} and\cite{FOOO2} (particularly because we
 do not consider orbifold domains).

Instead of repeating details of the construction of \cite{FOOO} and \cite{FOOO2} to our case,
we state the main result of the constructions.

\begin{theorem}[c.f. Theorem 3.7.21\cite{FOOO}, Theorem 15.1 \cite{FOOO2}]
Let $(L,L')$ be an arbitrary relatively spin pair of compact Lagrangian smooth submanifolds.
Then the family $\{ \mathfrak{n}_{k_1,k_2} \}$ of operators
$$B_{k_1}(C(L)[1])\,\,
\WH{\otimes}_{\Lambda_0} \,\, C(L,L')
\,\,\WH{\otimes}_{\Lambda_0} \,\, B_{k_1}(C(L')[1]) \to C(L,L')
$$
for $k_1,k_2\ge 0$, define a left $(C(L),m)$, and right $(C(L'),m')$ filtered $\AI$-bimodule
structure on $C(L,L')$.
\end{theorem}
The above includes the case of clean intersection also, and in the case of $L=L'$ the maps
$\frak n_{k_1,k_2}$ is defined as $\frak n_{k_1,k_2} =  m_{k_1+k_2+1}$.

Now, let  $L, \, L'$ be weakly unobstructed. We define
$\delta_{b,b'} : C(L,L') \to C(L,L')$ by
$$
\delta_{b,b'}(x) =
\sum_{k_1,k_2} \frak n_{k_1,k_2} (b^{\otimes k_1} \otimes
x \otimes b^{\prime \otimes k_2}) =  \mathfrak{\widehat n}(e^{b},x,e^{b'}).
$$
One can check that
the equation $\delta_{b,b'}\delta_{b,b'} = 0$ holds if
the potential functions agree $PO(b) = PO(b')$.
In such a case,
Floer cohomology is defined by
$$
HF((L,b),(L',b');\Lambda_0) = \mbox{Ker}\,\delta_{b,b'}/\mbox{Im}
\,\delta_{b,b'}.
$$

\begin{prop}[c.f. Lemma 12.9 \cite{FOOO2}, section 5.3 \cite{FOOO}]
For the case of $L' = \psi(L)$, and $b' = \psi_* b$, we have
$$HF((L,b), (L',b');\Lambda)  \cong HF((L,b), (L,b); \Lambda)$$
\end{prop}
Here RHS was defined in \eqref{de:fhdefine}.
This  implies the following
theorem on non-displaceability
\begin{theorem}[c.f. Theorem G \cite{FOOO}]\label{displace1}
Assume $\delta_{b,b'}\circ \delta_{b,b'} = 0$. Let
$\psi : X \to X$ be a Hamiltonian diffeomorphism
such that $\psi(L)$ is transversal to $L'$. Then we have
$$
\# (\psi(L) \cap L')
\ge \text{\rm rank}_{\Lambda}HF((L,b),(L',b');\Lambda).
$$
\end{theorem}

\subsection{Bulk deformed Lagrangian Floer cohomology}
Now, we consider  $J$-holomorphic orbifold discs and strips, whose information gives rise
to the bulk deformation of the smooth Lagrangian Floer theory for toric orbifolds.
This is similar to the construction in section \ref{sec:bulkmoduli}, where we constructed  bulk deformed $\AI$-algebra from smooth $\AI$-algebra by considering holomophic orbi-discs.
As explained before, the bulk deformation here is a bit different from  that of \cite{FOOO3} in that we considered bulk deformation by fundamental classes of twisted sectors. But the general formalism, and algebraic structures are the same.

In fact, the construction of $\AI$-bimodule for the bulk deformed Floer theory in our case is entirely analogous to that of \cite{FOOO3} except the following issue of time dependent $\{J_t\}$- holomorphic maps, which we first explain.

Consider a transversal pair of Lagrangian submanifolds $L$ and $\psi_1(L)$, for a hamiltonian isotopy $\psi_t$ with $\phi_0=id$.
To define Lagrangian Floer cohomology between them, and to show invariance under another hamiltonian isotopies, one considers $J$-holomorphic strip of several kinds with Lagrangian boundary conditions.
In general, one takes family of $J$'s parametrized by the domain of the strip. For example, to define
the differential of the Floer complex $C(L,\phi_1(L))$,   one takes one parameter family of compatible  almost complex structures $\CJ := \{J_t\}_{t \in [0,1]}$ such that $J_0$ is the (almost) complex structure $J$ of $\bX$, and $J_1 = \phi_*(J)$, and consider $\{J_t\}$-holomorphic strips
$$\frac{\partial u}{\partial \tau} + J_t (\frac{\partial u }{\partial t}) =0$$

Now, if the domain is an orbifold strip, namely it is $\R \times [0,1]$ with interior orbifold marked points
$z_1^+,\cdots, z_l^{+}$, then it is not obvious what it means to have $\CJ$-holomorphic strip, since
for orbifold $J$-holomorphic strips  what is  actually $J$-holomorphic is its local lifts. For orbifold discs, we use  a fixed almost complex structure which is invariant under local group action, hence this does not cause any problem. But when we consider family of almost complex structures which are $t$-dependent,  the coordinate $t$ of the domain strip becomes
complicated when we consider the branch covering near the marked points.

We find that this issue actually does not cause much trouble since the lift satisfies $\CJ'$-holomorphic equation
where $\CJ'$ is a family of compatible almost complex structures of $\bX$ parametrized
by a domain. We explain it in more detail as follows. Consider an orbifold point
$z^+=(\tau_0,t_0) \in \R \times I$ with $\Z/k$ orbifold structure.
Holomorphic structure near $z^+$ is given by the coordinate $\tau + i t$
(normalized so that at $z^+$, $\tau = t = 0$), and
we consider a local neighborhood $U$ of $z^+$, and a branch covering
$br:\WT{U} \to U$. Denote the coordinate of $\WT{U}$ as $\tilde{\tau} + i \tilde{t}$
and branch covering map is given by
$$br(\tilde{\tau} + i \tilde{t}) = (\tilde{\tau} + i \tilde{t})^k.$$
Then, $t$-coordinate of $br(\tilde{\tau} + i \tilde{t})$ is its imaginary part,
$Im(br(\tilde{\tau} + i \tilde{t}))$, which is a polynomial function of
$\tilde{\tau}$ and $\tilde{t}$.  We define
$u:(\R \times [0,1], \R \times \{0\}, \R \times \{1\}) \to (\bX, L, L')$ with interior orbifold marked points $\vec{z}^+$  to be an orbifold $\CJ$-holomorphic strip
if it is $\CJ$-holomorphic away from orbifold marked points and at each orbifold marked point,
with coordinate parametrized as above the local lift $\WT{u}$ satisfies

$$\frac{\partial \WT{u}}{\partial \tilde{\tau}} + J_{Im(br(\tilde{\tau} + i \tilde{t}))}
\frac{\partial \WT{u}}{\partial \tilde{t}} =0$$
We denote $\CJ' =\{ J_{Im(br(\tilde{\tau} + i \tilde{t}))}\}$ a family of compatible
almost complex structure, parametrized by the domain $\WT{U}$.
The way to deal  with  domain dependent almost complex structure $\CJ'$ is
also standard in Floer theory, and adds no additional difficulty in the construction
of Kuranishi structures and moduli spaces. For example, already
in \cite{FO}, authors used such domain dependent case to prove Arnold conjecture.
Note that the dependence is smooth since $Im(br(\tilde{\tau} + i \tilde{t}))$ is
polynomial function.

The rest of the details to construct the bulk deformed $\AI$-bimodule is
the same as that of section 8 of \cite{FOOO3} in our case, where
they describe de Rham version of bulk deformed Lagrangian Floer homology
of the pair $(L, L')$ of Lagrangian submanifolds.
We omit the details
of the construction, and just state the results of such constructions.

Let $L(u)$ be a Lagrangian torus fiber, and  let $L'=\psi(L(u))$.
Consider the bounding cochain $(\frak b, b)$, and $(\frak b, \psi_*b)$.

\begin{prop}[c.f.\cite{FOOO3} Proposition 8.24]
Lagrangian Floer cohomology between $(L(u), (\frak b, b))$ and
$(\psi(L(u)),  (\frak b, \psi_*b))$ can be defined as in \cite{FOOO3}, and
satisfies that
$$HF((L(u), \frak b, b), (\psi(L(u)),  \frak b, \psi_*b );\Lambda) \cong
HF((L(u), \frak b, b),(L(u), \frak b, b);\Lambda)$$
\end{prop}
Here the latter has been defined in Definition \ref{defbulkfl}.

The notion of balanced, and bulk-balanced fibers can be
defined in exactly the same way as in Definition 4.11 of \cite{FOOO2}, and
Definition 3.17 of \cite{FOOO3}, and we omit the details.

The following proposition can be proved in the same way as in \cite{FOOO3}.
\begin{prop}[Proposition 3.19 \cite{FOOO3}]\label{prof:bal2}
If $L(u) \subset X$ is bulk-balanced, then $L(u)$ is
non-displaceable.
Given  a Hamiltonian diffeomorphism $\psi: X \to X$  such that
$\psi(L(u))$ is transversal to $L(u)$, then we have
\begin{equation}
\# (L(u) \cap \psi(L(u))) \geq 2^n.
\end{equation}
\end{prop}

\section{Examples}\label{sec:ex}
\subsection{Tear drop orbifold}\label{sec:tear}
We first consider a tear drop orbifold $P(1,a)$ for some positive
integer $a \geq 2$ (see Figure \ref{tearfig}). The labelled
polytope, corresponding to $P(1,a)$, is given by the interval
$$P = [-\frac{1}{a}, 1]$$
with label $a$ on the vertex $-\frac{1}{a}$.

To find an associated fan and stacky vectors, recall that the
polytope $P$ is defined by $\langle x, \bb_j \rangle \geq
\lambda_j$ for $j=1,2$. In this case we have the lattice $N = \Z$,
and
 $$\bb_1=a, \bb_2 = -1, \lambda_1=\lambda_2 = -1. $$
 The stacky vectors  $\bb_1 $ and
$\bb_2 $ generate two $1$-dimensional cones $\sigma_1 =R_{\geq
0}$, $\sigma_2 = R_{\leq 0}$ of the fan $\Sigma$.

$P(1,a)$ is given as the quotient orbifold $\C^2 \setminus \{0\} /
 \C^* $ where
  $\C^*$ acts by $t\cdot (z_1, z_2) = (t z_1, t^a z_2)$. The unique orbifold point is $[0,1]$
with isotropy group $\Z/a$. Thus inertia components are labelled
by $\Z/a$.
$$Box' = \{ \nu_i \mid \nu_i = i \in \Z/a \;\; \textrm{for}\; i= 1,\cdots, a-1\}.$$

We take $u \in (-\frac{1}{a}, 1)$, and consider the Lagrangian circle fiber
$L(u)$.

The classification theorem (corollary \ref{smoothdisc}) tells us that there are two
smooth holomorphic discs with Maslov index two of class $\beta_1$
and $\beta_2$, corresponding to the stacky vectors $\bb_1, \bb_2$.
Explicitly, the holomorphic disc $w_1:(D^2,\partial D^2) \to
(P(1,a), L(u))$ of class $\beta_1$ is given
 by $w_1(z) = [cz, 1]$ for some constant $c$ to make $w_1(\partial D^2) \subset L(u)$ (up to $Aut(D^2)$).
The image of $w_1$ is an $a$-fold  uniformizing cover of a neighborhood of  the orbifold point $[0,1]$.
Holomorphic discs of $\beta_2$ classes are $w_2(z) = [c,z]$.

The smooth potential function of $P(1,a)$ thus has two terms corresponding these two smooth discs:
$$PO(b)= PO(b)_0 = T^{1-u}y^{-1} + T^{1+au}y^a.$$

To find a fiber $L(u)$ with holonomy whose smooth Floer cohomology is non-vanishing, we
find critical points of $PO(b)$. If the $T$-exponents of the two terms of $PO(b)$ is not equal,
then $PO(b)$ do not have non-trivial critical point. (since $y = e^x$, it cannot be zero.)

The areas of two smooth discs are the same, or, $1-u = 1+au$, exactly when $u=0$.
Notice that $u=0$ is not at the center of the polytope $P$, since the smooth disc
of class $\beta_1$, wraps around the orbifold point $a$ times.

In this case, the critical point equation becomes
$$-y^{-2} + a y^{a-1}=0 $$
or,  $y^{a+1} = 1/a$ which has solutions
$$y = \frac{1}{\sqrt[a+1]{a}} exp(\frac{2\pi k i}{ a+1}) \;\;\textrm{for}\;\; k=0,\cdots, a.$$
Thus the fiber $L(0)$ with flat line bundle of (non-unitary) holonomy as one of the above,   has non-trivial
Floer cohomology (see Figure \ref{tearfig}).

Now, we consider bulk deformations by orbi-discs.
From the classification theorem, we have $a-1$ holomorphic orbi-discs
corresponding to the elements of $Box'$(see Corollary \ref{cor:orbidisc}).
These correspond to holomorphic orbi-discs, wrapping around the orbifold points
$1,\cdots, a-1$ times.

The leading order bulk potential  $PO^{\frak b}_{orb, 0}$ can be explicitly written as
$$PO^{\frak b}_{orb, 0}= T^{1-u}y^{-1} + T^{1+au}y^a + \sum_k T^{k(u+1/a)}{\frak b}_{\nu_k} y^k.$$

In this example, we can set ${\frak b}_{\nu_k} = 0$ for $k=2, 3, \cdots, a-1$ as ${\frak b}_{\nu_1}$ is enough
in this case.

Note that $1 + au >  1/a + u$ for all $a \geq 2$, where the former is
the area of the smooth disc, and the latter is the area of the orbi-disc for $\nu_1$.
Hence the consider the leading term equation, it is enough to
compare $T$-exponent  $(1-u)$ of the first term, with that of the bulk term $(u + 1/a)$.

If $(1- u) > (u + 1/a)$, or equivalently $ \frac{1}{2}(1- \frac{1}{a}) > u$,
then we take
$${\frak b}_{\nu_1} =  T^{(1-u) - (\frac{1}{a} + u)}$$
so that
$$T^{1-u} =  T^{\frac{1}{a} + u}b_{\nu_1,0}.$$
Then, leading term equation (with $S_1 = 1-u$) is
$$\frac{\partial}{\partial y} ( y^{-1} + y^1)=0$$
or $y = \pm 1$. Thus,  the fiber $L(u)$ with
$ \frac{1}{2}(1- \frac{1}{a}) > u$  has non-trivial Floer homology, and therefore non-displaceable by
any Hamiltonian isotopy of $P(1,a)$. The result holds even for $u = \frac{1}{2}(1- \frac{1}{a})$
by the standard limit argument, and thus, exactly half of the interval $ [-\frac{1}{a}, 1]$
containing the image of the orbifold point, corresponds to non-displaceable circles in $P(1,a)$
(see the Figure \ref{tearfig} for the region of non-displaceability).

\subsection{Weighted projective space $\BP (1,a_1,\cdots, a_n)$ }
\label{sec:wps}
We consider the smooth Floer homology of weighted projective space
$\BP (1,a_1,\cdots, a_n)$ for positive integers $a_i \in \N$, $i=1,\cdots, n$.
The bulk deformed theories are much more complicated, and we will discuss several
examples in more detail in later subsections.

The polytope $P$ for $\BP (1,a_1,\cdots, a_n)$ is defined by
\begin{equation}\label{polyweight}
P =\{ (x_1,\cdots, x_n)\mid x_j+1 \geq 0, j=1,\cdots, n, \;\;\;
-(\sum_{j=1}^n a_j x_j) + 1 \geq 0\}
\end{equation}
Here $N = \Z^n$ and we take $\bb_0 = (-a_1,\cdots, -a_n)$ and $\bb_i = e_i$ for $i=1,\cdots, n$.
Then $\BP (1,a_1,\cdots, a_n)$ is obtained as a quotient orbifold of
$\C^{n+1} \setminus \{0\} / S^1$ where circle acts with weight $(1,a_1,\cdots, a_n)$.

There are $n+1$ smooth holomorphic discs corresponding to stacky vectors $\bb_0,\cdots,\bb_n$ whose
 homology classes are denoted as $\beta_0,\cdots, \beta_n$.

Thus the smooth potential $PO(b)=PO(b)_0$ is
$$PO(b)= T^{1-\langle \vec{a}, u \rangle}\frac{1}{y_1^{a_1}\cdots y_n^{a_n}}
+ T^{u_1+1}y_1 + \cdots T^{u_n+1}y_n.$$

The only non-trivial critical points can occur only when all the exponents of $T$ are the same, and
hence $u=0$. i.e. the central fiber $u=0$ admits $n+1$ smooth holomorphic discs of Maslov index two with
the same area $$\ell_0(u) = \cdots = \ell_n(u) = 1.$$
This is an analog of Clifford torus in projective spaces.

In this case  the critical point equations $\frac{\partial}{\partial y_i} PO(b)_{u=0}=0$ for all $i$
have a solution:
$\frac{1}{y_1^n\cdots y_n^{a_n}}= \frac{y_i}{a_i}$ or $y_i = a_i \lambda$
with $\lambda$ given as $\lambda = (a_1^{a_1} \cdots a_n^{a_n})^{1/(1-a_1 - \cdots -a_n)}$.

\begin{prop}\label{wp}
The central fiber $L(0)$ in the above weighted projective space $\BP (1,a_1,\cdots, a_n)$
with holonomy $a_i\lambda$ as above, has non-trivial smooth Floer cohomology.
Thus $L(0)$ is non-displaceable by any Hamiltonian isotopy.
\end{prop}

\subsection{Bulk Floer homology for $\BP(1,a,a)$}
Consider the space $\BP(1,a,a)$  for a positive integer $a \geq 2$.
We explain how to use bulk deformation to detect non-displaceable torus fibers.

The labelled polytope  is the same as  given  in \ref{polyweight},
where the facet corresponding to $\bb_0$ (to which $\bb_0$ is normal) has a label $a$ on it.
The whole divisor corresponding to $\bb_0$ has an isotropy group $\Z/a$.
It is not hard to check that
$$Box' = \{ \nu_k:= \frac{k}{a}\bb_0 \mid k=1, \cdots, a-1\}.$$
The smooth potential $PO(b)$ is given by
$$PO(b)= T^{1-au_1-au_2}\frac{1}{y_1^ay_2^a} + T^{u_1+1}y_1+T^{u_2+1}y_2.$$
We may take $$\frak{b}_{\nu_1} = T^{\alpha}, \; \frak{b}_{\nu_k} =0 \; \textrm{for}\; k \neq 1,$$
for some $\alpha >0$.
Then, leading order bulk potential (with the above choice of bulk deformation) becomes
$$PO^{\frak b}_{orb,0} = PO(b) + T^{\alpha} T^{\frac{1}{a}-u_1-u_2}\frac{1}{y_1y_2}.$$

Now, we try to find $\alpha$ such that  leading term equation of $PO^{\frak b}_{orb,0}$(which depends on $\alpha$) has
a non-trivial solution.  The idea is that on a given energy level, say $S_l$, if  the vectors $\bb$'s corresponding to
energy $S_l$ spans $d(l)$ dimensional space, then, we need at least $d(l)+1$ number
of vectors $\bb$'s to have a non-trivial solution of the leading term equation of level $S_l$.
In our case, we need at least three vectors of $\bb$'s correspond to the minimal energy level $S_1$.

As the area of the basic disc corresponding to $\bb_0$ is $a$ times bigger than that of $\nu_1$,
and as $PO^{\frak b}_{orb,0}$ has only four terms,
 the three terms excluding that of $\bb_0$ should
have the same $T$-exponent in order to have a non-trivial solution of leading term equation,

Thus basic discs corresponding to
$\bb_1$ and $\bb_2$ should have the same area, which then equals to the sum of $\alpha$
and the area of the $\nu_1$ orbi-disc.
$$\ell_1 = \ell_2 = \alpha + \ell_{\nu_1}.$$
This implies that we have
$$1+ u_1 =1+  u_2 = \frac{1}{a} - u_1 -u_2 + \alpha,$$
which gives
\begin{equation}\label{1aa}
u_1 = u_2,  u_1 = \frac{1}{3}( \alpha + \frac{1}{a}-1).
\end{equation}
Also, we need to require that the area of $\bb_0$ is bigger than that of $\bb_1$ or $\bb_2$.
Thus, we have $$1-au_1 -au_2 > 1+u_1.$$
 With the condition that $u_1=u_2$, we have
$u_1<0, u_2 <0$.
Since $\alpha>0$,  $$u_1=u_2 > \frac{1-a}{3a}.$$
Thus, for $(u_1,u_2)$ lying on the line segment, connecting $(0,0)$ and $(\frac{1-a}{3a}, \frac{1-a}{3a})$

\begin{figure}[h]
\begin{center}
\includegraphics[height=1.9in]{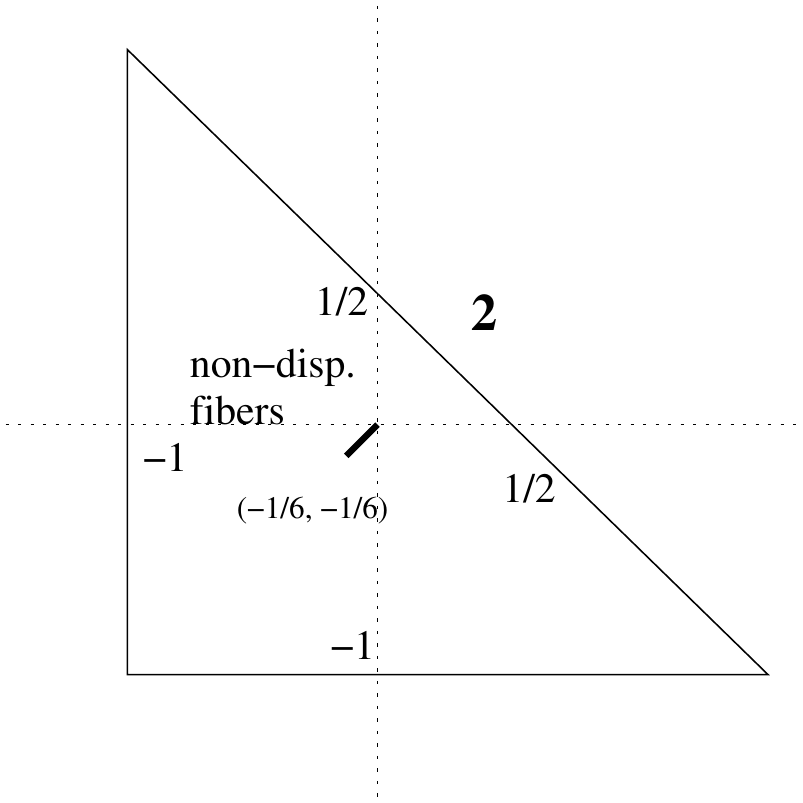}
\caption{Orbifold $\BP(1,2,2)$}
\label{orbfig1aa}
\end{center}
\end{figure}

Indeed for a fixed $(u_1, u_2)$ in the line segment above,
we choose $\alpha$ to satisfy \eqref{1aa}, then
the leading term equation (with the minimal energy $S_1 = 1+ u_1$) is nothing but
$$\frac{1}{y_1y_2} + y_1 + y_2 =0.$$
It is easy to check that this equation has a non-trivial critical points, which
describes the (non-unitary) holonomies to be put on the Lagrangian torus fiber at $(u_1,u_2)$
so that the resulting Floer cohomology is non-trivial and isomorphic to the singular cohomology
of the torus from the Theorem \ref{equivLTE} and Theorem \ref{computefloer2}

\subsection{Bulk Floer homology for $\BP(1,1,a)$}
Now, we discuss the case of  $\BP(1,1,a)$ for a positive integer $a \geq 3$.
The corresponding  moment polytope is shown in Figure \ref{orbi11a}.
The elements of $Box'$ is
$$Box' = \{ \nu_k:= \frac{k}{a}\bb_0 + \frac{k}{a} \bb_1 = (0,-k) \mid k=1, \cdots, a-1 \}.$$
\begin{figure}[h]
\begin{center}
\includegraphics[height=1.9in]{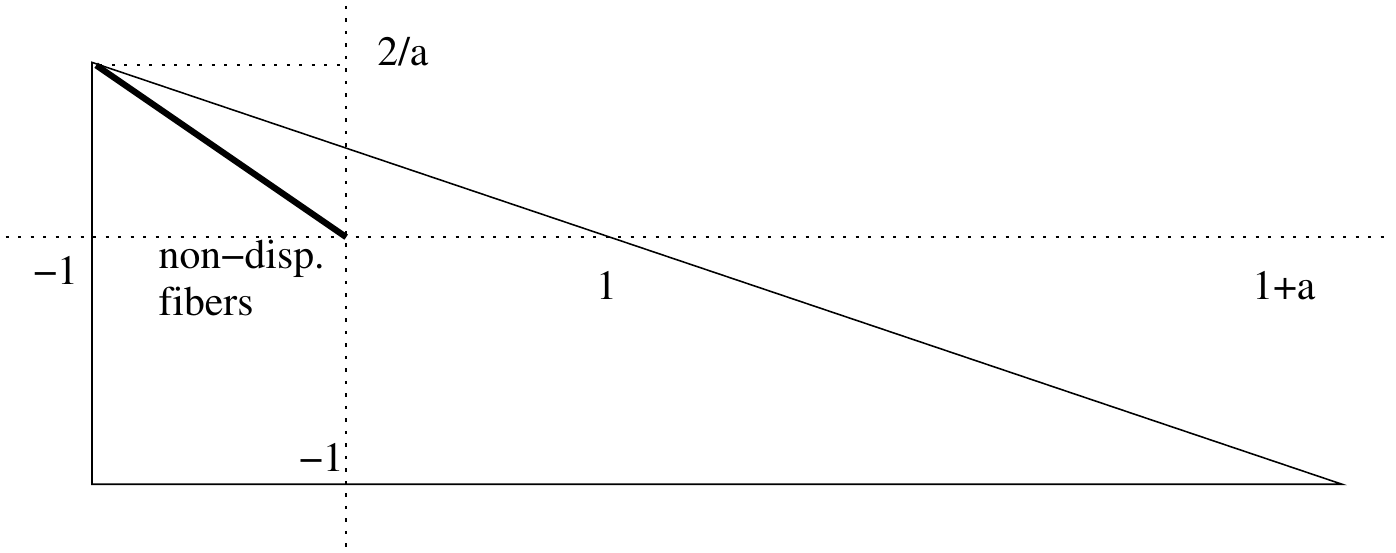}
\caption{Orbifold $\BP(1,1,a)$}
\label{orbi11a}
\end{center}
\end{figure}

The smooth potential $PO(b)$ is given by
$$PO(b)= T^{1-u_1-au_2}\frac{1}{y_1y_2^a} + T^{u_1+1}y_1+T^{u_2+1}y_2.$$
We take $$\frak{b}_{\nu_1} = T^{\alpha}, \; \frak{b}_{\nu_k} =0 \; \textrm{for}\; k \neq 1,$$
for some $\alpha >0$. Since
$$\ell_{\nu_1} = \frac{1}{a}\ell_0 + \frac{1}{a} \ell_1 = \frac{2}{a} - u_2$$
Then, the leading order potential with the above choice of bulk deformations become:
$$PO^{\frak b}_{orb,0} = PO(b) + T^{\alpha} T^{\frac{2}{a}-u_2}\frac{1}{y_2}.$$

And we try to find $\alpha$ which makes $PO^{\frak b}_{orb,0}$ to have
a solution in its leading term equation.
Note that $\nu_1$ and $\bb_2$ are in the opposite direction.

As in the previous example, we need three terms of $PO^{\frak b}_{orb,0}$ to have
the same $T$-exponent. The case when $\ell_0=\ell_1=\ell_2$, we get a solution in
the Proposition \ref{wp}. One can check that the remaining case with
non-trivial solution of leading term equation is the case that $\ell_0 = \ell_1 = \ell_{\nu} + \alpha$.
( other cases, contains both
$\nu_1$ and $\bb_2$ and the corresponding leading term equation
do not have a solution, as $\nu_1$, and $\bb_2$ are linearly dependent).

This implies that we have
$$1+ u_1  = 1 - u_1 -a u_2 = \frac{2}{a} - u_2 + \alpha,$$
which gives
\begin{equation}\label{11a}
u_2 = -\frac{2}{a}u_1,  u_1 =-1 + \alpha \frac{a}{a-2}.
\end{equation}
Also, we need to require that $\ell_1 \leq \ell_2$. Thus, $u_1 \leq u_2$.
This implies that $u_1\leq 0$ and $u_2 \geq 0$.
Thus, $(u_1,u_2)$ lies in the interior of the line segment connecting $(-1, 2/a)$ and $(0,0)$
as drawn in Figure \ref{orbi11a}.
It is not hard to check that the corresponding leading term equation has
a solution in such a case.

\begin{remark}\label{remyet}
It is shown in \cite{WW} Example 4.9 that in the case of $a=2$, the analogous line segment
is also the location of non-displaceable torus fibers. But unfortunately, we do not know how to
prove it using our methods. This is because that
computations, orbifold bulk deformation term with energy zero is needed for $a=2$, which
is not possible due to infinite sums in the definition of bulk deformation. We leave it
for future research.
\end{remark}

\subsection{Bulk Floer homology for $\BP(1,3,5)$}
The example $\BP(1,3,5)$ has been found to be very interesting example recently, see \cite{M}, and also in  \cite{WW} and \cite{ABM}. We show that the torus fibers which are inverse images of points in the colored region in the polytope (in Figure \ref{orbfig135}) are non-displaceable by Hamiltonian isotopy in our methods.
\begin{figure}[h]
\begin{center}
\includegraphics[height=2.5in]{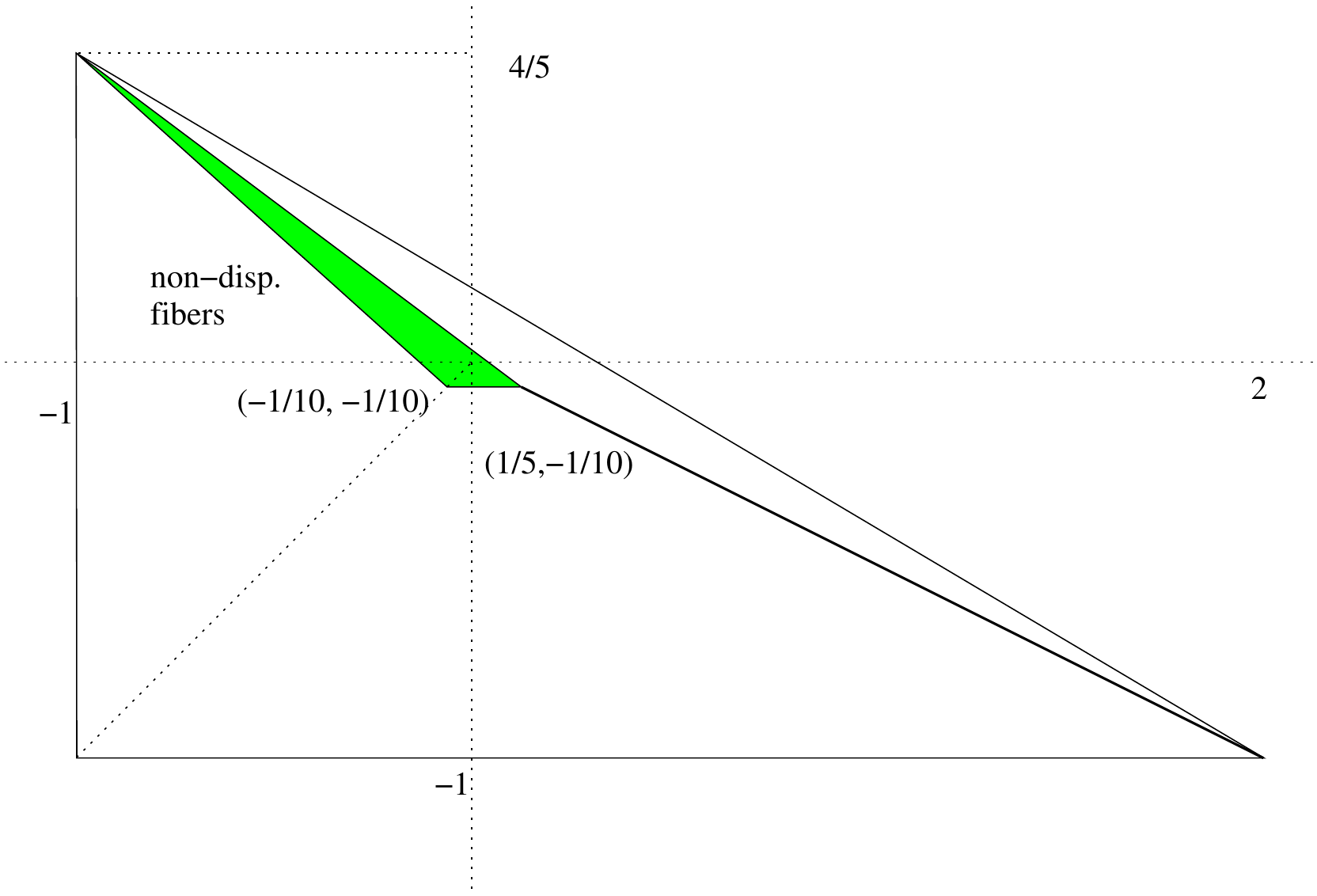}
\caption{Orbifold $\BP(1,3,5)$}
\label{orbfig135}
\end{center}
\end{figure}

As the shape indicates, we need a little detailed analysis on comparing the sizes of the areas of
holomorphic discs and orbi-discs.

First we identify elements of $Box'$.
We find that from $\bb_0$ and $\bb_1$,
\begin{eqnarray*}
\nu_1 &= \frac{1}{5} \bb_0 + \frac{3}{5} \bb_1 =& (0,-1) \\
\nu_2 &= \frac{2}{5} \bb_0 + \frac{1}{5} \bb_1 =& (-1,-2) \\
\nu_3 &= \frac{3}{5} \bb_0 + \frac{4}{5} \bb_1 =&(-1,-3) \\
\nu_4 &= \frac{4}{5} \bb_0 + \frac{2}{5} \bb_1 =& (-2,-4)
\end{eqnarray*}
From $\bb_0$ and $\bb_2$, we have
\begin{eqnarray*}
\nu_5 &= \frac{1}{3} \bb_0 + \frac{2}{3} \bb_1 =& (-1,-1) \\
\nu_6 &= \frac{2}{3} \bb_0 + \frac{1}{3} \bb_1 =& (-2,-3)
\end{eqnarray*}
The areas of holomorphic discs and orbi-discs are
$$ \ell_0 = 1-3u_1-5u_2, \; \ell_1 = 1+u_1, \; \ell_2 = 1+ u_2$$
$$\ell_{\nu_1} = \frac{1}{5} \ell_0 + \frac{3}{5} \ell_1 = \frac{4}{5} - u_2, \; \ell_{\nu_2} = \frac{3}{5} - u_1-2u_2$$
The areas $\ell_{\nu_3},\cdots, \ell_{\nu_6}$ can be computed similarly.

Near the vertex $(-1, 4/5)$ of the moment triangle in Figure \ref{orbfig135},
the areas of the following (orbi)discs (depending on the position $u \in P$),
$$\ell_0, \ell_1, \ell_{\nu_1}, \ell_{\nu_2}, \ell_{\nu_3}, \ell_{\nu_4}$$
are smaller than others, and could give relevant terms in the leading term equation.

As it is two dimensional, we would like to have triple of them to have same energy $S_1$.
Although the symplectic area $\ell_0, \ell_1$ is already fixed, we could add bulk deformation
term $\frak b_{\nu_i}$ to  $\ell_{\nu_i}$ suitably to increase the energy level.
Thus if $\ell_0$ is not equal to $\ell_1$, we need two other orbi-discs to make the triple,
and for this we need them to have smaller symplectic areas.

More precisely, we proceed as follows.
First, we consider the region where $\ell_0$ or  $\ell_1$ is smaller than $\ell_2$.
This implies that
$$\ell_0 < \ell_2  \Rightarrow u_2> -u_1/2,\;\;\;\;\, \ell_1 < \ell_2 \Rightarrow u_2 > u_1.$$

We divide it further into three cases:
\begin{enumerate}
\item  $\ell_0 < \ell_1$, or $u_1 > -4u_2/5$: to have at least three terms of least energy, we need
\begin{equation}
\ell_{\nu_1} < \ell_0 \; \textrm{and}\; \ell_{\nu_2} < \ell_0.
\end{equation}
The first inequality gives  $3u_1+ 4u_2 < 1/5$
and the second inequality gives $2u_1 + 3u_2 < 2/5$.
If this happens, we can add bulk deformation
$$\frak{b}_{\nu_1} =T^{\ell_0 - \ell_{\nu_1}}, \frak{b}_{\nu_2} =T^{\ell_0 - \ell_{\nu_2}}.$$
which will make $\bb_0, \bb_{\nu_1}, \bb_{\nu_2}$ to contribute to the leading term
equation of $PO^{\frak b}_{orb,0}$ of the same energy $S_1$.
We note that the first inequality implies the second inequality for the points in $P$, hence,
for the region bounded by
$$u_1 > -\frac{4u_2}{5},  \;\; 3u_1+ 4u_2 < \frac{1}{5} ,\;\;u_2> - \frac{u_1}{2},$$
we can choose
bulk deformation as above, so that the corresponding leading term equation has a solution.

\item  $\ell_0>\ell_1$ or $u_1 < -4u_2/5$:
 to have at least three terms of least energy,
we need
\begin{equation}
\ell_{\nu_1} < \ell_1 \; \textrm{and}\; \ell_{\nu_2} < \ell_1.
\end{equation}
Both equation translates to the inequality $u_1 + u_2 > -1/5$.
Thus, in the region
$$u_1 < -\frac{4u_2}{5}, \;\;\; u_1+ u_2 > \frak{1}{5}, \;\;\; u_2>u_1,$$
we can choose
bulk deformation as
$$\frak{b}_{\nu_1} =T^{\ell_1 - \ell_{\nu_1}}, \frak{b}_{\nu_2} =T^{\ell_1 - \ell_{\nu_2}},$$
 so that the corresponding leading term equation has a solution.

\item $\ell_0 =\ell_1$: similarly, we obtain that
the line segment $u_1 = -\frac{4u_2}{5}, u_1 <0$ supports bulk
deformation, whose leading term equation has a solution, which
we leave as an exercise.
\end{enumerate}
The above do not cover the whole colored region of the Figure \ref{orbfig135}.
The rest of the region, which is not covered is  the triangle $\Delta$ formed by three points $$(-1/10,-1/10),
(0,0), (1/5,-1/10).$$
For this region, the leading term equation involves two equations of two energy level $S_1$, $S_2$.
Note that the vectors $\bb_2 = (0,1) $ and $\bb_{\nu_1}= (0,-1)$ are opposite to
each other.  So in $\Delta$, $\ell_2$ is smaller than $\ell_1$ and $\ell_0$, and hence
we take
$$\frak{b}_{\nu_1} = T^{\ell_2 - \ell_{\nu_1}} = T^{2u_2 + \frac{1}{5}}.$$
This makes the terms corresponding to $\bb_1$ and $\bb_{\nu_1}$
contribute to the leading term equation of energy level $S_1 = \ell_2$.
Now, for the next level $S_2$, we have terms from
$\bb_1$, and $\bb_{\nu_2}$,
we again have a solution for $S_2$ energy level leading term equation.
We leave the detailed check tot he readers.

We remark that for the line segment from $(1/5, -1/10)$ to $(2,-1)$,
it is known by \cite{WW} to be non-displaceable, but our methods cannot
prove it yet as in Remark\ref{remyet}.

\subsection{Polytope with nontrivial integer labels}
\begin{prop}\label{prop:allnon}
If $P$ is compact rational simple polytope with $m$ facets,
and if the integer labels, $c_1,\cdots, c_m$ for facets satisfy
$$ c_i \geq 2 \;\; \textrm{for all} \;\; i=1,\cdots,m,$$
then for any $u \in Int(P)$, $L(u)$ is
non-displaceable.
\end{prop}
\begin{proof}
The first proof is due to Kaoru Ono, who have provided this alternative proof after  the first authors
gave a talk on this paper and this proposition.

Here is the first proof. As all facets have non-trivial integer labels, points in any toric divisors are not smooth points. Take a torus fiber $L(u)$ for any interior point $u \in Int(P)$. If $L(u)$ is displaceable
by Hamiltonian isotopy $\psi^H_1$, i.e. $L(u) \cap \psi^H_1(L(u))  =\emptyset$, then we can modify
$H$ so that its support lies in $Int(P)$, and satisfy the above displacing property as Hamiltonian isotopy sends smooth points to smooth points. But this is a contradiction, since  Lagrangian torus fibers in $(\C^*)^n$ are not displaceable by any compactly supported Hamiltonian isotopy.  As a symplectic manifold, the inverse image $L$ of $Int (P)$ is symplectically embedded in
$T^* T^n$.  Then a $T^n$-orbit is considered as the graph of a $T^n$-invariant $1$-form $\eta$ on $T^n$,
which is closed, embedded in $T^* T^n$. By a symplectomorphism, which comes from the fiberwise
addition of $-\eta$, we may assume that the $T^n$-orbit is the zero section of $T^* T^n$.
The non-displaceability of the zero section in the cotangent bundle is a well-known.

The second proof is by using the bulk deformation. Let $\bb_i$ be a stacky vector corresponding to
the $i$-th facet. We take $\nu_i$ to be the minimal integral vector in the direction of $\bb_i$
such that $\bb_i = c_i \nu_i$. We consider bulk deformations $\frak b_{\nu_i} = -c_iT^{\ell_i - \ell_{\nu_i}}$ for each $i=1,\cdots, m$. Then corresponding leading term potential $PO^{\frak b}_{orb,0}$
becomes
\begin{equation}\label{summ}
\sum_{i=1}^m T^{\ell_i}( y^{\bb_i} - c_iy^{\bb_{\nu_i}}),
\end{equation}
since each contribution of $\frak b_{\nu_i}$
is chosen to match with  the term of the potential corresponding to $\bb_i$.
For generic $u$, we may assume that all the areas $\ell_i$ are distinct.
Then, the leading term equation is the critical point equation of each summand of \ref{summ}
up to dimension of $P$. By denoting $y_l = y^{\bb_{\nu_i}}$, the summand equals
$y_l^{c_i} -c_i y_l$, and clearly has a non-trivial critical point $y_l = 1$.
This shows that generic $u \in Int(P)$ is non-displaceable. But  by the standard limit argument, this implies that $L(u)$ is
non-displaceable for all $u \in Int(P)$.
\end{proof}

\section{Appendix: Preliminaries on orbifold maps}
In this appendix, we recall definitions regarding maps between
orbifolds following \cite{CR}, \cite{CR2} with the condition that
the domain orbifold may have a nontrivial smooth boundary.

 Let ${\bX}$ be a differentiable ($C^{\infty}$) orbifold with
boundary and let $X$ be its underlying topological space. In
applications we will often deal with the case when $X$ is an
 orbifold Riemann surface with smooth boundary (i.e. orbifold singularity lies in the interior).
%  We refer readers to \cite{CP} for the general case with singularity at the boundary.

An uniformizing system for an open connected set $U \subset X$ is
a triple
 $(V, G, \pi)$ where $V$ is a smooth connected manifold
with boundary $\partial V$ (which may be empty), $G$ is a finite group acting smoothly on $V$
 (preserving $\partial V$) , and $\pi: V \to U$ is a continuous
map that induces a homeomorphism between $V/G$ and $U$.
 The orbifold analogue of inclusion of open sets in manifolds, is the notion of injection of
 uniformizing charts.
 \begin{definition} Let $i: U \hookrightarrow U^{\prime}$ be an inclusion of open sets uniformized by
 $(V,G, \pi)$ and $(V^{\prime},G^{\prime} , \pi^{\prime})$ respectively. An injection
 $(\phi, \rho):(V,G, \pi) \to (V^{\prime},G^{\prime} , \pi^{\prime})$ consists of
 an injective group homomorphism $\rho: G \to G^{\prime}$ and a $\rho$-equivariant
 open embedding $\phi: V \to V^{\prime}$ such that
 \begin{enumerate}
 \item $i \circ \pi = \pi^{\prime} \circ \phi$ and
 \item $\rho$ induces an isomorphism $\ker G \to \ker G^{\prime}$ where
  $\ker G := \{g\in G: g\cdot x = x \;{\rm for\; all\;} x\in V  \}$.
 \end{enumerate}
 \end{definition}
 If $\ker G$ is trivial for every uniformizing system, the
 orbifold is called effective or reduced.

 An injection $(\phi, \rho)$ is an isomorphism of uniformizing systems if there exists
  an inverse injection.
 An important fact is that given an automorphism $(\phi, \rho)$ of an uniformizing
  system $(V,G, \pi)$ of open set $U \subset X$
 in a $C^{\infty}$ orbifold ${\bX}$ (with boundary), there exists an element $g \in G$ such
  that $\phi(x) = gx$ and $\rho(h) = g h g^{-1}$, see Lemma 2.11 of \cite{MM}. This
  correspondence is one-to-one if $\ker G$ is trivial.

% remark: restrict to interior, get g, use continuity for boundary

 \begin{definition} A compatible cover of an open set $Y$ in an orbifold ${\bX}$ is an
 open cover $\mathcal{U}$ of $Y$ together with an uniformizing system $(V,G, \pi)$ for each
 $U \in \mathcal{U}$ and a collection of injections such that
 \begin{enumerate}
 \item If $U \subset U^{\prime}$ then there exists an injection
  $ (V,G, \pi) \to (V^{\prime},G^{\prime} , \pi^{\prime}) $.
 \item For every $p \in U_1 \bigcap U_2$, where $U_1, U_2 \in \mathcal{U}$, there exists
 an $U \in \mathcal{U}$ such that $p \in U \subset U_1 \bigcap U_2$.
 \end{enumerate}
 \end{definition}

 \begin{definition} Let ${\bX}$ and ${\bX}^{\prime}$ be orbifolds, possibly with boundary.
Suppose $U \subset X,\, U^{\prime} \subset X^{\prime}$ are open sets uniformized by
$(V, G, \pi)$ and
 $(V^{\prime}, G^{\prime}, \pi^{\prime})$  respectively.
  Given a continuous map $f: U \to U^{\prime}$,
 a $C^l$ lift of $f$ is a $C^l$ map $\tilde{f}: V \to V^{\prime}$ satisfying
 \begin{enumerate}
 \item $\pi^{\prime} \circ \tilde{f} = f \circ \pi$.
 \item Given any $g\in G$ there exists a $g^{\prime} \in G^{\prime}$ such that
       $\tilde{f} (gx) = g^{\prime} \tilde{f}(x)$  for all $x\in V$.
 \end{enumerate}
 \end{definition}
Note that the correspondence $g \to g^{\prime}$ is not required to
be a group homomorphism.

 \begin{definition} Two lifts $\tilde{f}_i: (V_i, G_i, \pi_i) \to
 (V_i^{\prime}, G_i^{\prime}, \pi_i^{\prime})$, $i=1,2$, are isomorphic
 if there are isomorphisms $(\phi, \rho): (V_1, G_1, \pi_1) \to (V_2, G_2, \pi_2)  $
 and $(\phi^{\prime}, \rho^{\prime}): (V_1^{\prime}, G_1^{\prime}, \pi_1^{\prime})
 \to (V_2^{\prime}, G_2^{\prime}, \pi_2^{\prime})$ such that
 $\phi^{\prime} \circ \tilde{f}_1 = \tilde{f}_2 \circ \phi$.
 \end{definition}

  Let $\tilde{f}: (V, G, \pi) \to
 (V^{\prime}, G^{\prime}, \pi^{\prime})$ be a $C^l$ lift of $f: U \to
 U^{\prime}$.
 Let $W,\, W^{\prime}$ be open sets such that $W\subset U$ and $ f(W) \subset W^{\prime} \subset  U^{\prime}$.
  Then $\tilde{f}$ naturally induces a unique isomorphism class of lift for $f: W \to
W^{\prime}$.

 \begin{definition}
  Two lifts $\tilde{f}_i: (V_i, G_i, \pi_i) \to
 (V_i^{\prime}, G_i^{\prime}, \pi_i^{\prime})$, $i=1,2$, of $f: X \to X^{\prime}$
  over open sets $U_1$ and $U_2$, are said to be equivalent
  at $p \in U_1 \bigcap U_2$  if they induce isomorphic lifts of $f: U \to U^{\prime}$
  for some open sets $U$ containing $p$ and $U^{\prime}$ containing $f(p)$.
 \end{definition}

 \begin{definition} A local $C^l$ lift of
$f: X \to X^{\prime}$ at a point $p \in X$ is a $C^l$ lift $\tilde{f}_p: V_p \to V_{f(p)}^{\prime}$,
for some uniformizing systems $(V_p, G_p, \pi_p)$ and
 $(V_{f(p)}^{\prime}, G^{\prime}_{f(p)}, \pi^{\prime}_{f(p)})$ on open sets
 containing $p$ and $f(p)$ respectively.
 \end{definition}

\begin{definition} Let ${\bX}$ and ${\bX}^{\prime}$ be orbifolds, possibly with
 boundary. Given a continuous map
$f: X \to X^{\prime}$,  a $C^l$ lift
 ${\tilde f}: {\bX} \to {\bX}^{\prime}$ of $f$ is a choice of a local $C^l$ lift
$\tilde{f}_p : V_p \to V_{f(p)}^{\prime}$  for each point $p$,
such that $\tilde{f}_p$ is equivalent to $\tilde{f}_q$ for each $q \in U_p$.
\end{definition}

 \begin{example}\label{ex C-inftymap} Consider the orbifold $\mathbb{C}/\mathbb{Z}_2$ where
 $\mathbb{Z}_2$ acts by reflection about the origin.
Consider the holomorphic coordinates $z$ on $\mathbb{C}$ and $w=
z^2$ on $\mathbb{C}/\mathbb{Z}_2$. Regard $S^1$ as $ \R/2\pi\Z $.
 Take the map $ f: S^1 \to \mathbb{C}/\mathbb{Z}_2$ defined by
  $ w \circ f(\theta) = e^{i\theta} $.
 Consider the covering of $S^1$ by the open sets $U_1 =(0,2\pi)$
  and $U_2 =(-\pi, \pi)$.
  The lifts $\tilde{f}_j: U_j \to \mathbb{C}$ given by
  \begin{equation}
  z \circ \tilde{f}_j(\theta) = e^{i\theta/2} \; {\rm for \, } j=1,2,
 \end{equation}
  define a $C^{\infty}$ lift of $f$.

  Note that not every continuous map of underlying spaces admits a lift.
  As an example,  the map $ h : \mathbb{C} \to \mathbb{C}/\mathbb{Z}_2$
 defined by $ w \circ h (t) = t $, does not admit even a $C^0$ lift
  near the origin.
\end{example}

\begin{definition} Two lifts $\{ \tilde{f}_{p,i}:  (V_{p,i}, G_{p,i}, \pi_{p,i}) \to
 (V_{f(p),i}^{\prime}, G^{\prime}_{p,i}, \pi^{\prime}_{f(p),i}) \}, \; i=1,2,$
 of $f$  are said to be equivalent
 if for each $p \in X$, $\tilde{f}_{p,1}$ and $\tilde{f}_{p,2}$ are equivalent at $p$.
\end{definition}

\begin{definition} A $C^l$ map of orbifolds ${\bf f}: {\bX} \to {\bX}^{\prime}$ is a
continuous map $f: X \to X^{\prime}$ of underlying spaces together
with the equivalence class of a $C^l$ lift $\tilde{f}$ of $f$.
\end{definition}

 Now we recall the crucial notion of a good map\cite{CR}.
 %Intuitively a good map is an orbifold map such that every local
 %lift is equivariant.
 Chen and Ruan used the notion of compatible system to describe a good
 map. A compatible system roughly consists of compatible covers of the domain and range of
 the map by uniformizing charts, choice of lifts on each chart and some algebraic
 data for injections of charts that encode how the lifts fit
 together.  This enables one to define the  pull-back of
 an orbifold vector bundle with respect to a good map.
 The notion of good map is very closely related to
 the notions of strong map \cite{MP} and orbifold morphism \cite{ALR}.

\begin{definition}\label{compsys}
Let ${\bf f}:{\bX} \to{\bX}^{\prime}$ be a $C^l$ map between orbifolds with
boundary whose underlying continuous map is denoted by $f$. Suppose $\mathcal{U}$
and  $\mathcal{U}^{\prime}$ are compatible covers of $X$ and an open set containing
$f(X)$ respectively satisfying the following conditions
\begin{enumerate}
\item There is a bijection between $\mathcal{U}$ and
$\mathcal{U}^{\prime}$ given by $U \leftrightarrow U^{\prime}$,
such that $f(U) \subset U^{\prime}$, and  $U_2 \subset U_1$
implies $U_2^{\prime} \subset U_1^{\prime}$.

\item There exists a collection of local
 $C^l$ lifts $\{\tilde{f}_{UU^{\prime}}: (V,G,\pi) \to (V^{\prime}, G^{\prime}, \pi^{\prime} ) \}$
of $f$ and
 an assignment of an injection $\lambda(i):  (V_2^{\prime}, G_2^{\prime},
 \pi_2^{\prime} ) \to (V_1^{\prime}, G_1^{\prime}, \pi_1^{\prime} ) $
to every injection
 $i:  (V_2, G_2,\pi_2 ) \to (V_1, G_1, \pi_1 ) $ such that
\begin{enumerate}
\item $\tilde{f}_{U_1 U_1^{\prime}} \circ i = \lambda(i) \circ
\tilde{f}_{U_2 U_2^{\prime}} $ \item $ \lambda(j \circ i)=
\lambda(j) \circ \lambda(i)$ for each composition $j \circ i$ of
injections.
\end{enumerate}

\item The $C^l$ lift of $f$ defined by the collection
$\{\tilde{f}_{UU^{\prime}} \}$ is in the equivalence class
corresponding to $\bf f$.
\end{enumerate}

 Then we say that $\{\tilde{f}_{UU^{\prime}}, \lambda \}$ is a compatible system of ${\bf f}$.
\end{definition}

Note that if ${\bX}^{\prime}$ is reduced, each automorphism $g\in
G$ of $(V,G,\pi)$ is assigned an automorphism
 $\lambda(g) \in G^{\prime}$ giving rise, by condition (2)(b), to a group homomorphism
$\lambda_{UU^{\prime}} : G \to G^{\prime}$ with respect to which $\tilde{f}_{UU^{\prime}} $ is
equivariant.

\begin{definition}
 A $ C^l$ map ${\bf f}:{\bX} \to{\bX}^{\prime}$ is called a  good $C^l$ map
  if it admits a compatible system.
\end{definition}

When an orbifold is reduced,  it may be represented as the
quotient of a manifold by the effective action of a compact Lie
group by the so-called frame bundle trick. However a good map
between ${\bX} = M/G$ and ${\bX}^{\prime}= N/H $ may not be
represented by an equivariant map from $M$ to $N$.
 This has to do with the fineness of the compatible
cover of ${\bX}$ used to define the good map.
% However we can always find a complicated enough groupoid ${\bf G}$ representing
%the orbifold ${\bX}$ so that we can represent the good map by
%morphism of groupoids ${\bf G} \to (H \ltimes N) $.
 Indeed, a
similar problem occurs for a good map from a manifold to an
orbifold. For instance consider the $C^{\infty}$ map $S^{1} \to
\C/\Z_p$ given by the lift $t \mapsto t^{1/p}$. We need to use a
suitable cover of $S^{1}$ to make sense of continuity of the lift.

A good $C^{\infty}$ map is what Chen and Ruan\cite{CR} calls a
good map. Not all orbifold maps admit a compatible system.  See
example 4.4.2a of \cite{CR}.

Chen and Ruan prove (cf. Lemma 4.4.6 and Remark 4.4.7 of
\cite{CR})  that given two compatible systems $ \xi_1=\{
\tilde{f}_{1,U U^{\prime}}, \alpha_1 : U \in \mathcal{U},
U^{\prime} \in \mathcal{U^{\prime}} \}$, and  $\xi_2 = \{
\tilde{f}_{2,R R^{\prime}}, \alpha_2 : R \in \mathcal{R},
R^{\prime} \in \mathcal{R^{\prime}} \}$ for a $C^{\infty}$ map
${\bf f}: {\bX} \to {\bX}^{\prime} $, there exist
\begin{enumerate}
 \item common refinements
$\mathcal{W}$ of $\mathcal{U} $ and $\mathcal{R}$,  and  $\mathcal{W}^{\prime}$
 of $\mathcal{U}^{\prime} $ and $\mathcal{R}^{\prime}$, that satisfy condition(1) of definition
\ref{compsys},
 \item  compatible systems $\{ \tilde{f}_{1,W
W^{\prime}}, \lambda_1 \}$  and $\{ \tilde{f}_{2,W W^{\prime}},
\lambda_2 \}$, where $ W \in \mathcal{W}$ and $W^{\prime} \in
\mathcal{W}^{\prime}$, for ${\bf f}$ induced by $\xi_1$ and
$\xi_2$ respectively.
\end{enumerate}

Chen-Ruan's proof actually works for any $C^l$ map where $l \ge
0$. An important consequence of Lemma 4.4.6 of \cite{CR} is that
the compatible systems $\{ \tilde{f}_{i,W W^{\prime}}, \lambda_i
\}$ can be assumed to be geodesic compatible
 systems (see Definition 4.4.5 of \cite{CR}). In particular,
 the open sets of $\mathcal{W}$ and $\mathcal{W}^{\prime}$  are images of the exponential
 map from some subset of the tangent space at some point in their interiors. This property is
 crucial to relate compatible systems with pull-back of vector bundles, especially the tangent bundle.
This will continue to hold if ${\bX}$ is an orbifold Riemann
surface with smooth boundary, for appropriate choice of Riemannian
metric on ${\bX}$. (The idea is to choose a Riemannian metric on
the double ${\mathcal Y}$ of ${\bX}$  that  agrees with a
 metric on the manifold $Y$  away from a
small neighborhood of the singular set. Then use the positivity of
the injective radius of the metric on $Y$.)

The following definition is equivalent to, but different  from,
the one in \cite{CR}.

\begin{definition}
Two compatible systems $\xi_1$ and $\xi_2$ of a good $C^l$  map
${\bf f}$ are said to be isomorphic if
 there exist  induced compatible systems $\{\tilde{f}_{1,WW^{\prime}}, \lambda_1 \}$
and $\{\tilde{f}_{2,WW^{\prime}}, \lambda_2\}$ corresponding to
$\xi_1$ and $\xi_2$ respectively, and an automorphism
$\delta_{V^{\prime}}$ of the uniformizing system $(V^{\prime},
G^{\prime}, \pi^{\prime})$ for each $W^{\prime} \in
\mathcal{W}^{\prime}$, such that
\begin{enumerate}
\item $\delta_{V^{\prime}} \circ \tilde{f}_{1,WW^{\prime}} =
\tilde{f}_{2,WW^{\prime}}$ and \item for each injection $i:
(W_2,G_2,\pi_2)\to (W_1, G_1,\pi_1)$, the relation $ \lambda_2(i)
= \delta_{V_1^{\prime}} \circ \lambda_1(i) \circ
(\delta_{V_2^{\prime}})^{-1} $ holds.
\end{enumerate}
\end{definition}
The proof of the following lemma is similar to Proposition 4.4.8 of \cite{CR}.
\begin{lemma}\label{csvb} Suppose ${\bf f}: {\bX} \to {\bX}^{\prime}$ is a good $C^{\infty}$ map where
${\bX}$ is an orbifold with smooth boundary. Then two compatible systems $\xi_1$ and $\xi_2$
are isomorphic if and only if the pullbacks of any orbifold vector bundle on ${\bX}^{\prime}$
 by $\xi_1$ and $\xi_2$ are isomorphic.
\end{lemma}

\bibliographystyle{amsalpha}

\end{document}